\documentclass[11pt, reqno]{amsart}
\usepackage{amsmath, amsthm, amscd, amsfonts, amssymb, graphicx, color, mathtools}
\usepackage{mathrsfs}
\usepackage[bookmarksnumbered, colorlinks, plainpages]{hyperref}
\usepackage{enumerate}
\usepackage{setspace}
\usepackage{multicol}
\usepackage[margin=2cm]{geometry}
\usepackage{comment}

\usepackage[normalem]{ulem}
\newcommand\tsout{\bgroup\markoverwith{\textcolor{red}{\rule[0.5ex]{2pt}{1.4pt}}}\ULon}
\newcommand{\stkout}[1]{\ifmmode\text{\tsout{\ensuremath{#1}}}\else\tsout{#1}\fi}
\allowdisplaybreaks

\theoremstyle{definition}
\newtheorem{theorem}{Theorem}[section]
\newtheorem{lemma}[theorem]{Lemma}
\newtheorem{proposition}[theorem]{Proposition}

\newtheorem{definition}[theorem]{Definition}

\newtheorem{remark}[theorem]{Remark}
\numberwithin{equation}{section}

\newcommand{\dif}{\mathrm{d}}
\newcommand{\bff}{\boldsymbol}
\newcommand{\bb}{\mathbb}

\newcommand{\dt}{\mathrm{d}t}

\newcommand{\dr}{\mathrm{d}r}

\newcommand{\ds}{\mathrm{d}s}

\newcommand{\dW}{\mathrm{d}W}

\newcommand{\norm}[2]{\left\|{#1}\right\|_{#2}}
\newcommand{\inpro}[2]{\left\langle#1,#2\right\rangle}

\newcommand{\abs}[1]{\left|{#1}\right|}

\def\be{\begin{equation}\label}
\def\ee{\end{equation}}
\def\bd{\begin{definition}\label}
\def\ed{\end{definition}}
\def\bt{\begin{theorem}\label}
\def\et{\end{theorem}}
\def\br{\begin{remark}\label}
\def\er{\end{remark}}
\def\bal{\[\begin{aligned}}
\def\eal{\end{aligned}\]}

\begin{document}
\setcounter{page}{1}

\title[The stochastic Landau--Lifshitz--Baryakhtar equation]{The stochastic Landau--Lifshitz--Baryakhtar equation: Global solution and invariant measure}

\author[Beniamin Goldys]{Beniamin Goldys}
\address{School of Mathematics and Statistics, The University of Sydney, Sydney 2006, Australia}
\email{\textcolor[rgb]{0.00,0.00,0.84}{beniamin.goldys@sydney.edu.au}}

\author[Agus L. Soenjaya]{Agus L. Soenjaya}
\address{School of Mathematics and Statistics, The University of New South Wales, Sydney 2052, Australia}
\email{\textcolor[rgb]{0.00,0.00,0.84}{a.soenjaya@unsw.edu.au}}

\author[Thanh Tran]{Thanh Tran}
\address{School of Mathematics and Statistics, The University of New South Wales, Sydney 2052, Australia}
\email{\textcolor[rgb]{0.00,0.00,0.84}{thanh.tran@unsw.edu.au}}

\date{\today}


\begin{abstract}
The Landau--Lifshitz--Baryakhtar (LLBar) equation perturbed by both additive and multiplicative noises is a system of fourth-order stochastic PDEs which models the evolution of magnetic spin fields in ferromagnetic materials at elevated temperatures, taking into account longitudinal damping, long-range interactions, spin current, and noise-induced phenomena at high temperatures. In this paper, we show the existence of a unique pathwise solution (which is analytically strong) to the stochastic LLBar equation posed in a bounded domain $\mathscr{D}\subset \mathbb{R}^d$, where $d=1,2,3$. We also prove the convergence of this pathwise solution to that of the stochastic Landau--Lifshitz--Bloch (LLB) equation in the limit of vanishing exchange relaxation parameter. Finally, we show the Feller property of the Markov semigroup associated with the strong solution, and prove the existence of nontrivial invariant measures. For temperatures above the Curie temperature, exponential stability of the solution and uniqueness of invariant measure are obtained under certain dissipativity conditions.
\end{abstract}
\maketitle

\tableofcontents

\section{Introduction}
In recent years there has been growing interest in systems of stochastic partial differential equations (SPDEs), especially systems with high order partial derivatives; see for example an important paper \cite{Portal-Veraar}. However, these high order systems are still little understood. On the other hand, well-posedness for some fourth-order scalar-valued SPDEs, such as the stochastic Kuramoto--Sivashinsky equation~\cite{DuaErv01, Fer08} and the stochastic Swift--Hohenberg equation~\cite{Gao17, WanSunDua05} have been analysed in the literature, but mainly for the case of one- or two-dimensional spatial domains. 

In this paper, we are concerned with a particular fourth-order system of nonlinear SPDEs arising in the theory of ferromagnetism. This is a semi-classical continuum theory that aims to describe the magnetic properties of a ferromagnetic object. The problem takes the following form:
Given a magnetic body $\mathscr{D} \subset \bb{R}^d$, $d\in \{1,2,3\}$, with sufficiently smooth boundary $\partial \mathscr{D}$, its magnetisation vector $\bff{u}(t,\bff{x}) \in \bb{R}^3$ for any time $t>0$ satisfies the following equations 
\begin{subequations}\label{equ:sllbar}
	\begin{alignat}{2}
		&\dif \bff{u}
		= 
		\big( \lambda_r \bff{H}
		- \lambda_e \Delta \bff{H} 
		- \gamma \bff{u} \times \bff{H} 
		+ S(\bff{u}) \big)\, \dt
		+
		\sum_{k=1}^\infty G_k(\bff{u}) \, \dif W_k(t)
		\; && \quad\text{for $(t,\bff{x})\in(0,T)\times \mathscr{D}$,}
  \label{equ:sllbar a}
		\\
		&\bff{H}
		= 
		\alpha\Delta \bff{u}
		+ \kappa_1 \bff{u}
		- \kappa_2 |\bff{u}|^2 \bff{u}
		\; && \quad\text{for $(t,\bff{x})\in(0,T)\times\mathscr{D}$,}
  \label{equ:sllbar b}
		\\[1ex]
		&\bff{u}(0,\bff{x})= \bff{u}_0(\bff{x}) 
		\; && \quad\text{for } \bff{x}\in \mathscr{D},
  \label{equ:sllbar c}
		\\[1ex]
			&\frac{\partial \bff{u}}{\partial \bff{n}}= \bff{0}, 
		\;\displaystyle{\frac{\partial \bff{H}}{\partial \bff{n}}= \bff{0}} 
		\; && \quad\text{for } (t,\bff{x})\in (0,T) \times \partial \mathscr{D},
  \label{equ:sllbar d}
	\end{alignat}
\end{subequations}
where $\{W_k\}_{k\in \bb{N}}$ is a family of independent real-valued Wiener processes and $S(\bff{u})$ is a nonlinear function of $\bff{u}$ to be specified later.
For physical reasons and concreteness we take 
\begin{equation}\label{equ:Gk}
G_k(\bff{u}):= \bff{g}_k + \gamma \bff{u} \times \bff{h}_k,
\end{equation}
where $\bff{g}_k$ and $\bff{h}_k$ are given functions with certain regularity. For simplicity of presentation, we set~$\alpha=\kappa_1=\kappa_2=1$ (corresponding to the regime below the Curie temperature), except in Theorem~\ref{the:stab above} and Theorem~\ref{the:llbar llb} where we set $\alpha=\kappa_2=1$ and $\kappa_1=-1$ (corresponding to the regime above the Curie point). This system is known as an initial-boundary value problem associated with the stochastic Landau--Lifshitz--Baryakhtar (LLBar) equation, the physical background of which (together with the physical meanings of all coefficients in the equations) and the literature review are deferred to a subsequent paragraph.

To the best of our knowledge, the existing general theory of SPDEs does not cover problem~\eqref{equ:sllbar}. The results in \cite{Portal-Veraar} would require the equations under consideration posed in the full space $\mathbb R^d$ and containing Lipschitz nonlinearities only. While an extension might be possible, we have chosen instead to use the classical arguments based on Galerkin approximations. We begin by proving the existence of an analytically strong global martingale solution. It is worth noting that the right-hand side of~\eqref{equ:sllbar a} does not define a monotone operator over the full range of parameters, making uniform-in-time estimates nontrivial. Next, we establish pathwise uniqueness, which in turn yields the existence of a unique pathwise global strong solution to the problem. These solutions generate a Markov family on $\mathbb H^2$, and we show that the corresponding transition semigroup $\{P_t\}_{t\geq 0}$ is Feller, in the sense that $P_tC_b(\mathbb H^2)\subset C_b(\mathbb H^2)$. We remark that, for related nonlinear problems such as the Landau--Lifshitz--Bloch equation, the Feller property has been established only when the state space is endowed with the weak topology~\cite{BrzGolLe20}. In \cite{Gus22}, the Feller property is proved for the strong topology but the domain is one-dimensional. The proof of the Feller property in $\mathbb H^2$ endowed with the norm topology requires some delicate arguments. We introduce two sequences of stopping times: the first one bounds the $\mathbb{H}^2$ norm of the solution, while the second one captures the instantaneous parabolic regularisation inherent in the equation. Together, these stopping times are essential in establishing a continuous dependence estimate, which in turn yields the Feller property of the transition semigroup on $\mathbb{H}^2$. The existence of an invariant measure then follows from the classical Krylov–Bogoliubov theorem.

Problem \eqref{equ:sllbar} is not dissipative in general. In Theorem \ref{the:stab above}, we consider the dissipative case under the assumption of the absence of additive noise in \eqref{equ:sllbar}. In this case, the Dirac measure $\delta_{\bff{0}}$ is trivially invariant. Theorem~\ref{the:stab above} further shows that $\delta_{\bff{0}}$ is the unique invariant measure for the system, and that solutions converge exponentially fast to this stationary state, provided the dissipativity is sufficiently strong. This dissipativity assumption corresponds to the situation in which the temperature of the ferromagnetic material lies above the Curie point; see the discussion below.

Finally, in Theorem~\ref{the:llbar llb}, we show that as~$\lambda_e \to 0^+$, the pathwise solution converges to that of the stochastic Landau–Lifshitz–Bloch (LLB) equation, given by~\eqref{equ:sllbar a} with~$\lambda_e = 0$. In the deterministic setting ($\bff{g}_k=\bff{h}_k=\bff{0}$), this fact allows us to approximate the solution of the LLB equation by that of the LLBar equation, the latter can be solved numerically by a mixed finite element scheme efficiently~\cite{Soe24}. In the stochastic setting, this result forms a basis for the mixed finite element method proposed in~\cite{GolSoeTra25b}, which further reduce the computational cost compared to the $C^1$-conforming method proposed in~\cite{GolJiaLe25}.

We now provide some physical background of the problem under consideration. The standard model in the theory of ferromagnetism is the Landau--Lifshitz--Gilbert (LLG) equation~\cite{LL35}. It is widely known, however, that this model is only valid at temperatures far below the Curie temperature of the material~\cite{BarIva15, ChuNie20}. On the other hand, many modern magnetic devices, including the heat-assisted magnetic recording (HAMR)~\cite{ZhuLi13} and the thermally-assisted magnetic random access memory~\cite{PreKer07}, operate at temperatures very close to (or even exceeding) its Curie temperature. As such, it is essential to have an accurate physical model for magnetisation dynamics at elevated temperatures. Several micromagnetic models which are valid at high temperatures have been proposed in the literature, two of the earliest being the LLB equation~\cite{Gar97} and the LLBar equation~\cite{BarBar98, Bar84}.

Some comments on a version of the LLBar equation we study in this paper are in place now. In equations~\eqref{equ:sllbar a} and~\eqref{equ:sllbar b}, the positive constants $\lambda_r, \lambda_e$, $\alpha$, and $\gamma$ are, respectively, the relativistic damping constant, exchange relaxation constant, exchange field constant, and electron gyromagnetic ratio. Physical considerations dictate that the constants~$\kappa_1$ and~$\kappa_2$ are positive for temperatures below the Curie temperature, while above the Curie temperature $\kappa_1$ is negative. 
In~\eqref{equ:sllbar a},
\begin{equation}\label{equ:S u}
	S(\bff{u}):= R(\bff{u})+ L(\bff{u}),
\end{equation}
where $L(\bff{u})$ is a Lipschitz function of $\bff{u}$ and $R(\bff{u})$ is the spin-torque terms (see~\cite{AyoKotMouZak21, MelPta13, SchHinKla09, Alb16, YasFasIvaMak22, ZhaLi04}) defined by
\begin{equation}\label{equ:torque}
	R(\bff{u})
    := 
    \beta_1 (\bff{\nu} \cdot \nabla) \bff{u} 
    + 
    \beta_2 \bff{u} \times (\bff{\nu} \cdot \nabla) \bff{u},
\end{equation}
with $\bff{\nu}(t): \mathscr{D} \to \bb{R}^d$ being the spin current, and $\beta_1$ and $\beta_2$ being constants.
More precise assumptions on $L(\bff{u})$ are elaborated in Section~\ref{subsec:assump}.

To incorporate random perturbation into the dynamics, many approaches are available in the physics literature. A suggestion is made in~\cite{EvaHin12} to perturb the precessional term and add a random torque, while another argument is put forth in~\cite{XuZha13}, by virtue of the fluctuation-dissipation theorem, to perturb only the precessional term. For the stochastic LLBar problem~\eqref{equ:sllbar} considered in this paper, we generally add noise to the precessional term and append a random forcing term (cf.~\eqref{equ:Gk}), except in Theorem~\ref{the:stab above} and in Section~\ref{sec:stab above} where noise enters only through the precessional term.

The deterministic version of~\eqref{equ:sllbar} (corresponding to~$G_k(\bff{u})\equiv0$ in~\eqref{equ:sllbar a}) has been studied in~\cite{SoeTra23, SoeTra23b}. Other fourth-order equations of Landau--Lifshitz-type are studied in~\cite{melcher}. 
It is noted that systems of fourth-order equations also appear in certain areas of chemistry and biology.
In particular, if in~\eqref{equ:sllbar a} $\gamma=0$ and $\beta_2=0$, then we have a generalised reaction-diffusion-advection model with long range effects, which is prominently used in mathematical biology~\cite[Chapter~11]{Mur02} (see also~\cite{CocDi23, CohMur81, Och84}), and in chemistry to model anomalous bi-flux diffusion process~\cite{BevGalSimRio13, JiaBevZhu20}.

The stochastic versions of the LLG and LLB equations have been studied before, for instance in~\cite{AloBouHoc14, BrzGolLe20, BrzGolLi24, GolLeTra16}. In particular, the well-posedness and long-time behaviour of the stochastic LLG equation, including large deviations results, have been studied in~\cite{BrzGolJeg17, GusHoc23}. The existence of invariant measures for $d=1$ has been established in~\cite{Gus22, NekPro13}, while numerical simulations aimed at capturing the behaviour of invariant measures in two- or three-dimensional domains have been carried out in~\cite{BanBrzNekPro14}. For the stochastic LLB equation, the existence of pathwise solutions and invariant measures for $d\leq 2$ has been shown in~\cite{BrzGolLe20}. 
After the completion of this work, we become aware of the preprint~\cite{XuZhaLiu24}, which proves the unique existence of \emph{analytically weak} pathwise global solutions for the stochastic LLBar equation for $d=1$. In contrast, our results establish the existence and uniqueness of \emph{analytically strong} pathwise global solutions for dimensions $d \leq 3$. A key element of our analysis is a uniform $\mathbb{H}^1$ estimate, obtained by exploiting the energy structure of the equation, which could also be employed to extend the results of~\cite{XuZhaLiu24} to higher dimensions. Furthermore, we show the existence of invariant measures supported on a more regular space and valid for $d \leq 3$, whereas~\cite{XuZhaLiu24} obtains such invariant measures only for $d=1$.

The main results of this paper are as follows:
\begin{enumerate}[(i)]
    \item existence and uniqueness of (probabilistically and analytically) strong solution (Theorem~\ref{the:unique}),
    \item existence of a non-trivial invariant measure (Theorem~\ref{the:invariant}),
    \item exponential stability of the solution and uniqueness of invariant measure for temperatures above the Curie point, in the case where the noise enters only through the precessional term (Theorem~\ref{the:stab above}),
    \item convergence of the strong solution of the stochastic LLBar equation to that of the stochastic LLB equation in the limit of vanishing exchange relaxation parameter (Theorem~\ref{the:llbar llb}).
\end{enumerate}

\section{Preliminaries}
\subsection{Notations}
We begin by defining some notations used in this paper. Let $k\in \bb{N}$ and $p\in [1,\infty]$. The function space $\bb{L}^p := \bb{L}^p(\mathscr{D}; \bb{R}^3)$ denotes the usual space of $p$-th integrable functions defined on~$\mathscr{D}$ and taking values in $\bb{R}^3$ with an obvious modification for $p=\infty$, and $\bb{W}^{k,p} := \bb{W}^{k,p}(\mathscr{D}; \bb{R}^3)$ denotes the usual Sobolev space of 
functions on $\mathscr{D} \subset \bb{R}^d$ taking values in $\bb{R}^3$. We
write $\bb{H}^k := \bb{W}^{k,2}$. The dual space of $\bb{H}^k$ is denoted by $\widetilde{\bb{H}}^{-k}$.
The partial derivative
$\partial/\partial x_i$ will be written by $\partial_i$ for short. The partial derivative of~$f$ with respect to time $t$ will be denoted by $\partial_t$.

If $X$ is a Banach space, the spaces $L^p(0,T; X)$ and $W^{k,p}(0,T;X)$ denote respectively the usual Lebesgue and Sobolev spaces of strongly measurable functions on $(0,T)$ taking values in $X$. The space $C([0,T];X)$ denotes the space of continuous functions on $[0,T]$ taking values in $X$, while $B_b(X)$ and $C_b(X)$ denotes the space of bounded real-valued Borel functions on $X$ and the space of bounded continuous functions on $X$, respectively. The space $L^p(\Omega; X)$ denotes the space of $X$-valued random variables with finite $p$-th moment, where $(\Omega,\mathcal{F},\bb{P})$ is a probability space. The expectation of a random variable $Y$ will be denoted by $\bb{E}[Y]$.

Throughout this paper, we denote the scalar product in a Hilbert space $H$ by $\langle \cdot, \cdot\rangle_H$ and its corresponding norm by $\|\cdot\|_H$. We will not distinguish between the scalar product of $\bb{L}^2$ vector-valued functions taking values in $\bb{R}^3$ and the scalar product of $\bb{L}^2$ matrix-valued functions taking values in $\bb{R}^{3\times 3}$, and still denote them by $\langle\cdot,\cdot\rangle_{\bb{L}^2}$.

In various estimates, the constant $C$ in the estimate denotes a
generic constant which takes different values at different occurrences. If
the dependence of $C$ on some variable, e.g.~$S$, is highlighted, we will write $C_S$ or $C(S)$. The notation $A\lesssim_S B$ means $A \le C_S B$, while the notation $A \lesssim B$ means $A \le C B$ where the specific form of
the constant $C$ is not important to clarify.

\subsection{Assumptions}\label{subsec:assump}
Assume that $\mathscr{D}$ is an open bounded domain with $C^4$-smooth boundary. Let $\Delta$ denote the Neumann Laplacian acting on $\bb{R}^3$-valued functions with the domain
\begin{equation}\label{equ:neumann delta}
{\text{D}(\Delta):= \left\{ \bff{v}\in \bb{H}^2 : \frac{\partial\bff{v}}{\partial \bff{n}}=0 \text{ on } \partial\mathscr{D} \right\}.}
\end{equation}
For equation~\eqref{equ:sllbar}, we assume the following:
\begin{enumerate}
	\item For each $k\in \bb{N}$,
	\[
		G_k(\bff{u}) = \bff{g}_k + \gamma \bff{u} \times \bff{h}_k,
	\]
	where $\bff{g}_k:\mathscr{D}\to \bb{R}^3$ and $\bff{h}_k:\mathscr{D}\to\bb{R}^3$ are functions such that {$\bff{g}_k, \bff{h}_k \in \mathrm{D}(\Delta)$} and
	\begin{equation}\label{equ:sigma g sigma h}
		\sigma_g^2:= \sum_{k=1}^\infty \norm{\bff{g}_k}{\bb{H}^2}^2 <\infty
        \quad \text{and} \quad
        \sigma_h^2:= \sum_{k=1}^\infty \norm{\bff{h}_k}{\bb{H}^2}^2 <\infty.
	\end{equation}
	\item The spin current vector field $\bff{\nu}\in L^\infty\big(\bb{R}^+;\bb{L}^\infty(\mathscr{D};\bb{R}^d)\big)$ is given.
	\item $L:\bb{R}^3 \to \bb{R}^3$ is a Lipschitz-continuous function representing any additional phenomenological torque terms arising from an applied field, an optical field, or others~\cite{MonBerOpp16, MonBerOpp18}. More precisely, there exists a constant $C>0$ such that for all $\bff{v}_1$, $\bff{v}_2\in \bb{R}^3$,
	\begin{equation*}
		|L(\bff{v}_1) - L(\bff{v}_2)| \leq C|\bff{v}_1-\bff{v}_2|.
	\end{equation*}
\end{enumerate}

For any $\beta>0$, we define the Hilbert space 
\begin{equation}\label{equ:X beta}
    \bb{X}^\beta= \text{D}(\Delta^\beta), \; \text{ equipped with the norm }\; \norm{\bff{v}}{\bb{X}^\beta}:= \norm{(-\Delta+I)^\beta \bff{v}}{\bb{L}^2}.
\end{equation}
The dual space of $\bb{X}^\beta$ will be denoted by $\bb{X}^{-\beta}$.
We also define the operator
\begin{equation}\label{equ:op A}
    A:= \lambda_e \Delta^2 -(\lambda_r-\lambda_e)\Delta + \beta_0 I, \quad \text{with} \;\;
    \text{D}(A) = \bb{X}^2,
\end{equation}
where $\beta_0>0$ is sufficiently large so that $A\geq \lambda_0 I$ for some $\lambda_0>0$. Then $A$ is a positive, self-adjoint operator in $\bb{L}^2$.

\subsection{Formulations and main results}

%
The notion of solution to \eqref{equ:sllbar} used in this paper can now be stated in the following definition.

\begin{definition}\label{def:mart sol}
    Given $T>0$ and initial data $\bff{u}_0\in\bb{H}^2$, a martingale solution $(\Omega, \mathcal{F}, \bb{F}, \bb{P}, W, \bff{u})$ to problem~\eqref{equ:sllbar} in $[0,T]$ consists of 
	\begin{enumerate}
	\item a filtered probability space $(\Omega, \mathcal{F}, \bb{F}, \bb{P})$ with the filtration $\bb{F}=\{\mathcal{F}_t\}_{t\geq 0}$ satisfying the usual conditions,
	\item a sequence of real-valued $\bb{F}$-adapted Wiener processes $W_{k}=\{W_{k}(t)\}_{t\geq 0}$,
	\item a progressively measurable process $\bff{u}: [0,T]\times \Omega \to \bb{H}^2$ such that $\bb{P}$-a.s.
	\[
	\bff{u}\in  L^\infty(0,T;\bb{H}^2) \cap L^2(0,T;\bb{H}^4),
	\]
	and for every $t\in [0,T]$, $\bb{P}$-a.s.
	\begin{align}\label{equ:weakform}
			\bff{u}(t)
			&=
			\bff{u}_0
			+
			\lambda_r 
			\int_{0}^{t}  \bff{H}(s)\,\ds
			-
			\lambda_e 
			\int_{0}^{t} \Delta \bff{H}(s) \,\ds
			-
			\gamma 
			\int_{0}^{t} \bff{u}(s) \times
			\bff{H}(s) \,\ds
			\nonumber \\
			&\quad
			+
			\int_0^t R\big(\bff{u}(s)\big) \,\ds
			+
			\int_0^t L\big(\bff{u}(s)\big) \, \ds
			+
			\sum_{k=1}^\infty \int_0^t \left(\bff{g}_k+  \gamma \bff{u}(s)\times \bff{h}_k \right) \dif W_k(s),
   \end{align}
   where $\bff{H}(t) =\Delta \bff{u}(t) \pm\bff{u}(t)-|\bff{u}(t)|^2 \bff{u}(t)$ a.e.
	\end{enumerate}
\end{definition}

The main results of this paper are stated below.

\begin{theorem}\label{the:exist}
Let $\mathscr{D}\subset \bb{R}^d$, where $d=1$, $2$, $3$. Let $\bff{u}_0\in \mathrm{D}(\Delta)$ and {$\bff{g}_k,\bff{h}_k\in \mathrm{D}(\Delta)$ be given satisfying \eqref{equ:sigma g sigma h}}. There exists a martingale solution $(\Omega, \mathcal{F}, \bb{F}, \bb{P}, W, \bff{u})$ of \eqref{equ:sllbar} such that
\begin{enumerate}
	\item for any $q\geq 1$,
		\begin{align*}
			\bff{u} &\in L^q \Big(\Omega; L^\infty(0,T;\bb{H}^2) 
            \cap L^2(0,T;\bb{H}^4)\Big),
		\end{align*}
    \item for every $\beta\in [\frac12,1)$, $\delta\in (0,1-\beta)$, $\bb{P}$-a.s.,
    \begin{equation}\label{equ:u holder reg}
        \bff{u}\in C^\delta\big((0,T);\text{D}(A^\beta)\big),
    \end{equation}
    where $A$ is defined in~\eqref{equ:op A}. In particular, $\bff{u}\in C([0,T]; \bb{H}^2)$.
\end{enumerate}
\end{theorem}

\begin{theorem}\label{the:unique}
Suppose that $(\Omega, \mathcal{F}, \bb{F}, \bb{P}, W, \bff{u}_1)$ and $(\Omega, \mathcal{F}, \bb{F}, \bb{P}, W, \bff{u}_2)$ are two martingale solutions to \eqref{equ:sllbar} in the sense of Definition~\ref{def:mart sol}.
Then, for $\bb{P}$-a.s. $\omega\in \Omega$,
\[
	\bff{u}_1(\cdot, \omega) = \bff{u}_2(\cdot, \omega).
\]
This implies the existence of a pathwise unique (probabilistically) strong solution of \eqref{equ:sllbar} and the uniqueness in law of martingale solution of \eqref{equ:sllbar} by the Yamada--Watanabe theorem ~\cite[Theorem~2.2 and Theorem~12.1]{Ond04}.
\end{theorem}

\begin{theorem}\label{the:invariant}
There exists an ergodic invariant measure for \eqref{equ:sllbar} supported on $\bb{H}^3$.
\end{theorem}

\begin{remark}
The above theorems remain valid under a more general assumption on the noise term, requiring only minor modifications to the proofs. To formulate this, we introduce the following notion. Let $Y$ and $Z$ be Banach spaces. A progressively measurable map $h:\Omega\times [0,T] \times Y\to Z$ is said to be uniformly Lipschitz with constant $L_Y$ if, for all $\bff{v},\bff{w}\in Y$ and $(\omega,t)\in \Omega\times[0,T]$,
\begin{align*}
	\norm{h(\omega,t,\bff{v})-h(\omega,t,\bff{w})}{Z} \leq L_Y \norm{\bff{v}-\bff{w}}{Y}
\quad\text{and}\quad
	\norm{h(\omega,t,\bff{v})}{Z} \leq L_Y \left(1+\norm{\bff{v}}{Y} \right).
\end{align*}
The collection of all such mappings is denoted by $\text{Lip}(Y, Z)$. With this notation, the above theorems extend to the case where $G:= \{G_k\}_{k\in \bb{N}}:\Omega\times [0,T] \times {\mathrm{D}(\Delta)} \to \ell^2({\mathrm{D}(\Delta)})$ in \eqref{equ:sllbar a} satisfies the condition $G\in \text{Lip}\big({\mathrm{D}(\Delta^s)}, \ell^2({\mathrm{D}(\Delta^s)})\big)$ for $s\in \big\{0,\frac12,1 \big\}$. We fix a specific form of $G$ since this is the form widely adopted for this model and for other stochastic ferromagnetic models.
\end{remark}

In the following theorem, we assume that the stochastic perturbation acts only on the precessional term, i.e. only the multiplicative noise is present. This form of the noise was also considered for Landau--Lifshitz-type equations in~\cite{BrzGolJeg13, EvaHin12, XuZha13}. Above the Curie temperature, we have the following exponential stability result when the relativistic damping constant $\lambda_r$ is sufficiently large (relative to $\lambda_e$ and $\gamma^2 \sigma_h^2$). This corresponds to a physically observable loss of magnetisation above the Curie temperature~\cite{Hah19}.

\begin{theorem}\label{the:stab above}
Suppose that $\lambda_r>\frac12 \gamma^2 \sigma_h^2$ and let $\mu:= \lambda_r- \frac12 \gamma^2 \sigma_h^2$. Let $\bff{u}$ be the solution of \eqref{equ:sllbar} given by Theorem~\ref{the:exist}, corresponding to the case $\kappa_1=-1$, $\kappa_2=1$, $\bff{g}_k\equiv \bff{0}$, and $S(\bff{u})\equiv \bff{0}$. Then $\bb{P}$-a.s. for all $t> 0$,
\begin{equation}\label{equ:u L2 decay}
	\norm{\bff{u}(t)}{\bb{L}^2}^2 \leq e^{-\mu t} \norm{\bff{u}_0}{\bb{L}^2}^2,
\end{equation}
Furthermore, for $d=1,2$, there exists a unique invariant measure which is the Dirac measure at $\bff{0}$.

For $d=3$, the uniqueness of invariant measure holds with an additional assumption $\lambda_r\geq 3\gamma^2\sigma_h^2$.
\end{theorem}

\subsection{Auxiliary facts}

In our analysis, we will frequently use the following results. For any vector-valued function $\bff{v}:\mathscr{D} \to\bb{R}^3$, we have
	\begin{align}
 	\label{equ:nor der v2v}
		\nabla (|\bff{v}|^2 \bff{v}) 
		&= 
		2 \bff{v} \ (\bff{v}\cdot \nabla\bff{v}) 
		+ |\bff{v}|^2 \nabla \bff{v},
        \\
        \label{equ:ddn v2v}
        \frac{\partial(\abs{\bff{v}}^2 \bff{v})}{\partial \bff{n}}
        &=
        2\bff{v} \left(\bff{v}\cdot \frac{\partial \bff{v}}{\partial \bff{n}}\right)
        +
        \abs{\bff{v}}^2 \frac{\partial \bff{v}}{\partial \bff{n}},
        \\
        \label{equ:del v2v}
		\Delta (|\bff{v}|^2 \bff{v}) 
		&= 
		2|\nabla \bff{v}|^2 \bff{v} 
		+ 2(\bff{v}\cdot \Delta \bff{v})\bff{v} 
		+ 4 \nabla \bff{v} \ (\bff{v}\cdot \nabla\bff{v})^\top
		+ |\bff{v}|^2 \Delta \bff{v},
	\end{align}
	provided that the partial derivatives are well defined. As a consequence of \eqref{equ:sllbar b}, \eqref{equ:sllbar d}, and \eqref{equ:ddn v2v}, for a sufficiently regular solution $\bff{u}$ of problem~\eqref{equ:sllbar}, we have $\partial(\Delta \bff{u})/\partial \bff{n}=\bff{0}$ on $\partial\mathscr{D}$.

\begin{lemma}
	Let $\mathscr{D} \subset \bb{R}^d$ be an open bounded domain with smooth boundary and $\epsilon>0$ be
	given. Then there exists a positive constant $C$ such that the following
	inequalities hold:
	\begin{enumerate}
		\renewcommand{\labelenumi}{\theenumi}
		\renewcommand{\theenumi}{{\rm (\roman{enumi})}}
		\item for any $\bff{v} \in \text{D}(\Delta)$, 
		\begin{align}
			\label{equ:nabla v L2}
			\norm{\nabla \bff{v}}{\bb{L}^2}^2 
			&\leq \frac{1}{4\epsilon} \norm{\bff{v}}{\bb{L}^2}^2 + \epsilon \norm{\Delta \bff{v}}{\bb{L}^2}^2,
		\end{align}
		
		\item for any $\bff{v},\bff{w} \in \bb{H}^s$, where $s>d/2$,
		\begin{align}
		\label{equ:prod Hs mat dot}
			\norm{\bff{v} \odot \bff{w}}{\bb{H}^s}
			&\leq
			C \norm{\bff{v}}{\bb{H}^s}
			\norm{\bff{w}}{\bb{H}^s},
			\\
			\label{equ:prod Hs triple}
			\norm{(\bff{u} \times \bff{v}) \odot \bff{w}}{\bb{H}^s}
			&\leq
			C \norm{\bff{u}}{\bb{H}^s} \norm{\bff{v}}{\bb{H}^s} \norm{\bff{w}}{\bb{H}^s}.
		\end{align}
		Here $\odot$ denotes either the dot product or cross product.
	\end{enumerate}
\end{lemma}

\begin{proof}
	This is shown in~\cite[Lemma~2.2]{SoeTra23b}.
\end{proof}

\begin{lemma}\label{lem:tec lem}
	For any~$\bff{v}\in\bb{H}^1$, $\bff{w}\in\bb{W}^{1,2d}$, and~$\delta>0$, the
	following inequality holds:
	\[
	\Big|
	\inpro{\bff{v}\times\nabla\bff{w}}{\nabla\bff{v}}_{\bb{L}^2}
	\Big|
	\le
	\Phi(\bff{w})
	\norm{\bff{v}}{\bb{L}^2}^2
	+
	\delta \norm{\nabla\bff{v}}{\bb{L}^2}^2
	\]
	where
	\begin{equation}\label{equ:Phi}
		\Phi(\bff{w})
		=
		\begin{cases}
			\frac{c^2}{\delta}
			\norm{\nabla\bff{w}}{\bb{L}^2}^2
			+
			\frac{c^2}{4\delta^3}
			\norm{\nabla\bff{w}}{\bb{L}^2}^4,
			\quad & d=1,
			\\[1ex]
			\frac{3^3 c^4}{2^8\delta^3}
			\norm{\nabla\bff{w}}{\bb{L}^4}^4 
			+ \delta,
			\quad & d=2,
			\\[1ex]
			\frac{3^3 c^4}{2^8\delta^3}
			\norm{\nabla\bff{w}}{\bb{L}^6}^4 
			+ \delta,
			\quad & d=3.
		\end{cases}
	\end{equation}
	The positive constant~$c$ (given by the Gagliardo--Nirenberg inequality) depends
	only on the domain $\mathscr{D}$.
\end{lemma}

\begin{proof}
    This is shown in~\cite[Lemma~A.3]{LeSoeTra24}.
\end{proof}

\begin{lemma}
Let $R$ be the map defined in \eqref{equ:torque} and $\bff{\nu}$ be given. Then for each $\epsilon>0$, there exists a positive constant $C_\epsilon$ such that for any $\bff{v}, \bff{w} \in \bb{H}^1 \cap \bb{L}^\infty$,
\begin{align}
	\label{equ:Rv v}
	\big| \inpro{R(\bff{v})}{\bff{v}}_{\bb{L}^2} \big|
	&\leq 
	C_\epsilon\norm{\bff{\nu}}{\bb{L}^2(\mathscr{D};\bb{R}^d)}^2 
	+
	\epsilon \norm{|\bff{v}|\, |\nabla \bff{v}|}{\bb{L}^2}^2,
	\\
	\label{equ:Rv w}
	\big| \inpro{R(\bff{v})}{\bff{w}}_{\bb{L}^2} \big|
	&\leq
	C_\epsilon \norm{\bff{\nu}}{\bb{L}^\infty(\mathscr{D};\bb{R}^d)}^2 
	\left( \norm{\nabla \bff{v}}{\bb{L}^2}^2
	+
	\norm{|\bff{v}|\, |\nabla \bff{v}|}{\bb{L}^2}^2 \right)
	+
	\epsilon \norm{\bff{w}}{\bb{L}^2}^2,
    \\
    \label{equ:Rvw vw}
    \big| \inpro{R(\bff{v})-R(\bff{w})}{\bff{v}-\bff{w}}_{\bb{L}^2} \big|
	&\leq
    C_\epsilon \norm{\bff{\nu}}{\bb{L}^\infty(\mathscr{D};\bb{R}^d)}^2 \left(1+\norm{\bff{w}}{\bb{L}^\infty}^2 \right)
    \norm{\bff{v}-\bff{w}}{\bb{L}^2}^2
    +
    \epsilon \norm{\nabla \bff{v}-\nabla \bff{w}}{\bb{L}^2}^2.
\end{align}
\end{lemma}

\begin{proof}
Inequality \eqref{equ:Rv v} follows directly from Young's inequality. To show \eqref{equ:Rv w}, we apply H\"older's and Young's inequalities to obtain
\begin{align*}
	\big| \inpro{R(\bff{v})}{\bff{w}}_{\bb{L}^2} \big|
    &=
    \big| \inpro{\beta_1(\bff{\nu}\cdot\nabla)\bff{v}+\beta_2 \bff{v}\times (\bff{\nu}\cdot\nabla)\bff{v}}{\bff{w}}_{\bb{L}^2} \big|
    \\
	&\leq
	C \Big(\norm{\bff{\nu}}{\bb{L}^\infty(\mathscr{D};\bb{R}^d)} \norm{\nabla \bff{v}}{\bb{L}^2}
	+
	\norm{\bff{\nu}}{\bb{L}^\infty(\mathscr{D};\bb{R}^d)} \norm{|\bff{v}|\, |\nabla \bff{v}|}{\bb{L}^2} \Big) \norm{\bff{w}}{\bb{L}^2}
	\\
	&\leq
	C_\epsilon \norm{\bff{\nu}}{\bb{L}^\infty(\mathscr{D};\bb{R}^d)}^2 
	\left( \norm{\nabla \bff{v}}{\bb{L}^2}^2
	+
	\norm{|\bff{v}|\, |\nabla \bff{v}|}{\bb{L}^2}^2 \right)
	+
	\epsilon \norm{\bff{w}}{\bb{L}^2}^2
\end{align*}
as required. Finally, writing
\begin{align*}
    R(\bff{v})-R(\bff{w})
    =
    -(\bff{\nu}\cdot \nabla) (\bff{v}-\bff{w})
    +
    \beta (\bff{v}-\bff{w}) \times (\bff{\nu}\cdot \nabla) \bff{v}
    +
    \beta \bff{w} \times (\bff{\nu}\cdot \nabla) (\bff{v}-\bff{w}),
\end{align*}
we then have by Young's inequality,
\begin{align*}
    \big| \inpro{R(\bff{v})-R(\bff{w})}{\bff{v}-\bff{w}}_{\bb{L}^2} \big|
    &\leq
    C_\epsilon \norm{\bff{\nu}}{\bb{L}^\infty(\mathscr{D};\bb{R}^d)}^2 \norm{\bff{v}-\bff{w}}{\bb{L}^2}^2
    +
    \epsilon \norm{\nabla \bff{v}-\nabla \bff{w}}{\bb{L}^2}^2
    \\
    &\quad
    +
    C_\epsilon \norm{\bff{\nu}}{\bb{L}^\infty(\mathscr{D};\bb{R}^d)}^2 \norm{\bff{w}}{\bb{L}^\infty}^2 
    \norm{\bff{v}-\bff{w}}{\bb{L}^2}^2.
\end{align*}
This completes the proof of the lemma.
\end{proof}

\section{Faedo--Galerkin Approximation}\label{sec:faedo}

Let $\{\bff{e}_i\}_{i\in \bb{N}}$ denote an orthonormal basis of $\bb{L}^2$
consisting of smooth eigenfunctions~$\bff{e}_i$ of $-\Delta$ such that
\begin{align*}
-\Delta \bff{e}_i =\mu_i \bff{e}_i \ \text{ in $\mathscr{D}$}
\quad\text{and}\quad
 \frac{\partial \bff{e}_i}{\partial \bff{n}}= \bff{0} \ \text{ on } \partial \mathscr{D},
 \quad \forall i\in \bb{N},
\end{align*}
where $\mu_i\geq 0$ are the eigenvalues of $-\Delta$, associated with
$\bff{e}_i$.


For any $n\in \bb{N}$, let $\bb{V}_n:= \text{span}\{\bff{e_1},\ldots,\bff{e_n}\}$ and $\Pi_n: \bb{L}^2\to
\bb{V}_n$ be the orthogonal projection defined by
\begin{align}\label{equ:proj}
	\inpro{ \Pi_n \bff{v}}{\bff{\phi}}_{\bb{L}^2}
	= 
	\inpro{ \bff{v}}{\bff{\phi}}_{\bb{L}^2},
	\quad \forall \bff{\phi}\in \bb{V}_n,
	\;\bff{v}\in\bb{L}^2.
\end{align} 
Note that $\Pi_n$ is self-adjoint and satisfies
\begin{align}
    \label{equ:Pi v L2}
	\norm{\Pi_n \bff{v}}{\bb{L}^2} 
	&\leq 
	\norm{\bff{v}}{\bb{L}^2},
	\quad \forall
	\bff{v}\in \bb{L}^2,
	\\
    \label{equ:nab Pi v L2}
	\norm{\nabla \Pi_n \bff{v}}{\bb{L}^2}
	&\leq
	\norm{\nabla \bff{v}}{\bb{L}^2},
	\quad \forall
	\bff{v} \in \bb{H}^1.
\end{align}
Moreover,
\begin{equation*}
	\inpro{\Pi_n \Delta \bff{v}}{\bff{\phi}}_{\bb{L}^2}
	=
	\inpro{\Delta \Pi_n \bff{v}}{\bff{\phi}}_{\bb{L}^2}
    =
    -\inpro{\nabla \Pi_n \bff{v}}{\nabla \bff{\phi}}_{\bb{L}^2},
	\quad \forall
	\bff{\phi}\in \bb{V}_n,\;
	\bff{v}\in \bb{H}^2.
\end{equation*}

The Faedo--Galerkin method seeks to approximate the solution 
to~\eqref{equ:sllbar} by~$(\bff{u}_n(t), \bff{H}_n(t)) \in \bb{V}_n \times \bb{V}_n$ satisfying 
\begin{equation}\label{equ:faedo}
	\left\{
	\begin{alignedat}{2}
		&\dif \bff{u}_n
		=
		\big(\lambda_r \bff{H}_n
		-
		\lambda_e \Delta \bff{H}_n
		-
		\gamma \Pi_n (\bff{u}_n \times \bff{H}_n)
		+
		\Pi_n S(\bff{u}_n)
		\big) \,\dt
		+
		\sum_{k=1}^n \Pi_n G_k(\bff{u}_n) \, \dif W_k(t)
		&&
		\quad\text{in $(0,T)\times\mathscr{D}$,}
		\\
		&\bff{H}_n
		=
		\Delta \bff{u}_n
		+
		\bff{u}_n
		-
		\Pi_n ( |\bff{u}_n|^2 \bff{u}_n ) 
		&&
		\quad\text{in $(0,T)\times\mathscr{D}$,}
		\\
		&\bff{u}_n(0) = \bff{u}_{0n}
		&& \quad\text{in $\mathscr{D}$,}
	\end{alignedat}
	\right.
\end{equation}
where $\bff{u}_{0n}= \Pi_n \bff{u}_0 \in \bb{V}_n$. Note that $\bff{u}_n, \Delta \bff{u}_n, \bff{H}_n \in \bb{V}_n$; thus in particular $\partial \bff{u}_n/\partial \bff{n}= \partial (\Delta \bff{u}_n)/\partial \bff{n} = \partial \bff{H}_n/\partial \bff{n}= \bff{0}$ on $\partial\mathscr{D}$.

Substituting the second equation in \eqref{equ:faedo} into the first gives
\begin{align}\label{equ:subs faedo}
	\nonumber
	\dif \bff{u}_n
	&=
	\big((\lambda_r-\lambda_e) \Delta \bff{u}_n
	-
	\lambda_e \Delta^2 \bff{u}_n
	+
	\lambda_r \bff{u}_n
	-
	\lambda_r \Pi_n ( |\bff{u}_n|^2 \bff{u}_n ) 
	+
	\lambda_e \Delta \Pi_n  (|\bff{u}_n|^2 \bff{u}_n )
	-
	\gamma \Pi_n (\bff{u}_n \times \Delta \bff{u}_n)
	\\
	&\quad
	+
	\gamma \Pi_n( \bff{u}_n\times \Pi_n(|\bff{u}_n|^2 \bff{u}_n)) 
	+
	\Pi_n S(\bff{u}_n)
	\big) \,\dt
	+
	\sum_{k=1}^n \Pi_n G_k(\bff{u}_n) \, \dif W_k(t).
\end{align}
The existence of a unique local strong solution to \eqref{equ:subs faedo}, hence to~\eqref{equ:faedo}, follows from Lemmas~\ref{lem:Lip} and \ref{lem:growth} shown in the following and~\cite[Theorem~10.6]{ChuWil14}.

\begin{lemma}\label{lem:Lip}
	For each $n\in\bb{N}$ and $\bff{v}\in\bb{V}_n$, define
	\begin{align*}
		F_n^1(\bff{v}) &= \Delta \bff{v}+ \bff{v},
		\\
		F_n^2(\bff{v}) &= -\Delta^2 \bff{v},
		\\
		F_n^3(\bff{v}) &= -\Pi_n (|\bff{v}|^2 \bff{v}),
		\\
		F_n^4(\bff{v}) &= \Pi_n \Delta (|\bff{v}|^2 \bff{v}),
		\\
		F_n^5(\bff{v}) &= \Pi_n (\bff{v}\times \Delta \bff{v}),
		\\
		F_n^6(\bff{v}) &= \Pi_n( \bff{v}\times \Pi_n(|\bff{v}|^2 \bff{v})),
		\\
		F_n^7(\bff{v}) &= \Pi_n S(\bff{v}),
		\\
		G_n(\bff{v}) &= \Pi_n G_k(\bff{v}).
	\end{align*}
	Then $F_n^j$, $j=1$, $2,\ldots,7$, are well-defined mappings from $\bb{V}_n$
	into~$\bb{V}_n$. Moreover, $F_n^1$, $F_n^2$, and $G_n$ are globally Lipschitz while
	$F_n^3$, $F_n^4$, $F_n^5$, $F_n^6$, and $F_n^7$ are locally Lipschitz.
\end{lemma}

\begin{proof}
Global Lipschitz continuity for $F_n^1$ and $F_n^2$, and local Lipschitz continuity for $F_n^3$, $F_n^4$, and $F_n^5$ were shown in \cite[Lemma~3.2]{SoeTra23}. Also, $G_n$ is linear, so it is clearly globally Lipschitz. Next, for any $\bff{v}$, $\bff{w}\in \bb{V}_n$,
\begin{align*}
	\norm{F_n^6(\bff{v})-F_n^6(\bff{w})}{\bb{L}^2}
	&\leq
	\norm{\bff{v} \times F_n^3(\bff{v}) - \bff{w}\times F_n^3(\bff{w})}{\bb{L}^2}
	\\
	&\leq
	\norm{\bff{v} \times \big(F_n^3(\bff{v}) -F_n^3(\bff{w}) \big)}{\bb{L}^2}
	+
	\norm{(\bff{v}-\bff{w}) \times F_n^3(\bff{w})}{\bb{L}^2}
	\\
	&\leq
	\norm{\bff{v}}{\bb{L}^\infty} \norm{F_n^3(\bff{v}) -F_n^3(\bff{w})}{\bb{L}^2}
	+
	\norm{\bff{w}}{\bb{L}^\infty}^3 \norm{\bff{v}-\bff{w}}{\bb{L}^2}.
\end{align*}
Since all norms on $\bb{V}_n$ are equivalent and $F_n^3$ is locally Lipschitz, $F_n^6$ is locally Lipschitz.

Finally, $F_n^7(\bff{v})= \Pi_n R(\bff{v})+ \Pi_n L(\bff{v})$. The map $L$ is globally Lipschitz by assumption. Moreover, for any $\bff{v}$, $\bff{w}\in \bb{V}_n$,
\begin{align*}
	&\norm{\Pi_n R(\bff{v})- \Pi_n R(\bff{w})}{\bb{L}^2}
 \\
	&\leq
	\norm{(\bff{\nu}\cdot \nabla) (\bff{v}-\bff{w})}{\bb{L}^2}
	+
	\beta \norm{(\bff{v}-\bff{w}) \times (\bff{\nu}\cdot\nabla)\bff{v}}{\bb{L}^2}
	+
	\beta \norm{\bff{w} \times (\bff{\nu}\cdot \nabla)(\bff{v}-\bff{w})}{\bb{L}^2}
	\\
	&\leq
	\norm{\bff{\nu}}{\bb{L}^\infty(\mathscr{D};\bb{R}^d)} \norm{\bff{v}-\bff{w}}{\bb{H}^1}
	+
	\beta \norm{\bff{\nu}}{\bb{L}^\infty(\mathscr{D};\bb{R}^d)} 
        \norm{\bff{v}}{\bb{H}^1} \norm{\bff{v}-\bff{w}}{\bb{L}^\infty}
	+
	\beta \norm{\bff{\nu}}{\bb{L}^\infty(\mathscr{D};\bb{R}^d)} \norm{\bff{w}}{\bb{L}^\infty}
	\norm{\bff{v}-\bff{w}}{\bb{H}^1}.
\end{align*}
Since all norms on $\bb{V}_n$ are equivalent, this implies $F_n^7$ is locally Lipschitz. This completes the proof.
\end{proof}

\begin{lemma}\label{lem:growth}
For each $n\in \bb{N}$ and $\bff{v} \in \bb{V}_n$,
\begin{align*}
	\inpro{\sum_{j=1}^7 F_n^j(\bff{v})}{\bff{v}}_{\bb{L}^2}
	+
	\sum_{k=1}^n \norm{\Pi_n G_k(\bff{v})}{\bb{L}^2}^2
	\leq
	C_n \left(1+ \norm{\bff{v}}{\bb{L}^2}^2 \right),
\end{align*}
where $C_n$ is a constant, possibly depending on $n$.
\end{lemma}

\begin{proof}
We will estimate each term on the left-hand side. Firstly, since $\bff{v},\Delta \bff{v}\in \bb{V}_n$, using integration by parts we have
\begin{align*}
	\inpro{F_n^1(\bff{v})+F_n^2(\bff{v})}{\bff{v}}_{\bb{L}^2}
	=
	-
	\norm{\nabla \bff{v}}{\bb{L}^2}^2
	+
	\norm{\bff{v}}{\bb{L}^2}^2
	-
	\norm{\Delta \bff{v}}{\bb{L}^2}^2
	\leq
	\norm{\bff{v}}{\bb{L}^2}^2.
\end{align*}
Moreover, noting \eqref{equ:proj} and \eqref{equ:ddn v2v}, we obtain
\begin{align*}
	\inpro{F_n^3(\bff{v})+F_n^4(\bff{v})}{\bff{v}}_{\bb{L}^2}
	=
	- \norm{\bff{v}}{\bb{L}^4}^4
	- \norm{|\bff{v}|\, |\nabla \bff{v}|}{\bb{L}^2}^2
	- \frac{1}{2} \norm{\nabla |\bff{v}|^2}{\bb{L}^2}^2
	\leq 
	- \norm{|\bff{v}|\, |\nabla \bff{v}|}{\bb{L}^2}^2.
\end{align*}
Similarly, by \eqref{equ:proj} and the fact that $(\bff{a}\times\bff{b})\cdot \bff{a}=0$ for any $\bff{a},\bff{b}\in \bb{R}^3$, we have
\begin{align*}
	\inpro{F_n^5(\bff{v})+F_n^6(\bff{v})}{\bff{v}}_{\bb{L}^2} = 0.
\end{align*}
Also, by Young's inequality and the assumptions on $\bff{\nu}$ and $L$,
\begin{align*}
	\inpro{F_n^7(\bff{v})}{\bff{v}}_{\bb{L}^2}
	=
	\inpro{R(\bff{v})+L(\bff{v})}{\bff{v}}_{\bb{L}^2}
	\leq
	C \norm{\bff{\nu}}{\bb{L}^2(\mathscr{D};\bb{R}^d)}^2 
	+
	\epsilon \norm{|\bff{v}|\, |\nabla \bff{v}|}{\bb{L}^2}^2
	+
	C \big( 1+\norm{\bff{v}}{\bb{L}^2}^2 \big)
\end{align*}
for any $\epsilon>0$. Finally, by H\"older's inequality and \eqref{equ:Pi v L2},
\begin{align*}
	\sum_{k=1}^n \norm{\Pi_n G_k(\bff{v})}{\bb{L}^2}^2
	\leq
	C \sigma_g \sigma_h \left(1+ \norm{\bff{v}}{\bb{L}^2}^2 \right).
\end{align*}
Adding all the above inequalities yields the required result.
\end{proof}


Before proceeding to prove uniform bounds for the approximate solutions $\bff{u}_n$, to simplify notations we write~\eqref{equ:faedo} as
\begin{equation}\label{equ:faedo ito}
	\left\{
	\begin{alignedat}{2}
		&\dif \bff{u}_n
		=
		F_n(\bff{u}_n, \bff{H}_n) \,\dt
		+
		\sum_{k=1}^n G_k(\bff{u}_n) \, \dif W_k(t),
		&&
		\quad\text{in $(0,T)\times\mathscr{D}$,}
		\\
		&\bff{H}_n
		=
		\Delta \bff{u}_n
		+
		\bff{u}_n
		-
		\Pi_n ( |\bff{u}_n|^2 \bff{u}_n ) 
		&&
		\quad\text{in $(0,T)\times\mathscr{D}$,}
		\\
		&\bff{u}_n(0) = \bff{u}_{0n}
		&& \quad\text{in $\mathscr{D}$,}
	\end{alignedat}
	\right.
\end{equation}
where
\begin{align}\label{equ:Fn}
	F_n(\bff{u}_n, \bff{H}_n) 
	&:=
	\lambda_r \bff{H}_n
	-
	\lambda_e \Delta \bff{H}_n
	-
	\gamma \Pi_n(\bff{u}_n\times \bff{H}_n)
	-
	\Pi_n S(\bff{u}_n),
	\\
	\label{equ:Gn}
	G_k(\bff{u}_n)
	&:=
	\Pi_n(\bff{g}_k+ \gamma \bff{u}_n \times \bff{h}_k).
\end{align}
Now, we prove uniform bounds for the approximate solutions $\bff{u}_n$.

\begin{proposition}\label{pro:E sup un L2}
For any $p\geq 1$, $n\in \bb{N}$, $t\in (0,\infty)$,
\begin{align*}
	&\bb{E} \left[\sup_{\tau \in [0,t]} \norm{\bff{u}_n(\tau)}{\bb{L}^2}^{2p} \right]
	+
	\bb{E} \left[\left( \int_0^t \norm{\nabla \bff{u}_n(s)}{\bb{L}^2}^2 \ds \right)^p \right]
	+
	\bb{E} \left[ \left( \lambda_e \int_0^t \norm{\Delta \bff{u}_n(s)}{\bb{L}^2}^2 \ds \right)^p \right]
	\\
	&\quad
    +
	\bb{E} \left[\left(\int_0^t \norm{|\bff{u}_n(s)|\, |\nabla \bff{u}_n(s)|}{\bb{L}^2}^2 \ds \right)^p \right]
	+
	\bb{E} \left[\left(\int_0^t \norm{\bff{u}_n(s)}{\bb{L}^4}^4 \ds \right)^p \right]
	\leq 
	C\left(1+\norm{\bff{u}_0}{\bb{L}^2}^{2p}\right) e^{Ct},
\end{align*}
where $C$ is a positive constant depending on $p$,  $\sigma_g$, and $\sigma_h$ (but is independent of $n$ and $t$).
\end{proposition}

\begin{proof}
Let $\psi:\bb{V}_n\to \bb{R}$ be a function defined by $\bff{v}\mapsto \frac{1}{2} \norm{\bff{v}}{\bb{L}^2}^2$. Then
\[
	\psi'(\bff{v})(\bff{h}) 
	= \inpro{\bff{v}}{\bff{h}}_{\bb{L}^2}
	\quad \text{and} \quad
	\psi''(\bff{v})(\bff{h},\bff{k}) 
	= \inpro{\bff{k}}{\bff{h}}_{\bb{L}^2},
	\quad \forall \bff{h}, \bff{k} \in \bb{V}_n.
\]
By It\^o's Lemma,
\begin{equation}\label{equ:d psi un}
	\dif \psi(\bff{u}_n)
	=
	\inpro{\dif \bff{u}_n}{\bff{u}_n}_{\bb{L}^2}
	+
	\frac{1}{2} \inpro{\dif \bff{u}_n}{\dif \bff{u}_n}_{\bb{L}^2}.
\end{equation}
Therefore, by \eqref{equ:d psi un} and \eqref{equ:faedo ito}, we deduce that
\begin{equation}\label{equ:d un L2}
	\frac{1}{2} \dif \norm{\bff{u}_n(t)}{\bb{L}^2}^2
	=
	\left( \inpro{F_n(\bff{u}_n,\bff{H}_n)}{\bff{u}_n}_{\bb{L}^2}
	+
	\frac{1}{2} \sum_{k=1}^n \norm{G_k(\bff{u}_n)}{\bb{L}^2}^2
	\right) \dt 
	+
	\sum_{k=1}^n \inpro{G_k(\bff{u}_n)}{\bff{u}_n}_{\bb{L}^2}  \dif W_k(t). 
\end{equation}
Note that by \eqref{equ:Fn} and \eqref{equ:Gn}, we have
\begin{align}\label{equ:Fn un un}
	\nonumber
	\inpro{F_n(\bff{u}_n,\bff{H}_n)}{\bff{u}_n}_{\bb{L}^2}
	&=
	\lambda_r \inpro{\bff{H}_n}{\bff{u}_n}_{\bb{L}^2}
	-
	\lambda_e \inpro{\Delta \bff{H}_n}{\bff{u}_n}_{\bb{L}^2}
	+
	\inpro{S(\bff{u}_n)}{\bff{u}_n}_{\bb{L}^2}
	\\
	\nonumber
	&=
	\lambda_r \inpro{\bff{H}_n}{\bff{u}_n}_{\bb{L}^2}
	-
	\lambda_e \inpro{\Delta \bff{H}_n}{\bff{u}_n}_{\bb{L}^2}
	-
	\inpro{(\bff{\nu}\cdot \nabla)\bff{u}_n}{\bff{u}_n}_{\bb{L}^2}
	+
	\inpro{L(\bff{u}_n)}{\bff{u}_n}_{\bb{L}^2}
	\\
	\nonumber
	&=
	(\lambda_e-\lambda_r)\norm{\nabla \bff{u}_n}{\bb{L}^2}^2
	-
	\lambda_e \norm{\Delta \bff{u}_n}{\bb{L}^2}^2
	+
	\lambda_r \norm{\bff{u}_n}{\bb{L}^2}^2
	-
	\lambda_r \norm{\bff{u}_n}{\bb{L}^4}^4
	-
	\lambda_e \norm{|\bff{u}_n| |\nabla \bff{u}_n|}{\bb{L}^2}^2
	\\
	&\quad
	-
	2\lambda_e\norm{\bff{u}_n\cdot \nabla \bff{u}_n}{\bb{L}^2}^2
	-
	\inpro{(\bff{\nu}\cdot \nabla)\bff{u}_n}{\bff{u}_n}_{\bb{L}^2}
	+
	\inpro{L(\bff{u}_n)}{\bff{u}_n}_{\bb{L}^2}
\end{align}
and
\begin{equation}\label{equ:Gn un un}
	\inpro{G_k(\bff{u}_n)}{\bff{u}_n}_{\bb{L}^2}
	=
	\inpro{\bff{g}_k}{\bff{u}_n}_{\bb{L}^2}.
\end{equation}
Substituting \eqref{equ:Fn un un} and \eqref{equ:Gn un un} into~\eqref{equ:d un L2} and applying H\"older's inequality yield
\begin{align}\label{equ:un L2}
	\nonumber
	&\frac{1}{2} \norm{\bff{u}_n(t)}{\bb{L}^2}^2
	+
	\lambda_r \int_0^t \norm{\nabla \bff{u}_n(s)}{\bb{L}^2}^2 \ds
	+
	\lambda_e \int_0^t \norm{\Delta \bff{u}_n(s)}{\bb{L}^2}^2 \ds
	+
	\lambda_r \int_0^t \norm{\bff{u}_n(s)}{\bb{L}^4}^4 \ds
	\\
	\nonumber
	&\quad
	+
	\lambda_e \int_0^t \norm{|\bff{u}_n(s)|\, |\nabla \bff{u}_n(s)|}{\bb{L}^2}^2 \ds
	+
	2\lambda_e \int_0^t \norm{\bff{u}_n(s) \cdot \nabla \bff{u}_n(s)}{\bb{L}^2}^2 \ds
	\\
	\nonumber
	&=
	\frac{1}{2} \norm{\bff{u}_{0}}{\bb{L}^2}^2
	+
	\lambda_r \int_0^t \norm{\bff{u}_n(s)}{\bb{L}^2}^2 \ds
    +
     \lambda_e \int_0^t \norm{\nabla \bff{u}_n(s)}{\bb{L}^2}^2 \ds
	-
	\int_0^t \inpro{(\bff{\nu}\cdot \nabla)\bff{u}_n(s)}{\bff{u}_n(s)}_{\bb{L}^2} \ds
	\\
	\nonumber
	&\quad 
 	+
	\int_0^t \inpro{L(\bff{u}_n(s))}{\bff{u}_n(s)}_{\bb{L}^2} \ds
	+
	\frac{1}{2} \sum_{k=1}^n \int_0^t \norm{G_k(\bff{u}_n)}{\bb{L}^2}^2 \ds
	+
	\sum_{k=1}^n \int_0^t \inpro{\bff{g}_k}{\bff{u}_n(s)}_{\bb{L}^2} \dif W_k(s)
	\\
	\nonumber
	&\leq
	\frac{1}{2} \norm{\bff{u}_{0}}{\bb{L}^2}^2
	+
	\lambda_r \int_0^t \norm{\bff{u}_n(s)}{\bb{L}^2}^2 \ds
    +
    \lambda_e \int_0^t \norm{\nabla \bff{u}_n(s)}{\bb{L}^2}^2 \ds
	\\
	\nonumber
	&\quad
 	+
	C \int_0^t \norm{\bff{\nu}}{\bb{L}^2(\mathscr{D};\bb{R}^d)} \norm{|\bff{u}_n(s)|\, |\nabla \bff{u}_n(s)|}{\bb{L}^2} \ds
        +
	C \int_0^t \big(1+ \norm{\bff{u}_n(s)}{\bb{L}^2}^2 \big) \ds
        \\
        \nonumber
        &\quad
	+
	\frac{1}{2} \sum_{k=1}^n \int_0^t \big( \norm{\bff{g}_k}{\bb{L}^2}^2 + \norm{\bff{h}_k}{\bb{L}^\infty}^2 \norm{\bff{u}_n(s)}{\bb{L}^2}^2 \big) \ds
	+
	\sum_{k=1}^n \int_0^t \inpro{\bff{g}_k}{\bff{u}_n(s)}_{\bb{L}^2} \dif W_k(s)
	\\
         \nonumber
	&\leq
	Ct
	+
	C\norm{\bff{u}_0}{\bb{L}^2}^2 
	+ 
	C\int_0^t \norm{\bff{u}_n(s)}{\bb{L}^2}^2 \ds
        +
	\epsilon \int_0^t \norm{\Delta \bff{u}_n(s)}{\bb{L}^2}^2 \ds
	+
	\epsilon \int_0^t \norm{|\bff{u}_n(s)|\, |\nabla \bff{u}_n(s)|}{\bb{L}^2}^2 \ds
        \\
        &\quad
	+
	\sum_{k=1}^n \int_0^t \inpro{\bff{g}_k}{\bff{u}_n(s)}_{\bb{L}^2} \dif W_k(s),
\end{align}
for any $\epsilon>0$, where in the penultimate step we used the assumptions on $L$ and $\bff{\nu}$, while in the final step we used the assumptions on $\bff{g}_k$ and $\bff{h}_k$, Young's inequality, and interpolation inequality~\eqref{equ:nabla v L2}.

It follows from \eqref{equ:un L2} and Jensen's inequality that, after rearranging the terms, for any $p\geq 1$,


\begin{align}\label{equ:un L2 2p}
	\nonumber
	&\sup_{\tau\in [0,t]} \norm{\bff{u}_n(\tau)}{\bb{L}^2}^{2p}
	+
	\left( \int_0^t \norm{\nabla \bff{u}_n(s)}{\bb{L}^2}^2 \ds \right)^p
	+
	\left( \int_0^t \norm{\Delta \bff{u}_n(s)}{\bb{L}^2}^2 \ds \right)^p
	+
	\left( \int_0^t \norm{\bff{u}_n(s)}{\bb{L}^4}^4 \ds \right)^p
	\\
	&\leq
	Ct^p
	+
	C\norm{\bff{u}_0}{\bb{L}^2}^{2p}
	+
	C\int_0^t \norm{\bff{u}_n(s)}{\bb{L}^2}^{2p} \ds
	+
	C \sup_{\tau\in [0,t]} \left|\sum_{k=1}^n \int_0^\tau \inpro{\bff{g}_k}{\bff{u}_n(s)}_{\bb{L}^2} \dif W_k(s) \right|^p.
\end{align}
It remains to estimate the last term in \eqref{equ:un L2 2p}. To this end, by the Burkholder--Davis--Gundy inequality and the H\"older inequality (noting the assumption on $\bff{g}_k$), we have
\begin{align*}
	\bb{E} \left[\sup_{\tau \in [0,t]} \left|\sum_{k=1}^n \int_0^\tau \inpro{ \bff{g}_k}{\bff{u}_n(s)}_{\bb{L}^2} \dif W_k(s) \right|^p \right]
	&\leq
	C_p \,\bb{E} \left[\left( \sum_{k=1}^n \int_0^t \big|\inpro{\bff{g}_k}{\bff{u}_n(s)}_{\bb{L}^2} \big|^2 \,\ds \right)^{p/2} \right]
	\\
	&\leq
	C \, \bb{E} \left[ \left(\sum_{k=1}^n \int_0^t \norm{\bff{g}_k}{\bb{L}^2}^2 \norm{\bff{u}_n(s)}{\bb{L}^2}^2 \ds \right)^{p/2} \right]
	\\
	&=
	C \left( \sum_{k=1}^n \norm{\bff{g}_k}{\bb{L}^2}^2 \right)^{p/2} 
	\bb{E} \left[ \left( \int_0^t \norm{\bff{u}_n(s)}{\bb{L}^2}^2 \ds \right)^{p/2} \right]
	\\
	&\leq
	C+ C\int_0^t \bb{E} \left[\sup_{\tau \in [0,s]} \norm{\bff{u}_n(\tau)}{\bb{L}^2}^{2p} \right] \ds.
\end{align*}
Therefore, taking expectation on both sides of~\eqref{equ:un L2 2p}, we obtain
\begin{align*}
	&\bb{E} \left[\sup_{\tau \in [0,t]} \norm{\bff{u}_n(\tau)}{\bb{L}^2}^{2p} \right]
	+
	\bb{E} \left[\left( \int_0^t \norm{\nabla \bff{u}_n(s)}{\bb{L}^2}^2 \ds \right)^p \right]
	+
	\bb{E} \left[ \left( \int_0^t \norm{\Delta \bff{u}_n(s)}{\bb{L}^2}^2 \ds \right)^p \right]
	\\
	&\quad
	+
	\bb{E} \left[\left( \int_0^t \norm{|\bff{u}_n(s)|\, |\nabla \bff{u}_n(s)|}{\bb{L}^2}^2 \ds \right)^p \right]
	+
	\bb{E} \left[\left(\int_0^t \norm{\bff{u}_n(s)}{\bb{L}^4}^4 \ds \right)^p \right]
	\\
    &\leq 
	C(1+t^p) +
	C\norm{\bff{u}_0}{\bb{L}^2}^{2p} + C\int_0^t \bb{E} \left[\sup_{\tau \in [0,s]} \norm{\bff{u}_n(\tau)}{\bb{L}^2}^{2p} \right] \ds.
\end{align*}
The required estimate then follows by the Gronwall inequality.
\end{proof}

\begin{remark}\label{rem:small lambda e}
If $\lambda_e<\lambda_r$, then in the last step of~\eqref{equ:un L2} the integral of $\lambda_e \norm{\nabla \bff{u}_n}{\bb{L}^2}^2$ can be directly absorbed to the left-hand side without using the interpolation inequality $\norm{\nabla \bff{u}_n}{\bb{L}^2}^2 \leq C_\epsilon \norm{\bff{u}_n}{\bb{L}^2}^2 + \epsilon \norm{\Delta \bff{u}_n}{\bb{L}^2}^2$, in which case the constant $C$ in Proposition~\ref{pro:E sup un L2} is independent of $\lambda_e$.
\end{remark}

\begin{proposition}\label{pro:E sup nab un L2}
	For any $p\geq 1$, $n\in \bb{N}$, $t\in (0,\infty)$,
	\begin{align*}
		\bb{E} \left[\sup_{\tau \in [0,t]} \left(\norm{\nabla \bff{u}_n(\tau)}{\bb{L}^2}^{2p} + \norm{\bff{u}_n(\tau)}{\bb{L}^4}^{4p} \right) \right]
		&+
		\bb{E} \left[\left(\lambda_e \int_0^t \norm{\bff{H}_n(s)}{\bb{H}^1}^2 \ds \right)^p \right]
        \\
		&\leq 
		C \left(\norm{\bff{u}_0}{\bb{H}^1}^{2p} + \norm{\bff{u}_0}{\bb{L}^4}^{4p} \right) + C\left(1+\norm{\bff{u}_0}{\bb{L}^2}^{2p}\right) e^{Ct},
	\end{align*}
	where $C$ is a positive constant depending on $p$, $\sigma_g$, and $\sigma_h$ (but is independent of $n$ and $t$).
\end{proposition}

\begin{proof}
Let $\psi:\bb{V}_n\to \bb{R}$ be a function defined by $\bff{v}\mapsto \frac{1}{2} \norm{\nabla \bff{v}}{\bb{L}^2}^2 + \frac{1}{4} \norm{\bff{v}}{\bb{L}^4}^4 - \frac{1}{2} \norm{\bff{v}}{\bb{L}^2}^2$. Note that by It\^o's formula~\cite{Kry10}, we have
\begin{align*}
	\frac{1}{2} \dif \norm{\nabla \bff{u}_n(t)}{\bb{L}^2}^2
	&=
	\left( -\inpro{F(\bff{u}_n, \bff{H}_n)}{\Delta \bff{u}_n}_{\bb{L}^2}
	+
	\frac{1}{2} \sum_{k=1}^n \norm{\nabla G_k(\bff{u}_n)}{\bb{L}^2}^2
	\right) \dt
	-
	\sum_{k=1}^n \inpro{G_k(\bff{u}_n)}{\Delta \bff{u}_n}_{\bb{L}^2} \dif W_k(t),
	\\
	\frac{1}{4} \dif \norm{\bff{u}_n(t)}{\bb{L}^4}^4
	&=
	\left( \inpro{F(\bff{u}_n, \bff{H}_n)}{|\bff{u}_n|^2 \bff{u}_n}_{\bb{L}^2}
	+
	\frac{3}{2} \sum_{k=1}^n \norm{|\bff{u}_n| |G_k(\bff{u}_n)|}{\bb{L}^2}^2
	\right) \dt
    \\
    &\quad
	+
	\sum_{k=1}^n \inpro{G_k(\bff{u}_n)}{|\bff{u}_n|^2 \bff{u}_n}_{\bb{L}^2} \dif W_k(t),
	\\
	-\frac{1}{2} \dif \norm{\bff{u}_n(t)}{\bb{L}^2}^2
	&=
	\left( - \inpro{F(\bff{u}_n, \bff{H}_n)}{\bff{u}_n}_{\bb{L}^2}
	-
	\frac{1}{2} \sum_{k=1}^n \norm{G_k(\bff{u}_n)}{\bb{L}^2}^2
	\right) \dt
	-
	\sum_{k=1}^n \inpro{G_k(\bff{u}_n)}{\bff{u}_n}_{\bb{L}^2} \dif W_k(t).
\end{align*}
Adding the above equations and noting~\eqref{equ:faedo ito}, we obtain
\begin{align}\label{equ:d psi un L4}
	\nonumber
	\dif \psi(\bff{u}_n)
	&=
	\left( -\inpro{F_n(\bff{u}_n,\bff{H}_n)}{\bff{H}_n}_{\bb{L}^2}
	+
	\sum_{k=1}^n \left[
	\frac{1}{2}  \norm{\nabla G_k(\bff{u}_n)}{\bb{L}^2}^2
	+
	\frac{3}{2} \norm{|\bff{u}_n|\, |G_k(\bff{u}_n)|}{\bb{L}^2}^2
	-
	\frac{1}{2} \norm{G_k(\bff{u}_n)}{\bb{L}^2}^2 \right]
	\right) \dt
	\\
	&\quad
	-
	\sum_{k=1}^n
	\inpro{G_k(\bff{u}_n)}{\bff{H}_n}_{\bb{L}^2} \dif W_k(t).
\end{align}
By \eqref{equ:Fn} and \eqref{equ:Gn}, noting that $\bff{H}_n\in \bb{V}_n$ so that $\partial \bff{H}_n/\partial \bff{n}=\bff{0}$ on $\partial \mathscr{D}$, we have
\begin{align}\label{equ:Fn un Hn}
	\inpro{F_n(\bff{u}_n, \bff{H}_n)}{\bff{H}_n}_{\bb{L}^2}
	=
	\lambda_r \norm{\bff{H}_n}{\bb{L}^2}^2
	+
	\lambda_e \norm{\nabla \bff{H}_n}{\bb{L}^2}^2
	+
	\inpro{S(\bff{u}_n)}{\bff{H}_n}_{\bb{L}^2},
\end{align}
and
\begin{align}\label{equ:Gn un Hn}
	\inpro{G_k(\bff{u}_n)}{\bff{H}_n}_{\bb{L}^2}
	=
	\inpro{\bff{g}_k+ \gamma \bff{u}_n \times \bff{h}_k}{\bff{H}_n}_{\bb{L}^2}.
\end{align}
Substituting~\eqref{equ:Fn un Hn} and~\eqref{equ:Gn un Hn} into~\eqref{equ:d psi un L4}, then applying H\"older's inequality and inequality \eqref{equ:Rv w} (noting the assumptions on $R$ and $L$) yield
\begin{align}\label{equ:psi un nab Hn}
	&\psi(\bff{u}_n(t))
	+
	\lambda_r \int_0^t \norm{\bff{H}_n(s)}{\bb{L}^2}^2 \ds
	+
	\lambda_e \int_0^t \norm{\nabla \bff{H}_n(s)}{\bb{L}^2}^2 \ds
	+
	\frac{1}{2} \sum_{k=1}^n \int_0^t \norm{G_k(\bff{u}_n(s))}{\bb{L}^2}^2 \ds
	\nonumber \\
	&=
	\psi(\bff{u}_n(0))
	-
	\int_0^t \inpro{R(\bff{u}_n)}{\bff{H}_n}_{\bb{L}^2} \ds
	-
	\int_0^t \inpro{L(\bff{u}_n)}{\bff{H}_n}_{\bb{L}^2} \ds
	+
	\frac{1}{2} \sum_{k=1}^n \int_0^t \norm{\nabla G_k(\bff{u}_n(s))}{\bb{L}^2}^2 \ds
	\nonumber\\
	&\quad 
	+
	\frac{3}{2} \sum_{k=1}^n \int_0^t \norm{|\bff{u}_n(s)|\, |G_k(\bff{u}_n(s))|}{\bb{L}^2}^2 \ds
	-
	\sum_{k=1}^n \int_0^t \inpro{\bff{g}_k+ \gamma \bff{u}_n(s) \times \bff{h}_k}{\bff{H}_n}_{\bb{L}^2} \dif W_k(s)
	\nonumber\\
	&\leq
	\psi(\bff{u}_n(0))
	+
	C \int_0^t \norm{\bff{\nu}}{\bb{L}^\infty(\mathscr{D};\bb{R}^d)}^2 \left(\norm{\nabla \bff{u}_n}{\bb{L}^2}^2 + \norm{|\bff{u}_n|\, |\nabla \bff{u}_n|}{\bb{L}^2}^2 \right) \ds
	+
	\epsilon \int_0^t \norm{\bff{H}_n}{\bb{L}^2}^2 \ds
	+
	C\int_0^t \norm{\bff{u}_n}{\bb{L}^2}^2 \ds
	\nonumber\\
	&\quad
	+
	C \sum_{k=1}^n \int_0^t \left( \norm{\nabla \bff{g}_k}{\bb{L}^2}^2 + \norm{\bff{h}_k}{\bb{L}^2}^2 \norm{\nabla \bff{u}_n}{\bb{L}^2}^2 + \norm{\bff{u}_n}{\bb{L}^2}^2 \norm{\nabla \bff{h}_k}{\bb{L}^2}^2 \right) \ds
	\nonumber\\
	&\quad 
	+
	C \sum_{k=1}^n \int_0^t \left( \norm{\bff{g}_k}{\bb{L}^\infty}^2 \norm{\bff{u}_n}{\bb{L}^2}^2 + \norm{\bff{h}_k}{\bb{L}^\infty}^2 \norm{\bff{u}_n}{\bb{L}^4}^4 \right) \ds
	-
	\sum_{k=1}^n \int_0^t \inpro{\bff{g}_k + \gamma \bff{u}_n \times \bff{h}_k}{\bff{H}_n}_{\bb{L}^2} \dif W_k(s)
	\nonumber\\
	&\leq
	\psi(\bff{u}_n(0))+ \epsilon \int_0^t \norm{\bff{H}_n}{\bb{L}^2}^2 \ds 
    +
    C \int_0^t \norm{\nabla \bff{u}_n}{\bb{L}^2}^2 \ds
	+
	C\int_0^t \norm{|\bff{u}_n|\, |\nabla \bff{u}_n|}{\bb{L}^2}^2 \ds
	\nonumber\\
	&\quad 
	+
	C\left( \sum_{k=1}^n \left(\norm{\bff{g}_k}{\bb{H}^1}^2+\norm{\bff{h}_k}{\bb{H}^1}^2 \right)\right)
    \left( \int_0^t \left(1+ \norm{\bff{u}_n}{\bb{L}^2}^2 + \norm{\nabla \bff{u}_n}{\bb{L}^2}^2 \right)\ds \right)
    \nonumber\\
    &\quad
	+
	\left( \sum_{k=1}^n \left(\norm{\bff{g}_k}{\bb{L}^\infty}^2+\norm{\bff{h}_k}{\bb{L}^\infty}^2 \right) \right) \left( \int_0^t \left( \norm{\bff{u}_n}{\bb{L}^2}^2 + \norm{\bff{u}_n}{\bb{L}^4}^4 \right) \ds \right)
	-
	\sum_{k=1}^n \int_0^t \inpro{\bff{g}_k + \gamma \bff{u}_n \times \bff{h}_k}{\bff{H}_n}_{\bb{L}^2} \dif W_k(s)
	\nonumber\\
	&\leq
	\psi(\bff{u}_n(0))+ Ct + \epsilon \int_0^t \norm{\bff{H}_n}{\bb{L}^2}^2 \ds 
	+
	\int_0^t \norm{|\bff{u}_n|\, |\nabla \bff{u}_n|}{\bb{L}^2}^2 \ds
	\nonumber\\
	&\quad 
	+
	\int_0^t \left( \norm{\bff{u}_n}{\bb{L}^2}^2 + \norm{\bff{u}_n}{\bb{L}^4}^4 + \norm{\nabla \bff{u}_n}{\bb{L}^2}^2 \right) \ds 
	-
	\sum_{k=1}^n \int_0^t \inpro{\bff{g}_k + \gamma \bff{u}_n \times \bff{h}_k}{\bff{H}_n}_{\bb{L}^2} \dif W_k(s)
\end{align}
for any $\epsilon>0$, where in the last step we used the assumptions on $\bff{g}_k$ and $\bff{h}_k$.
By Jensen's inequality, after rearranging the terms and noting Proposition~\ref{pro:E sup un L2}, it follows that
\begin{align}\label{equ:sup nab L2}
	\nonumber
	&\sup_{\tau \in [0,t]} \left(\norm{\nabla \bff{u}_n(\tau)}{\bb{L}^2}^{2p} + \norm{\bff{u}_n(\tau)}{\bb{L}^4}^{4p} \right)
	+
	\left(\int_0^t \norm{\bff{H}_n(s)}{\bb{L}^2}^2 \ds \right)^p
	+
	\left( \int_0^t \norm{\nabla \bff{H}_n(s)}{\bb{L}^2}^2 \ds \right)^p
	\\
	\nonumber
	&\leq
	   C\big(\psi(\bff{u}_n(0))+t \big)^p 
       +
	C\left(\int_0^t \left(\norm{\bff{u}_n(s)}{\bb{L}^2}^{2} 
	+ 
	\norm{\bff{u}_n(s)}{\bb{L}^4}^{4}
	+
	\norm{\nabla \bff{u}_n(s)}{\bb{L}^2}^{2} + \norm{|\bff{u}_n|\, |\nabla \bff{u}_n|}{\bb{L}^2}^2  \right) \ds \right)^p
	\\
	&\quad
	+
	C\sup_{\tau \in [0,t]} \left| \sum_{k=1}^n \int_0^\tau \inpro{\bff{g}_k + \gamma \bff{u}_n \times \bff{h}_k}{\bff{H}_n}_{\bb{L}^2} \dif W_k(s) \right|^p.
\end{align}
Now, we estimate the last term on the right-hand side of \eqref{equ:sup nab L2}. By the Burkholder--Davis--Gundy inequality and the H\"older inequality, we have
\begin{align}\label{equ:sup Hn dW}
	\nonumber
	&\bb{E} \left[ \sup_{\tau \in [0,t]} \left| \sum_{k=1}^n \int_0^\tau \inpro{\bff{g}_k + \gamma \bff{u}_n \times \bff{h}_k}{\bff{H}_n}_{\bb{L}^2} \dif W_k(s) \right|^p \right]
	\\
	\nonumber
	&\leq
	C_p\, \bb{E} \left[ \left( \sum_{k=1}^n \int_0^t \big| \inpro{\bff{g}_k + \gamma \bff{u}_n \times \bff{h}_k}{\bff{H}_n}_{\bb{L}^2} \big|^2 \,\ds \right)^{p/2} \right]
	\\
	\nonumber
	&\leq 
	C \left( \sum_{k=1}^n \left(\norm{\bff{g}_k}{\bb{L}^\infty}^2 + \norm{\bff{h}_k}{\bb{L}^\infty}^2 \right) \right)^{p/2} 
	\bb{E} \left[ \left( \int_0^t
	\big(1+ \norm{\bff{u}_n}{\bb{L}^2}^2 \big)
	\norm{\bff{H}_n}{\bb{L}^2}^2 \ds \right)^{p/2} \right]
	\\
	\nonumber
	&\leq
	C \, \bb{E} \left[ \left(1+ \sup_{\tau\in (0,t)} \norm{\bff{u}_n(\tau)}{\bb{L}^2}^2\right)^{p/2} \left(\int_0^t \norm{\bff{H}_n(s)}{\bb{L}^2}^2 \ds \right)^{p/2} \right]
	\\
	\nonumber
	&\leq
	C\bb{E} \left[ \left(1+ \sup_{\tau\in [0,t]} \norm{\bff{u}_n(\tau)}{\bb{L}^2}^2 \right)^p \right]
	+
	\epsilon\bb{E} \left[ \left(\int_0^t 
	\norm{\bff{H}_n(s)}{\bb{L}^2}^2 \ds \right)^p \right]
	\\
	&\leq
	C\left(1+\norm{\bff{u}_0}{\bb{L}^2}^{2p}\right) e^{Ct} + \epsilon \bb{E} \left[ \left(\int_0^t \norm{\bff{H}_n}{\bb{L}^2}^2 \ds \right)^p \right]
\end{align}
for any $\epsilon>0$, where in the penultimate step we used Young's inequality, while in the final step we used Proposition~\ref{pro:E sup un L2}. Taking expectation on both sides of~\eqref{equ:sup nab L2}, choosing $\epsilon>0$ sufficiently small, and substituting~\eqref{equ:sup Hn dW} into the resulting expression, we obtain
\begin{align*}
	&\bb{E} \left[\sup_{\tau \in [0,t]} \left(\norm{\nabla \bff{u}_n(\tau)}{\bb{L}^2}^{2p} + \norm{\bff{u}_n(\tau)}{\bb{L}^4}^{4p} \right) \right]
	+
	\bb{E} \left[ \left(\int_0^t \norm{\bff{H}_n(s)}{\bb{L}^2}^2 \ds \right)^p \right]
	+
	\bb{E} \left[ \left( \int_0^t \norm{\nabla \bff{H}_n(s)}{\bb{L}^2}^2 \ds \right)^p \right]
	\\
	&\leq
	C \left(\norm{\bff{u}_0}{\bb{H}^1}^{2p} + \norm{\bff{u}_0}{\bb{L}^4}^{4p} \right) + C\left(1+\norm{\bff{u}_0}{\bb{L}^2}^{2p}\right) e^{Ct}
    + 
    \bb{E}\left[\left(\int_0^t \norm{\bff{u}_n(s)}{\bb{H}^1}^2 + \norm{\bff{u}_n(s)}{\bb{L}^4}^4 \ds\right)^p\right].
\end{align*}
Applying Proposition~\ref{pro:E sup un L2} again yields the required estimate.
\end{proof}

\begin{proposition}\label{pro:E nab Delta un}
For any $p\geq 1$, $n\in \bb{N}$, $t\in (0,\infty)$,
\begin{equation*}
	\bb{E}\left[ \left(\int_0^t \norm{\nabla \Delta \bff{u}_n(s)}{\bb{L}^2}^2 \ds \right)^p \right]
	\leq
    C\left(1+e^{Ct}\right),
\end{equation*}
where $C$ is a positive constant depending on $p$, $\sigma_g$, $\sigma_h$, and $\norm{\bff{u}_0}{\bb{H}^1}$ (but is independent of $n$ and $t$).
\end{proposition}

\begin{proof}
Using the second equation in~\eqref{equ:faedo} and H\"older's inequality,
\begin{align*}
	&\bb{E}\left[ \left(\int_0^t \norm{\nabla \Delta \bff{u}_n(s)}{\bb{L}^2}^2 \ds \right)^p \right]
	\\
	&\leq
	\bb{E} \left[\left(\int_0^t \norm{\nabla \bff{u}_n(s)}{\bb{L}^2}^2 \ds \right)^p \right]
	+
	\bb{E} \left[\left(\int_0^t \norm{\nabla \bff{H}_n(s)}{\bb{L}^2}^2 \ds \right)^p \right]
	+
	\bb{E} \left[\left(\int_0^t \norm{\nabla \big(|\bff{u}_n(s)|^2 \bff{u}_n(s)\big)}{\bb{L}^2}^2 \ds \right)^p \right]
	\\
	&\leq
	\bb{E}\left[\left(\int_0^t \norm{\nabla \bff{u}_n(s)}{\bb{L}^2}^2 \ds \right)^p \right]
	+
	\bb{E}\left[\left(\int_0^t \norm{\nabla \bff{H}_n(s)}{\bb{L}^2}^2 \ds \right)^p \right]
	+
	\bb{E} \left[\left(\int_0^t \norm{\bff{u}_n(s)}{\bb{L}^6}^4
	\norm{\nabla \bff{u}_n(s)}{\bb{L}^6}^2 \ds \right)^p \right]
	\\
	&\leq
	\bb{E} \left[\left(\int_0^t \norm{\nabla \bff{u}_n(s)}{\bb{L}^2}^2 \ds \right)^p \right]
	+
	\bb{E} \left[ \left(\int_0^t \norm{\nabla \bff{H}_n(s)}{\bb{L}^2}^2 \ds \right)^p \right]
	\\
	&\quad 
	+
	C\bb{E}\left[\left(\sup_{\tau\in [0,t]} \norm{\bff{u}_n(\tau)}{\bb{H}^1}^4\right)^p
	\left(\int_0^t 
	\norm{\bff{u}_n(s)}{\bb{H}^2}^2 \ds \right)^p \right]
	\\
	&\leq
	\bb{E} \left[\left(\int_0^t \norm{\nabla \bff{u}_n(s)}{\bb{L}^2}^2 \ds \right)^p \right]
	+
	\bb{E} \left[ \left(\int_0^t \norm{\nabla \bff{H}_n(s)}{\bb{L}^2}^2 \ds \right)^p \right]
	\\
	&\quad 
	+
	C \bb{E}\left[ \left(\sup_{\tau\in [0,t]} \norm{\bff{u}_n(\tau)}{\bb{H}^1}^4 \right)^{2p} \right]
	+
	C \bb{E}\left[ \left(\int_0^t 
	\norm{\bff{u}_n(s)}{\bb{H}^2}^2 \ds \right)^{2p} \right]
    \leq
    Ce^{Ct},
\end{align*}
where $C$ depends on $\norm{\bff{u_0}}{\bb{H}^1}$.
Here, in the second inequality we also used H\"older's inequality with exponents $3$ and $3/2$, while in the third inequality we used the Sobolev embedding $\bb{H}^1\subset \bb{L}^6$. In the last step we used Young's inequality, Proposition~\ref{pro:E sup un L2}, and \ref{pro:E sup nab un L2}. This shows the required estimate.
\end{proof}

\begin{proposition}\label{pro:E sup un H2}
	For any $p\geq 1$, $n\in \bb{N}$, $t\in (0,\infty)$,
	\begin{align}\label{equ:E sup un H2}
		\bb{E} \left[ \sup_{\tau \in [0,t]} \norm{\Delta \bff{u}_n(\tau)}{\bb{L}^2}^{2p} \right]
		+
		\bb{E} \left[\left( \int_0^t \norm{\nabla \Delta \bff{u}_n(s)}{\bb{L}^2}^2 \ds \right)^p \right]
        +
		\bb{E} \left[ \left( \int_0^t \norm{\Delta^2 \bff{u}_n(s)}{\bb{L}^2}^2 \ds \right)^p \right]
		\leq 
		C\left(1+e^{Ct}\right).
	\end{align}
    Furthermore,
    \begin{align}\label{equ:E sup Hn L2}
        \bb{E} \left[\sup_{\tau \in [0,t]} \norm{\bff{H}_n(\tau)}{\bb{L}^2}^{2p} \right]
        \leq 
        C\left(1+e^{Ct}\right).
    \end{align}
Here, $C$ is a positive constant depending on $p$,  $\sigma_g$, $\sigma_h$, and $\norm{\bff{u}_0}{\bb{H}^2}$ (but is independent of $n$ and $t$).
\end{proposition}

\begin{proof}
Let $\psi:\bb{V}_n\to \bb{R}$ be a function defined by $\bff{v}\mapsto \frac{1}{2} \norm{\Delta \bff{v}}{\bb{L}^2}^2$. Similarly to the proof of Proposition~\ref{pro:E sup un L2}, by It\^o's lemma,
\begin{align}\label{equ:d Delta un L2}
	\frac{1}{2} \dif \norm{\Delta \bff{u}_n(t)}{\bb{L}^2}^2
	=
	\left( \inpro{F_n(\bff{u}_n,\bff{H}_n)}{\Delta^2 \bff{u}_n}_{\bb{L}^2}
	+
	\frac{1}{2} \sum_{k=1}^n \norm{\Delta G_k(\bff{u}_n)}{\bb{L}^2}^2
	\right) \dt 
	+
	\sum_{k=1}^n \inpro{G_k(\bff{u}_n)}{\Delta^2 \bff{u}_n}_{\bb{L}^2}  \dif W_k(t). 
\end{align}
By~\eqref{equ:Fn} and \eqref{equ:Gn}, noting that~$\bff{u}_n,\bff{H}_n\in \bb{V}_n$ so that $\partial \bff{u}_n/\partial \bff{n}= \partial (\Delta \bff{u}_n)/\partial \bff{n} = \partial \bff{H}_n/\partial \bff{n}= \bff{0}$ on $\partial\mathscr{D}$ and all boundary terms vanish, integrating by parts as necessary we have
\begin{align}\label{equ:Fn un Hn Delta2}
	\nonumber
	&\inpro{F_n(\bff{u}_n,\bff{H}_n)}{\Delta^2 \bff{u}_n}_{\bb{L}^2}
	\\
	\nonumber
	&=
	\lambda_r \inpro{\bff{H}_n}{\Delta^2 \bff{u}_n}_{\bb{L}^2}
	-
	\lambda_e \inpro{\Delta \bff{H}_n}{\Delta^2 \bff{u}_n}_{\bb{L}^2}
    -
	\gamma \inpro{\bff{u}_n \times \bff{H}_n}{\Delta^2 \bff{u}_n}_{\bb{L}^2}
	+
    \inpro{S(\bff{u}_n)}{\bff{H}_n}_{\bb{L}^2}
	\\
	\nonumber
	&=
	\lambda_r \inpro{\bff{H}_n}{\Delta^2 \bff{u}_n}_{\bb{L}^2}
	-
	\lambda_e \norm{\Delta^2 \bff{u}_n}{\bb{L}^2}^2
	+
	\lambda_e \norm{\nabla \Delta \bff{u}_n}{\bb{L}^2}^2
	+
	\lambda_e \inpro{\Delta(|\bff{u}_n|^2 \bff{u}_n)}{\Delta^2 \bff{u}_n}_{\bb{L}^2}
    \\
	&\quad 
	-
	\gamma \inpro{\bff{u}_n \times \bff{H}_n}{\Delta^2 \bff{u}_n}_{\bb{L}^2}
	+
	\inpro{R(\bff{u}_n)}{\Delta^2 \bff{u}_n}_{\bb{L}^2}
    +
	\inpro{L(\bff{u}_n)}{\Delta^2 \bff{u}_n}_{\bb{L}^2}
\end{align}
and
\begin{align}\label{equ:Gn un Delta2 un}
	\inpro{G_k(\bff{u}_n)}{\Delta^2 \bff{u}_n}_{\bb{L}^2}
	&=
	\inpro{\bff{g}_k + \gamma \bff{u}_n\times \bff{h}_k}{\Delta^2 \bff{u}_n}_{\bb{L}^2}.
\end{align}
Substituting \eqref{equ:Fn un Hn Delta2} and \eqref{equ:Gn un Delta2 un} into \eqref{equ:d Delta un L2} yields
\begin{align}\label{equ:Delta un L2 exp}
	\nonumber
	&\frac{1}{2} \norm{\Delta \bff{u}_n(t)}{\bb{L}^2}^2
	+
	\lambda_e \int_0^t \norm{\Delta^2 \bff{u}_n(s)}{\bb{L}^2}^2 \ds
	\\
	\nonumber
	&=
	\frac{1}{2} \norm{\Delta \bff{u}_0}{\bb{L}^2}^2
    +
    \frac{1}{2} \sum_{k=1}^n \int_0^t \norm{\Delta G_k(\bff{u}_n)}{\bb{L}^2}^2 \ds
	+
	\lambda_r \int_0^t \inpro{\bff{H}_n}{\Delta^2 \bff{u}_n}_{\bb{L}^2} \ds
	+
	\lambda_e \int_0^t \norm{\nabla \Delta \bff{u}_n}{\bb{L}^2}^2 \ds
    \\
	\nonumber
	&\quad 
	+
	\lambda_e \int_0^t \inpro{\Delta(|\bff{u}_n|^2 \bff{u}_n)}{\Delta^2 \bff{u}_n}_{\bb{L}^2} \ds
	-
	\gamma \int_0^t \inpro{\bff{u}_n \times \bff{H}_n}{\Delta^2 \bff{u}_n}_{\bb{L}^2} \ds 
    +
	\int_0^t \inpro{R(\bff{u}_n)}{\Delta^2 \bff{u}_n}_{\bb{L}^2} \ds
    \\
    \nonumber
	&\quad 
    +
	\int_0^t \inpro{L(\bff{u}_n)}{\Delta^2 \bff{u}_n}_{\bb{L}^2} \ds
	+
	\sum_{k=1}^n \int_0^t \inpro{\bff{g}_k + \gamma \bff{u}_n\times \bff{h}_k}{\Delta^2 \bff{u}_n}_{\bb{L}^2} \dif W_k(s)
    \\
    &=:
    \frac{1}{2} \norm{\Delta \bff{u}_0}{\bb{L}^2}^2 + I_2(t) + I_3(t) + I_4(t) + I_5(t)
    + I_6(t) + I_7(t) + I_8(t) + I_9(t).
\end{align}
We shall bound each term appearing on the right-hand side of~\eqref{equ:Delta un L2 exp}.
Firstly, for the term $I_2(t)$, by \eqref{equ:prod Hs mat dot},
\begin{align*}
    |I_2(t)|
    \leq
    C \sum_{k=1}^n \int_0^t \left( \norm{\bff{g}_k}{\bb{H}^2}^2
    + \gamma \norm{\bff{u}_n}{\bb{H}^2}^2 \norm{\bff{h}_k}{\bb{H}^2}^2 \right) \ds
    \leq
    C \left( t + \int_0^t \norm{\bff{u}_n}{\bb{H}^2}^2 \ds \right).
\end{align*}
Secondly, for the term $I_3(t)$, by Young's inequality,
\begin{align*}
    |I_3(t)|
    \leq
    C\int_0^t \norm{\bff{H}_n}{\bb{L}^2}^2 \ds
    +
    \epsilon \int_0^t \norm{\Delta^2 \bff{u}_n}{\bb{L}^2}^2 \ds
\end{align*}
for any $\epsilon>0$. The term $I_4(t)$ will be left as is. For the term $I_5(t)$, noting \eqref{equ:del v2v} and applying H\"older's and Young's inequalities, we have
\begin{align*}
    |I_5(t)|
    &\leq
    \lambda_e \int_0^t \left(\norm{\bff{u}_n}{\bb{L}^\infty}^2 \norm{\Delta \bff{u}_n}{\bb{L}^2} \norm{\Delta^2 \bff{u}_n}{\bb{L}^2}
	+ 
	\norm{\bff{u}_n}{\bb{L}^\infty} \norm{\nabla \bff{u}_n}{\bb{L}^4}^2 \norm{\Delta^2 \bff{u}_n}{\bb{L}^2} \right) \ds 
    \\
    &\leq 
    C \int_0^t \norm{\bff{u}_n}{\bb{L}^\infty}^4 \norm{\Delta \bff{u}_n}{\bb{L}^2}^2 \ds
	+
	\epsilon \int_0^t \norm{\Delta^2 \bff{u}_n}{\bb{L}^2}^2 \ds
	+
	C \int_0^t \norm{\bff{u}_n}{\bb{L}^\infty}^2 \norm{\nabla \bff{u}_n}{\bb{L}^4}^4 \ds
    \\
    &\leq
    C \int_0^t \norm{\bff{u}_n}{\bb{H}^1}^4 \norm{\bff{u}_n}{\bb{H}^3}^2 \ds
    +
    \epsilon \int_0^t \norm{\Delta^2 \bff{u}_n}{\bb{L}^2}^2 \ds
\end{align*}
for any $\epsilon>0$, where in the last step we assumed $d=3$ for brevity and used the Gagliardo--Nirenberg inequalities
\begin{align*}
    \norm{\bff{u}_n}{\bb{L}^\infty}^2 
    \leq C \norm{\bff{u}_n}{\bb{H}^1} \norm{\bff{u}_n}{\bb{H}^2}
    \quad
    \text{and}
    \quad
    \norm{\nabla \bff{u}_n}{\bb{L}^4}^4
    &\leq C \norm{\bff{u}_n}{\bb{H}^1} \norm{\bff{u}_n}{\bb{H}^2}^3,
\end{align*}
and the interpolation inequality
\begin{align*}
    \norm{\bff{u}_n}{\bb{H}^2}^2 
    \leq C \norm{\bff{u}_n}{\bb{H}^1} \norm{\bff{u}_n}{\bb{H}^3}.
\end{align*}
The case $d=1$ and $d=2$ can be handled similarly. Next, for the term $I_6(t)$, we applied Young's inequality and Sobolev embedding to obtain
\begin{align*}
    |I_6(t)|
    \leq
    C \int_0^t \norm{\bff{u}_n}{\bb{L}^4}^2 \norm{\bff{H}_n}{\bb{L}^4}^2 \ds
    +
    \epsilon \int_0^t \norm{\Delta^2 \bff{u}_n}{\bb{L}^2}^2 \ds 
    \leq
    C \int_0^t \norm{\bff{u}_n}{\bb{L}^4}^2 \norm{\bff{H}_n}{\bb{H}^1}^2 \ds
    +
    \epsilon \int_0^t \norm{\Delta^2 \bff{u}_n}{\bb{L}^2}^2 \ds.
\end{align*}
For the term $I_7(t)$, we used \eqref{equ:Rv w} to obtain
\begin{align*}
    |I_7(t)|
    \leq
    C \int_0^t \norm{\bff{\nu}}{\bb{L}^\infty(\mathscr{D};\bb{R}^d)}^2
    \left( \norm{\nabla \bff{u}_n}{\bb{L}^2}^2 + \norm{|\bff{u}_n|\, |\nabla \bff{u}_n|}{\bb{L}^2}^2 \right) \ds
    +
    \epsilon \int_0^t \norm{\Delta^2 \bff{u}_n}{\bb{L}^2}^2 \ds.
\end{align*}
For the term $I_8(t)$, by the assumption on $L$ and Young's inequality,
\begin{align*}
    |I_8(t)|
    \leq
    C \int_0^t \left(1+ \norm{\bff{u}_n}{\bb{L}^2}^2 \right) \ds
    +
    \epsilon \int_0^t \norm{\Delta^2 \bff{u}_n}{\bb{L}^2}^2 \ds.
\end{align*}
For the moment, we keep the term $I_9(t)$ as is. Altogether, substituting these bounds into \eqref{equ:Delta un L2 exp} implies
\begin{align*}
	\nonumber
	&\frac{1}{2} \norm{\Delta \bff{u}_n(t)}{\bb{L}^2}^2
	+
	\lambda_e \int_0^t \norm{\Delta^2 \bff{u}_n(s)}{\bb{L}^2}^2 \ds
	\\
	\nonumber
	&\lesssim_{t} 
    1+ \norm{\Delta \bff{u}_0}{\bb{L}^2}^2
    +
    \epsilon \int_0^t \norm{\Delta^2 \bff{u}_n}{\bb{L}^2}^2 \ds
    +
    \int_0^t \norm{\bff{u}_n}{\bb{H}^3}^2 \ds
    +
    \int_0^t \norm{\bff{H}_n}{\bb{L}^2}^2 \ds
    +
    \int_0^t \norm{\bff{u}_n}{\bb{H}^1}^4 \norm{\bff{u}_n}{\bb{H}^3}^2 \ds
    \\
    &\quad 
    +
    \int_0^t \norm{\bff{u}_n}{\bb{L}^4}^2 \norm{\bff{H}_n}{\bb{H}^1}^2 \ds
    +
    \int_0^t \norm{|\bff{u}_n|\, |\nabla \bff{u}_n|}{\bb{L}^2}^2 \ds
	+
	\sum_{k=1}^n \int_0^t \inpro{\bff{g}_k + \gamma \bff{u}_n\times \bff{h}_k}{\Delta^2 \bff{u}_n}_{\bb{L}^2} \dif W_k(s)
\end{align*} 
for any $\epsilon>0$. Choosing $\epsilon>0$ sufficiently small, rearranging the terms, and applying Jensen's inequality, we obtain for $p\geq 1$,
\begin{align}\label{equ:sup Delta un L2}
	\nonumber
	&\sup_{\tau\in [0,t]} \norm{\Delta \bff{u}_n(\tau)}{\bb{L}^2}^{2p}
	+
	\left( \int_0^t \norm{\Delta^2 \bff{u}_n(s)}{\bb{L}^2}^2 \ds \right)^p
	\\
	\nonumber
	&\lesssim_t
    1
    +
    \norm{\Delta \bff{u}_0}{\bb{L}^2}^{2p}
    +
    \left(\int_0^t \norm{\bff{H}_n(s)}{\bb{L}^2}^2 \ds \right)^p
    +
    \left(1+\sup_{\tau\in [0,t]} \norm{\bff{u}_n(\tau)}{\bb{H}^1}^4\right)^p
	\left(\int_0^t \norm{\bff{u}_n(s)}{\bb{H}^3}^2 \ds \right)^p
    \\
    \nonumber
    &\quad 
    +
    \left(\sup_{\tau\in [0,t]} \norm{\bff{u}_n(\tau)}{\bb{L}^4}^2\right)^p
	\left(\int_0^t \norm{\bff{H}_n(s)}{\bb{H}^1}^2 \ds \right)^p
    +
    \left( \int_0^t \norm{|\bff{u}_n|\, |\nabla \bff{u}_n|}{\bb{L}^2}^2 \ds \right)^p
    \\
    &\quad 
	+
	\sup_{\tau\in [0,t]} \left|\sum_{k=1}^n \int_0^t \inpro{\bff{g}_k + \gamma \bff{u}_n\times \bff{h}_k}{\Delta^2 \bff{u}_n}_{\bb{L}^2} \dif W_k(s) \right|^p.
\end{align} 
We will estimate the last term on the right-hand side of~\eqref{equ:sup Delta un L2}. By the Burkholder--Davis--Gundy inequality and the H\"older inequality, we have
\begin{align}\label{equ:sup Delta2 un dW}
	\nonumber
	&\bb{E} \left[\sup_{\tau\in [0,t]} \left|\sum_{k=1}^n \int_0^\tau \inpro{\bff{g}_k + \gamma \bff{u}_n\times \bff{h}_k}{\Delta^2 \bff{u}_n}_{\bb{L}^2} \dif W_k(s) \right|^p \right]
	\\
	\nonumber
	&\leq 
	C_p \, \bb{E} \left[ \left(\sum_{k=1}^n \int_0^t \big| \inpro{ \bff{g}_k +  \gamma \bff{u}_n\times \bff{h}_k}{\Delta^2 \bff{u}_n}_{\bb{L}^2} \big|^2 \ds \right)^{p/2} \right]
	\\
	\nonumber
	&\leq
	C \left(\sum_{k=1}^n \left(\norm{\bff{g}_k}{\bb{L}^\infty}^2 +\norm{\bff{h}_k}{\bb{L}^\infty}^2 \right)\right)^{p/2}
    \bb{E} \left[\left( \int_0^t 
	\big(1+ \norm{\bff{u}_n(s)}{\bb{L}^2}^2 \big)
	\norm{\Delta^2 \bff{u}_n(s)}{\bb{L}^2}^2 \ds \right)^{p/2} \right]
	\\
	\nonumber
	&\leq
	C \bb{E} \left[ \left(1+ \sup_{\tau\in [0,t]} \norm{\bff{u}_n(\tau)}{\bb{L}^2}^2\right)^{p/2} \left(\int_0^t \norm{\Delta^2 \bff{u}_n(s)}{\bb{L}^2}^2 \ds \right)^{p/2} \right]
	\\
	\nonumber
	&\leq
	C\bb{E} \left[ \left(1+ \sup_{\tau\in [0,t]} \norm{\bff{u}_n(\tau)}{\bb{L}^2}^2 \right)^p \right] 
	+
	\epsilon \bb{E} \left[ \left(\int_0^t 
	\norm{\Delta^2 \bff{u}_n(s)}{\bb{L}^2}^2 \ds \right)^p \right]
	\\
	&\lesssim_{t}
	1 + \epsilon \bb{E} \left[ \left(\int_0^t \norm{\Delta^2 \bff{u}_n(s)}{\bb{L}^2}^2 \ds \right)^p \right]
\end{align}
for any $\epsilon>0$, where in the penultimate step we used Young's inequality, while in the last step we used Proposition~\ref{pro:E sup un L2}. Taking expectation on both sides of~\eqref{equ:sup Delta un L2}, substituting~\eqref{equ:sup Delta2 un dW} into the resulting expression, and choosing $\epsilon>0$ sufficiently small, we obtain
\begin{align*}
	&\bb{E} \left[ \sup_{\tau\in [0,t]} \norm{\Delta \bff{u}_n(\tau)}{\bb{L}^2}^{2p} \right]
	+
	\bb{E} \left[ \left( \int_0^t \norm{\Delta^2 \bff{u}_n(s)}{\bb{L}^2}^2 \ds \right)^p \right]
	\\
	&\lesssim_t
    1 + \norm{\Delta \bff{u}_0}{\bb{L}^2}^{2p}
    +
    \bb{E} \left[\left(\int_0^t \norm{\bff{H}_n}{\bb{L}^2}^2 \ds \right)^p \right]
    +
    \bb{E} \left[ \left(1+ \sup_{\tau\in [0,t]} \norm{\bff{u}_n(\tau)}{\bb{H}^1}^4\right)^p
	\left(\int_0^t \norm{\bff{u}_n(s)}{\bb{H}^3}^2 \ds \right)^p \right]
    \\
    \nonumber
    &\quad 
    +
    \bb{E} \left[ \left(\sup_{\tau\in [0,t]} \norm{\bff{u}_n(\tau)}{\bb{L}^4}^2\right)^p
	\left(\int_0^t \norm{\bff{H}_n(s)}{\bb{H}^1}^2 \ds \right)^p \right]
    +
    \bb{E} \left[ \left(\int_0^t \norm{|\bff{u}_n|\, |\nabla \bff{u}_n|}{\bb{L}^2}^2 \ds \right)^p \right]
    \\
    &\lesssim_t
    1 + \norm{\Delta \bff{u}_0}{\bb{L}^2}^{2p}
    +
    \bb{E} \left[\left(\int_0^t \norm{\bff{u}_n(s)}{\bb{H}^3}^2 \ds \right)^p \right]
    +
    \bb{E} \left[\left(\int_0^t \norm{\bff{H}_n(s)}{\bb{L}^2}^2 \ds \right)^p \right]
    \\
    &\quad 
    +
    \bb{E} \left[ \left(\sup_{\tau\in [0,t]} \norm{\bff{u}_n(\tau)}{\bb{H}^1}^4\right)^{2p} \right]
    +
	\bb{E} \left[ \left(\int_0^t \norm{\bff{u}_n(s)}{\bb{H}^3}^2 \ds \right)^{2p} \right]
    \\
    \nonumber
    &\quad 
    +
    \bb{E} \left[ \left(\sup_{\tau\in [0,t]} \norm{\bff{u}_n(\tau)}{\bb{L}^4}^2\right)^{2p} \right]
    +
	\bb{E} \left[\left(\int_0^t \norm{\bff{H}_n(s)}{\bb{H}^1}^2 \ds \right)^{2p} \right]
    +
    \bb{E} \left[ \left(\int_0^t \norm{|\bff{u}_n|\, |\nabla \bff{u}_n|}{\bb{L}^2}^2 \ds \right)^p \right]
	\lesssim_t
	1,
\end{align*}
where in the penultimate step we used Young's inequality and in the final step we used Proposition~\ref{pro:E sup un L2}, \ref{pro:E sup nab un L2}, and~\ref{pro:E nab Delta un}. This proves~\eqref{equ:E sup un H2}.

Finally, since $\bff{H}_n=\Delta \bff{u}_n +\bff{u}_n - \Pi_n(|\bff{u}_n|^2 \bff{u}_n)$, we obtain~\eqref{equ:E sup Hn L2} by using \eqref{equ:E sup un H2}, Proposition~\ref{pro:E sup nab un L2}, and the triangle inequality. This completes the proof of the proposition.
\end{proof}

\begin{proposition}\label{pro:E un Hn L2}
For any $p\geq 1$, $n\in \bb{N}$, $t\in (0,\infty)$,
\begin{align}
    \label{equ:E Delta Hn 1}
	\bb{E}\left[ \left( \int_0^t \norm{\Delta \bff{H}_n(s)}{\bb{L}^2}^2 \ds \right)^p \right]
	&\leq C(t),
	\\
	\label{equ:E u H H2}
	\bb{E}\left[ \left( \int_0^t \norm{\Pi_n\big(\bff{u}_n(s) \times \bff{H}_n(s) \big)}{\bb{H}^2}^2 \ds \right)^p \right] 
	&\leq 
	C(t),
	\\
	\label{equ:E S un L2}
	\bb{E}\left[ \left( \int_0^t \norm{\Pi_n S(\bff{u}_n)}{\bb{L}^2}^2 \ds \right)^p \right] 
	&\leq C(t),
\end{align}
where $C(t)$ is a positive constant depending on $p$, $t$, $\sigma_g$, $\sigma_h$, and $\norm{\bff{u}_0}{\bb{H}^2}$. 
\end{proposition}

\begin{proof}
By \eqref{equ:faedo} and \eqref{equ:prod Hs triple},
\begin{align*}
	\norm{\Delta \bff{H}_n}{\bb{L}^2}
	&\leq 
	\norm{\Delta^2 \bff{u}_n}{\bb{L}^2}
	+
	\norm{\Delta \bff{u}_n}{\bb{L}^2}
	+
	\norm{\Delta(|\bff{u}_n|^2 \bff{u}_n)}{\bb{L}^2}
	\\
	&\leq
	\norm{\Delta^2 \bff{u}_n}{\bb{L}^2}
	+
	\norm{\Delta \bff{u}_n}{\bb{L}^2}
	+
	C\norm{\bff{u}_n}{\bb{H}^2}^3.
\end{align*}
Similarly, by H\"older's and Young's inequalities,
\begin{align*}
	\bb{E}\left[ \left( \int_0^t \norm{\Delta \bff{H}_n(s)}{\bb{L}^2}^2 \ds \right)^p \right]
	&\leq
	\bb{E} \left[\left( \int_0^t \norm{\Delta^2 \bff{u}_n(s)}{\bb{L}^2}^2 \ds \right)^p \right] 
	+
	\bb{E} \left[\left( \int_0^t \norm{\Delta \bff{u}_n(s)}{\bb{L}^2}^2 \ds \right)^p \right] 
	\\
	&\quad 
	+
	\bb{E} \left[\left( \int_0^t \norm{ \bff{u}_n(s)}{\bb{H}^2}^6 \ds \right)^p \right] 
	\leq C(t),
\end{align*}
where in the last step we used Proposition~\ref{pro:E sup un L2}, \ref{pro:E sup nab un L2}, and~\ref{pro:E nab Delta un}. This proves~\eqref{equ:E Delta Hn 1}.

Next, by \eqref{equ:prod Hs mat dot} and Young's inequality,
\begin{align*}
    \norm{\Pi_n\big(\bff{u}_n \times \bff{H}_n\big)}{\bb{L}^2}
    &\leq
    \norm{\bff{u}_n\times \Delta \bff{u}_n}{\bb{H}^2}
    +
    \norm{\bff{u}_n \times \Pi_n(|\bff{u}_n|^2 \bff{u}_n)}{\bb{H}^2}
    \\
    &\leq
    C\norm{\bff{u}_n}{\bb{H}^2} \norm{\Delta \bff{u}_n}{\bb{H}^2}
    +
    C\norm{\bff{u}_n}{\bb{H}^2}^4
    \\
    &\leq
    C\norm{\bff{u}_n}{\bb{H}^2}^2 
    +
    C\norm{\Delta \bff{u}_n}{\bb{H}^2}^2
    +
    C\norm{\bff{u}_n}{\bb{H}^2}^4.
\end{align*}
A similar argument as before (noting~\eqref{equ:E sup un H2}) yields~\eqref{equ:E u H H2}.

Finally, by the definition of $R$ and $L$, and H\"older's inequality,
\begin{align*}
	\norm{\Pi_n S(\bff{u}_n)}{\bb{L}^2}^2
	\leq
	C \norm{\bff{\nu}}{\bb{L}^\infty(\mathscr{D};\bb{R}^d)}^2 \norm{\nabla \bff{u}_n}{\bb{L}^2}^2
    +
    C \norm{\bff{\nu}}{\bb{L}^\infty(\mathscr{D};\bb{R}^d)}^2 \norm{\bff{u}_n}{\bb{L}^\infty}^2 \norm{\nabla \bff{u}_n}{\bb{L}^2}^2
    +
    C \big(1+ \norm{\bff{u}_n}{\bb{L}^2}^2 \big).
\end{align*}
Therefore, taking expectation, then applying Young's inequality and Sobolev embedding,
\begin{align*}
    \bb{E}\left[ \left( \int_0^t \norm{\Pi_n S(\bff{u}_n)}{\bb{L}^2}^2 \ds \right)^p \right] 
    &\lesssim_t
    1+
    \bb{E} \left[ \left(\int_0^t \norm{\bff{u}_n(s)}{\bb{H}^1}^2\right)^p \right]
    +
    \bb{E} \left[ \left(\int_0^t \norm{\bff{u}_n(s)}{\bb{H}^2}^4\right)^p \right]
    \leq C(t),
\end{align*}
where we also used Proposition~\ref{pro:E sup un L2}, \ref{pro:E sup nab un L2}, and~\ref{pro:E nab Delta un}. This shows \eqref{equ:E S un L2}, thus completing the proof of the proposition.
\end{proof}

\section{Proof of Theorem~\ref{the:exist}: Existence of a Martingale Solution}

In this section, we prove the existence of martingale solution to \eqref{equ:sllbar} as stated in Theorem~\ref{the:exist}. First, note that equation \eqref{equ:faedo ito} can be written as
\begin{align}\label{equ:process}
	\bff{u}_n(t)
	&= 
	\bff{u}_n(0)
	+
	\int_0^t F_n(\bff{u}_n, \bff{H}_n) \,\ds 
	+
	\sum_{k=1}^n \int_0^t G_k(\bff{u}_n)\,\dif W_k(s)
    \nonumber \\
    &=:
    \bff{u}_n(0)
	+
	\int_0^t F_n(\bff{u}_n, \bff{H}_n) \,\ds 
	+
    B_n(\bff{u}_n, W)(t),
\end{align}
where $F_n$ and $G_k$ were defined in~\eqref{equ:Fn} and~\eqref{equ:Gn}, and~$\bff{H}_n=\Delta \bff{u}_n+\bff{u}_n- \Pi_n(|\bff{u}_n|^2 \bff{u}_n)$. Moreover, uniform bounds for $\bff{u}_n$ established in Section~\ref{sec:faedo} imply the following proposition.

\begin{proposition}\label{pro:E int Fn Gk}
Let $p\geq 2$, $q\geq 1$, and $\alpha\in \left(0,\frac{1}{2} \right)$ with $\alpha p>1$. There exists a constant $C$ such that for all $n\in \bb{N}$,
\begin{align}
	\label{equ:E H1 L2}
	\bb{E} \left[ \norm{\int_0^t F_n(\bff{u}_n, \bff{H}_n) \,\ds}{\bb{H}^1(0,T;\bb{L}^2)}^q \right] &\leq C,
	\\
	\label{equ:E Wap H2}
	\bb{E} \left[ \norm{B_n(\bff{u}_n,W)(t)}{\bb{W}^{\alpha,p}(0,T; \bb{L}^2)}^q \right] &\leq C,
\end{align}
where $C$ is a constant depending on $p$, $q$, $\alpha$, $T$, $\sigma_g$, $\sigma_h$, and $\norm{\bff{u}_0}{\bb{H}^2}$ (but is independent of $n$).
\end{proposition}

\begin{proof}
Inequality \eqref{equ:E H1 L2} follows immediately from Proposition~\ref{pro:E sup nab un L2} and \ref{pro:E un Hn L2}, while inequality~\eqref{equ:E Wap H2} is a consequence of Proposition~\ref{pro:E sup nab un L2}.
\end{proof}

Recalling the definition of the space $\bb{X}^{-\beta}$ as the dual of $\bb{X}^\beta$ defined in \eqref{equ:X beta}, we have the following result on tightness of laws.

\begin{proposition}
For any $p\in [1,\infty)$ and $\beta>0$, the set of laws $\{\mathcal{L}(\bff{u}_n): n\in \bb{N}\}$ on the Banach space
\begin{align}\label{equ:space Y}
	\bb{Y}:= L^p(0,T;\bb{W}^{1,4}) \cap L^2(0,T;\bb{H}^3) \cap C([0,T]; \bb{X}^{-\beta})
\end{align}
is tight.
\end{proposition}

\begin{proof}
Proposition \ref{pro:E sup un L2}, \ref{pro:E sup nab un L2}, \ref{pro:E sup un H2}, and \ref{pro:E int Fn Gk} show that for any $q\geq 1$,
\[
	\bb{E} \left[ \norm{\bff{u}_n}{L^p(0,T;\bb{H}^2) \,\cap\, L^2(0,T;\bb{H}^4) \,\cap\, W^{\alpha,p}(0,T;\bb{L}^2)}^q \right] \leq C.
\]
This and the following compact embeddings
\begin{align*}
	L^p(0,T;\bb{H}^2) \cap W^{\alpha,p}(0,T;\bb{L}^2) \hookrightarrow 
	L^p(0,T;\bb{W}^{1,4}) \cap C([0,T]; \bb{X}^{-\beta}),
	\\
	L^2(0,T;\bb{H}^4) \cap W^{\alpha,p}(0,T;\bb{L}^2) \hookrightarrow 
	L^2(0,T;\bb{H}^3) \cap C([0,T]; \bb{X}^{-\beta})
\end{align*}
imply the required result.
\end{proof}

By the above proposition, we can find a subsequence of $\{\bff{u}_n\}$, which is not relabelled, such that the laws $\mathcal{L}(\bff{u}_n, W)$ converge weakly to a probability measure $\mu$ on $\bb{Y} \times C([0,T]; \bb{R}^\infty)$. Noting that the space $\bb{Y} \times C([0,T]; \bb{R}^\infty)$ is separable, we then have the following proposition by the Skorohod theorem.

\begin{proposition}\label{pro:un conv Y}
Let $\bb{Y}$ be the space defined in \eqref{equ:space Y}. Then there exist
\begin{enumerate}
	\item a probability space $(\Omega', \mathcal{F}', \bb{P}')$,
	\item a sequence of random variables $\{(\bff{u}_n', W_n')\}$ defined on $(\Omega', \mathcal{F}', \bb{P}')$ and taking values in the space $\bb{Y} \times C([0,T]; \bb{R}^\infty)$,
	\item a random variable $(\bff{u}', W')$ defined on $(\Omega', \mathcal{F}', \bb{P}')$ and taking values in the space $\bb{Y} \times C([0,T]; \bb{R}^\infty)$,
\end{enumerate}
such that in the space $\bb{Y}\times C([0,T]; \bb{R}^\infty)$,
\begin{enumerate}
	\item the laws $\mathcal{L}(\bff{u}_n, W)= \mathcal{L}(\bff{u}_n', W_n')$ for all $n\in\bb{N}$,
	\item $(\bff{u}_n',W_n') \to (\bff{u}', W')$ strongly, $\bb{P}'$-a.s.
\end{enumerate}
\end{proposition}

\begin{remark}\label{rem:est Eu}
By the Kuratowski--Suslin theorem, the Borel subsets of $C([0,T]; \bb{V}_n)$ are Borel subsets of $\bb{Y}$. Moreover, $\bb{P}$-a.s., $\bff{u}_n \in C([0,T];\bb{V}_n)$. Therefore, we may assume that $\bff{u}_n'$ takes values in $\bb{V}_n$ and that the laws on $C([0,T]; \bb{V}_n)$ of $\bff{u}_n$ and $\bff{u}_n'$ are equal. It is then straightforward to show that the sequence $\{\bff{u}_n'\}$ satisfies the same estimates as the original sequence $\{\bff{u}_n\}$, namely for any $q\geq 1$,
\begin{align}\label{equ:E un dash H2}
	\sup_{n\in \bb{N}} \bb{E}' \left[ \sup_{t\in [0,T]} \norm{\bff{u}_n'(t)}{\bb{H}^2}^{2q} \right] 
	&< \infty,
	\\
	\label{equ:E int un dash H4}
	\sup_{n\in \bb{N}} \bb{E}' \left[ \left( \int_0^T \norm{\bff{u}_n'(s)}{\bb{H}^4}^2 \ds \right)^q \right] 
	&< \infty,
	\\
	\label{equ:E int Hn dash H2}
	\sup_{n\in \bb{N}} \bb{E}' \left[ \left( \int_0^T \norm{\bff{H}_n'(s)}{\bb{H}^2}^2 \ds \right)^q \right] 
	&< \infty,
    \\
    \label{equ:E int S un L2}
    \sup_{n\in\bb{N}} \bb{E}'\left[ \left( \int_0^t \norm{\Pi_n S(\bff{u}_n')}{\bb{L}^2}^2 \ds \right)^q \right] 
	&< \infty,
\end{align}
where $\bff{H}_n'= \Delta \bff{u}_n' + \bff{u}_n' - \Pi_n(|\bff{u}_n'|^2 \bff{u}_n')$.
\end{remark}

Subsequently, we will work solely with processes defined on the probability space $(\Omega', \mathcal{F}', \mathbb{F}', \bb{P}')$. To simplify notations, we will write $(\Omega, \mathcal{F}, \mathbb{F}, \bb{P})$ instead. The new processes $W_n'$, $W'$, and $\bff{u}_n'$ will also be written as $W_n$, $W$, and $\bff{u}_n$, respectively.
We remark that, as in~\cite{BrzGolJeg13}, the processes $W_n'$ and $W'$ are Wiener processes on $(\Omega', \mathcal{F}', \mathbb{F}', \bb{P}')$.

Now, define a sequence of $\bb{L}^2$-valued processes
\begin{equation}\label{equ:Mn}
    \bff{M}_n(t):= \bff{u}_n(t)-\bff{u}_n(0)-\int_0^t F_n(\bff{u}_n,\bff{H}_n) \,\ds.
\end{equation}
Then for each $t\in [0,T]$, we have
\[
    \bff{M}_n(t)= B_n(\bff{u}_n, W_n)(t), \quad \bb{P}\text{-a.s.},
\]
where $B_n$ and $F_n$ are as in~\eqref{equ:process}. 

%

\begin{lemma}\label{lem:Hn H}
Let $\bff{H}_n:= \Delta \bff{u}_n+\bff{u}_n-\Pi_n(|\bff{u}_n|^2 \bff{u}_n)$ and $\bff{H}:= \Delta \bff{u}+\bff{u}-|\bff{u}|^2 \bff{u}$. Then $\bb{P}$-a.s.,
\begin{align*}
    \bff{H}_n\to \bff{H} \text{ strongly in } L^2(0,T;\bb{H}^1).
\end{align*}
\end{lemma}

\begin{proof}
Since $\bff{u}_n\to\bff{u}$ strongly in $L^2(0,T;\bb{H}^3)$, it is clear that $\Delta \bff{u}_n\to \Delta \bff{u}$ strongly in $L^2(0,T;\bb{H}^1)$. Moreover, by H\"older's inequality,
\begin{align*}
    \norm{|\bff{u}_n|^2 \bff{u}_n- |\bff{u}|^2 \bff{u}}{L^2(0,T;\bb{H}^1)}
    &\leq
    \norm{|\bff{u}_n|^2 (\bff{u}_n-\bff{u})}{L^2(0,T;\bb{H}^1)}
    +
    \norm{(\bff{u}_n-\bff{u})\cdot (\bff{u}_n+\bff{u})\bff{u}}{L^2(0,T;\bb{H}^1)}
    \\
    &\leq
    \norm{\bff{u}_n}{L^6(0,T;\bb{L}^6)}^2 \norm{\bff{u}_n-\bff{u}}{L^6(0,T;\bb{L}^6)}
    \\
    &\quad
    +
    2 \norm{\bff{u}_n}{L^8(0,T;\bb{L}^8)} \norm{\nabla \bff{u}_n}{L^4(0,T;\bb{L}^4)} \norm{\bff{u}_n-\bff{u}}{L^8(0,T;\bb{L}^8)}
    \\
    &\quad
    +
    \norm{\bff{u}_n-\bff{u}}{L^6(0,T;\bb{L}^6)}
    \norm{\bff{u}_n+\bff{u}}{L^6(0,T;\bb{L}^6)}
    \norm{\bff{u}}{L^6(0,T;\bb{L}^6)}
    \\
    &\quad
    +
    \norm{\nabla \bff{u}_n-\nabla \bff{u}}{L^4(0,T;\bb{L}^4)}
    \norm{\bff{u}_n+\bff{u}}{L^8(0,T;\bb{L}^8)}
    \norm{\bff{u}}{L^8(0,T;\bb{L}^8)}
    \\
    &\quad
    +
    \norm{\bff{u}_n-\bff{u}}{L^8(0,T;\bb{L}^8)}
    \norm{\nabla \bff{u}_n+\nabla \bff{u}}{L^4(0,T;\bb{L}^4)}
    \norm{\bff{u}}{L^8(0,T;\bb{L}^8)}
    \\
    &\quad
    +
    \norm{\bff{u}_n-\bff{u}}{L^8(0,T;\bb{L}^8)}
    \norm{\bff{u}_n+\bff{u}}{L^8(0,T;\bb{L}^8)}
    \norm{\nabla \bff{u}}{L^4(0,T;\bb{L}^4)}.
\end{align*}
Noting the embedding $\bb{W}^{1,4}\hookrightarrow \bb{L}^8$ and using Proposition~\ref{pro:un conv Y}, we have that $\Pi_n(|\bff{u}_n|^2 \bff{u}_n) \to |\bff{u}|^2 \bff{u}$ strongly, $\bb{P}$-a.s. in $L^2(0,T;\bb{H}^1)$. This implies $\bff{H}_n\to \bff{H}$ strongly in $L^2(0,T;\bb{H}^1)$, as required.
\end{proof}

\begin{lemma}
Let $\bff{u}_n$ and $\bff{u}$ be processes defined in Proposition~\ref{pro:un conv Y}. Then for any $q \geq 1$,
\begin{align}
    \label{equ:weak Linfty H2}
    &\bff{u}_n \rightharpoonup \bff{u} \text{ weakly$^\ast$ in } L^{2q}\big(\Omega; L^\infty(0,T; \bb{H}^2)\big),
    \\
    \label{equ:weak L2 H4}
    &\bff{u}_n \rightharpoonup \bff{u} \text{ weakly in } L^{2q}\big(\Omega; L^2(0,T; \bb{H}^4)\big).
\end{align}
In particular, $\bff{u}\in L^{2q}\big(\Omega; L^\infty(0,T; \bb{H}^2)\big) \cap L^{2q}\big(\Omega; L^2(0,T; \bb{H}^4)\big)$.
Furthermore, $\bff{u}\in L^{2q}\big(\Omega; C([0,T]; \bb{H}^2_\mathrm{w})\big)$.
\end{lemma}

\begin{proof}
First, we show \eqref{equ:weak Linfty H2}. The Banach--Alaoglu theorem (noting \eqref{equ:E un dash H2})  yields a subsequence of $\{\bff{u}_n\}$, which we do not relabel, and $\bff{v}\in L^{2q}\big(\Omega; L^\infty(0,T; \bb{H}^2) \cap L^2(0,T; \bb{H}^4)\big)$ such that
\begin{align*}
    &\bff{u}_n \rightharpoonup \bff{v} \text{ weakly$^\ast$ in } L^{2q}\big(\Omega; L^\infty(0,T; \bb{H}^2)\big).
\end{align*}
In other words, we have
\begin{equation}\label{equ:E un L2 2}
    \bb{E}\left[\int_0^T  \inpro{\bff{u}_n(t)}{\bff{\psi}(t)}_{\bb{L}^2} \dt\right] \to
    \bb{E}\left[\int_0^T \inpro{\bff{v}(t)}{\bff{\psi}(t)}_{\bb{L}^2} \dt\right], \quad \forall \bff{\psi}\in L^{\frac{2q}{2q-1}}\big(\Omega;L^1(0,T;\widetilde{\bb{H}}^{-2})\big).
\end{equation}

Next, we will show that in fact $\bff{v}=\bff{u}$. To this end, note that by a straightforward argument utilising Proposition~\ref{pro:un conv Y} and Remark~\ref{rem:est Eu}, for any $\bff{\phi}\in L^2\big(\Omega;L^2(0,T;\bb{L}^2)\big)$, the sequence $\left\{\int_0^T \inpro{\bff{u}_n(t)}{\bff{\phi}(t)}_{\bb{L}^2} \dt\right\}$ is uniformly integrable on $\Omega$. Therefore, by the almost sure convergence of $\bff{u}_n$ to $\bff{u}$ in $L^2(0,T;\bb{L}^2)$ and Vitali's convergence theorem, we have
\begin{equation}\label{equ:E un L2 1}
    \bb{E} \left[ \int_0^T \inpro{\bff{u}_n(t)}{\bff{\phi}(t)}_{\bb{L}^2} \dt \right] \to
    \bb{E} \left[ \int_0^T \inpro{\bff{u}(t)}{\bff{\phi}(t)}_{\bb{L}^2} \dt \right],
    \quad \forall \bff{\phi}\in L^2\big(\Omega;L^2(0,T;\bb{L}^2)\big).
\end{equation}
By the density of $L^2\big(\Omega;L^2(0,T;\bb{L}^2)\big)$ in $L^{\frac{2q}{2q-1}}\big(\Omega;L^1(0,T;\widetilde{\bb{H}}^{-2})\big)$, we conclude from~\eqref{equ:E un L2 2} and \eqref{equ:E un L2 1} that $\bff{u}=\bff{v}$. In particular, by $\bff{u}\in L^{2q}\big(\Omega; L^\infty(0,T; \bb{H}^2)\big)$ by the estimate \eqref{equ:E un dash H2}. By a similar argument and noting the estimate \eqref{equ:E int un dash H4}, we also obtain \eqref{equ:weak L2 H4} and that $\bff{u}\in L^{2q}\big(\Omega;  L^2(0,T; \bb{H}^4)\big)$.

Finally, Proposition~\ref{pro:un conv Y} implies $\bff{u}_n\in C\big([0,T];\bb{H}^2_{\mathrm{w}}\big)$. Since $\bff{u}_n\rightharpoonup \bff{u}$ weakly$^\ast$ in $L^{2q}\big(\Omega; L^\infty(0,T; \bb{H}^2)\big)$ and the space $C\big([0,T]; \bb{H}^2_{\mathrm{w}}\big)$ is complete, we deduce $\bff{u}\in L^{2q}\big(\Omega; C([0,T]; \bb{H}^2_\mathrm{w})\big)$. This completes the proof of the lemma.
\end{proof}

We now show the convergence of each term in \eqref{equ:process}.

\begin{lemma}\label{lem:limit E}
For any $\bff{\chi}\in \bb{H}^1$,
\begin{align*}
    \lim_{n\to\infty} \bb{E}\left[\int_0^t \inpro{\bff{H}_n(s)}{\bff{\chi}}_{\bb{L}^2} \ds\right]
    &=
    \bb{E} \left[\int_0^t \inpro{\bff{H}(s)}{\bff{\chi}}_{\bb{L}^2} \ds \right],
    \\
    \lim_{n\to\infty} \bb{E}\left[\int_0^t \inpro{\nabla \bff{H}_n(s)}{\nabla \bff{\chi}}_{\bb{L}^2} \ds\right]
    &=
    \bb{E} \left[\int_0^t \inpro{\nabla \bff{H}(s)}{\nabla \bff{\chi}}_{\bb{L}^2} \ds \right],
    \\
    \lim_{n\to\infty} \bb{E} \left[\int_0^t \inpro{\Pi_n \big(\bff{u}_n(s) \times \bff{H}_n(s)\big)}{\bff{\chi}}_{\bb{L}^2} \ds \right]
    &=
    \bb{E} \left[\int_0^t \inpro{\bff{u}(s) \times \bff{H}(s)}{\bff{\chi}}_{\bb{L}^2} \ds \right],
    \\
    \lim_{n\to\infty} \bb{E} \left[\int_0^t  \inpro{\Pi_n S\big(\bff{u}_n(s)\big)}{\bff{\chi}}_{\bb{L}^2} \ds\right]
    &=
    \bb{E} \left[\int_0^t \inpro{S\big(\bff{u}(s)\big)}{\bff{\chi}}_{\bb{L}^2} \ds \right].
\end{align*}
\end{lemma}

\begin{proof}
By the same argument as in~\cite[Proposition 4.1]{SoeTra23}, using Proposition~\ref{pro:un conv Y} and Lemma~\ref{lem:Hn H}, we have $\bb{P}$-a.s.
\begin{align*}
    \lim_{n\to\infty} \int_0^t \inpro{\bff{H}_n(s)}{\bff{\chi}}_{\bb{L}^2} \ds
    &=
    \int_0^t \inpro{\bff{H}(s)}{\bff{\chi}}_{\bb{L}^2} \ds,
    \\
    \lim_{n\to\infty} \int_0^t \inpro{\nabla \bff{H}_n(s)}{\nabla \bff{\chi}}_{\bb{L}^2} \ds
    &=
    \int_0^t \inpro{\nabla \bff{H}(s)}{\nabla \bff{\chi}}_{\bb{L}^2} \ds,
    \\
    \lim_{n\to\infty} \int_0^t \inpro{\Pi_n \big(\bff{u}_n(s) \times \bff{H}_n(s)\big)}{\bff{\chi}}_{\bb{L}^2} \ds
    &=
    \int_0^t \inpro{\bff{u}(s) \times \bff{H}(s)}{\bff{\chi}}_{\bb{L}^2} \ds,
    \\
    \lim_{n\to\infty} \int_0^t \inpro{\Pi_n S\big(\bff{u}_n(s)\big)}{\bff{\chi}}_{\bb{L}^2} \ds
    &=
    \int_0^t \inpro{S\big(\bff{u}(s)\big)}{\bff{\chi}}_{\bb{L}^2} \ds.
\end{align*}
Furthermore, by H\"older's inequality (noting~\eqref{equ:E un dash H2}, \eqref{equ:E int un dash H4}, \eqref{equ:E int Hn dash H2}, and \eqref{equ:E int S un L2}), we have
\begin{align*}
    \sup_{n\in\bb{N}} \bb{E}\left[\left|\int_0^t \inpro{\bff{H}_n(s)}{\bff{\chi}}_{\bb{L}^2} \ds \right|^2\right]
    &\leq
    \norm{\bff{H}_n}{L^4(\Omega;L^2(0,T;\bb{L}^2))}^2
    \norm{\bff{\chi}}{L^4(\Omega;L^2(0,T;\bb{L}^2))}^2
    < \infty,
    \\
    \sup_{n\in \bb{N}} \bb{E}\left[\left|\int_0^t \inpro{\nabla \bff{H}_n(s)}{\nabla \bff{\chi}}_{\bb{L}^2} \ds \right|^2\right]
    &\leq
    \norm{\nabla \bff{H}_n}{L^4(\Omega;L^2(0,T;\bb{L}^2))}^2
    \norm{\nabla \bff{\chi}}{L^4(\Omega;L^2(0,T;\bb{L}^2))}^2
    < \infty,
    \\
    \sup_{n\in \bb{N}} \bb{E} \left[\left|\int_0^t \inpro{\Pi_n \big(\bff{u}_n(s) \times \bff{H}_n(s)\big)}{\bff{\chi}}_{\bb{L}^2} \ds\right|^2\right]
    &\leq
    \norm{\bff{u}_n}{L^8(\Omega;L^4(0,T;\bb{L}^4))}^2
    \norm{\bff{H}_n}{L^4(\Omega;L^2(0,T;\bb{L}^2))}^2
    \norm{\bff{\chi}}{L^8(\Omega;L^4(0,T;\bb{L}^4))}^2
    \\
    &< \infty,
    \\
    \sup_{n\in \bb{N}} \bb{E} \left[\left|\int_0^t  \inpro{\Pi_n S\big(\bff{u}_n(s)\big)}{\bff{\chi}}_{\bb{L}^2} \ds\right|^2 \right]
    &\leq
    \norm{S\big(\bff{u}_n(s)\big)}{L^4(\Omega;L^2(0,T;\bb{L}^2))}^2
    \norm{\bff{\chi}}{L^4(\Omega;L^2(0,T;\bb{L}^2))}^2
    < \infty.
\end{align*}
The results then follow from Vitali's convergence theorem.
\end{proof}

\begin{lemma}\label{lem:Mn to M}
Let $\beta>0$. For every $t\in [0,T]$, the sequence of random variables $\bff{M}_n(t)$ defined in~\eqref{equ:Mn} converges weakly in $L^2(\Omega;\bb{X}^{-\beta})$ to a limit $\bff{M}(t)$ given by
\begin{align*}
    \bff{M}(t)
    &=
    \bff{u}(t)
    -
    \bff{u}(0)
    -
    \lambda_r \int_0^t \bff{H}(s)\,\ds 
    +
    \lambda_e \int_0^t \Delta \bff{H}(s)\,\ds
    +
    \gamma \int_0^t \bff{u}(s)\times \bff{H}(s)\,\ds
    \\
    &\quad
    -
    \int_0^t R(\bff{u}(s))\,\ds
    -
    \int_0^t L(\bff{u}(s))\,\ds,
\end{align*}
where $\bff{H}(s)= \Delta \bff{u}(s)+ \bff{u}(s)- |\bff{u}(s)|^2  \bff{u}(s)$.
\end{lemma}

\begin{proof}
Let $t\in [0,T]$ and $\bff{\phi}\in L^2(\Omega;\bb{X}^\beta)$. Since $\bff{u}_n$ converges to $\bff{u}$ in $C([0,T];\bb{X}^{-\beta})$, $\bb{P}$-a.s., we infer that
\[
    \lim_{n\to\infty} {}_{\bb{X}^{-\beta}}\inpro{\bff{u}_n(t)}{\bff{\phi}}_{\bb{X}^\beta}
    =
    {}_{\bb{X}^{-\beta}} \inpro{\bff{u}(t)}{\bff{\phi}}_{\bb{X}^\beta},
    \quad \bb{P}\text{-a.s.}.
\]
Furthermore, by the embedding $\bb{H}^1\hookrightarrow \bb{X}^{-\beta}$ and~\eqref{equ:E un dash H2}, we have
\begin{equation*}
    \sup_{n\in \bb{N}} \bb{E} \left[\big|{}_{\bb{X}^{-\beta}} \inpro{\bff{u}_n(t)}{\bff{\phi}}_{\bb{X}^\beta} \big|^2\right]
    \leq
    \sup_{n\in \bb{N}} \left(\bb{E} \left[\norm{\bff{u}_n(t)}{\bb{X}^{-\beta}}^4\right]\right)^{\frac{1}{2}}
    \left(\bb{E} \left[\norm{\bff{\phi}}{\bb{X}^{\beta}}^4\right]\right)^{\frac{1}{2}}
    < \infty.
\end{equation*}
Therefore, by Vitali's convergence theorem,
\begin{equation*}
    \lim_{n\to\infty} \bb{E} \big[{}_{\bb{X}^{-\beta}} \inpro{\bff{u}_n(t)}{\bff{\phi}}_{\bb{X}^\beta}\big]
    =
    \bb{E} \big[ {}_{\bb{X}^{-\beta}} \inpro{\bff{u}(t)}{\bff{\phi}}_{\bb{X}^\beta} \big].
\end{equation*}
This, together with Lemma~\ref{lem:limit E}, imply the required result.
\end{proof}

\begin{lemma}\label{lem:dWk to dW}
Let $\beta>0$, and let $\bff{u}_n$ and $\bff{u}$ be processes defined in Proposition~\ref{pro:un conv Y}. Then for every $t\in[0,T]$,
\begin{equation*}
    \lim_{n\to\infty} \norm{\sum_{k=1}^n \left(\int_0^t G_k\big(\bff{u}_n(s)\big) \,\mathrm{d}W_{k,n}(s) - \int_0^t G_k\big(\bff{u}(s)\big)\, \mathrm{d}W_k(s)\right)}{\bb{X}^{-\beta}} = 0.
\end{equation*}
\end{lemma}

\begin{proof}
The proof of this lemma is similar to that of~\cite[Lemma~5.2]{BrzGolJeg13} and is omitted.
\end{proof}

We can now prove the first main theorem of the paper (Theorem~\ref{the:exist}).

\begin{proof}[Proof of Theorem~\ref{the:exist}]
By equation~\eqref{equ:Mn}, Lemma~\ref{lem:Mn to M}, and Lemma~\ref{lem:dWk to dW}, we deduce that for every $t\in[0,T]$,
\[
    \bff{M}(t)= \sum_{k=1}^\infty \int_0^t G_k(\bff{u})\, \mathrm{d}W_k(s) \quad \text{in } L^2(\Omega;\bb{X}^{-\beta}).
\]
This implies that $\bb{P}$-a.s., $\bff{u}\in C([0,T]; \bb{H}^2_{\mathrm{w}}) \cap L^2(0,T;\bb{H}^3)$, and $(\bff{u},\bff{W})$ satisfies equation~\eqref{equ:weakform}. The proof that $\bff{u}\in C([0,T];\bb{H}^2)$ and~\eqref{equ:u holder reg} is deferred to Lemma~\ref{lem:H3} below.
\end{proof}

The proof of Theorem~\ref{the:exist} will be complete once we show the following parabolic regularisation estimate.
Let $e^{-tA}$ denote the analytic semigroup generated by $A$, where $A$ was defined in~\eqref{equ:op A}.
Given~$\bff{u}_0\in \text{D}(A^{1/2})$, we can write the solution (see~\cite{DaZab14}) as:
\begin{align}\label{equ:u mild}
    \nonumber
    \bff{u}(t)
    &= 
    e^{-tA}\bff{u}_0
    +
    (\lambda_r+\beta) \int_0^t e^{-(t-s)A}\bff{u}(s) \,\ds
    -
    \lambda_r \int_0^t e^{-(t-s)A} \big(|\bff{u}(s)|^2 \bff{u}(s)\big) \,\ds
    \\
    \nonumber
    &\quad
    +
    \lambda_e \int_0^t e^{-(t-s)A} \Delta\big(|\bff{u}(s)|^2 \bff{u}(s)\big) \,\ds
    -
    \gamma \int_0^t e^{-(t-s)A}\big(\bff{u}(s) \times \Delta \bff{u}(s) \big) \,\ds
    +
    \int_0^t e^{-(t-s)A} S\big(\bff{u}(s)\big)\, \ds
    \\
    &\quad
    +
    \sum_{k=1}^\infty \int_0^t e^{-(t-s)A} G_k(\bff{u}(s))\, \dif W_k(s)
    \nonumber \\
    &=:
    I_0(t)+I_1(t)+\ldots +I_6(t).
\end{align}

\begin{lemma}\label{lem:H3}
Let $\bff{u}_0\in \text{D}(A^{1/2})$, $\beta\in [\frac12,1)$, and $\delta\in (0,1-\beta)$. Then $\bb{P}$-a.s., we have $\bff{v}\in C^\delta([0,T];\text{D}(A^\beta))$.
In particular, $\bff{u} \in C([0,T];\bb{H}^2)$. 
\end{lemma}

\begin{proof}
Let $p>2$ and $\delta>0$ be numbers to be specified later. We aim to estimate $\bb{E}\big[\norm{\bff{v}(t+h)-\bff{v}(t)}{\text{D}(A^\beta)}^p\big]$ for any $h>0$.
To this end, it suffices to bound 
\[
\bb{E}\left[\norm{A^{\beta}I_j(t+h)-A^{\beta} I_j(t)}{\bb{L}^2}^p\right]
\]
for $j=1,2,\ldots,6$ in the following, where $I_j$ are defined in~\eqref{equ:u mild}. 
For the first term, noting Theorem~\ref{the:exist}, $\bb{P}$-a.s. we have
\begin{align}\label{equ:I1 th}
    \bb{E}\left[\norm{A^{\beta}I_1(t+h)-A^{\beta} I_1(t)}{\bb{L}^2}^p \right]
    &\leq
    \bb{E} \left[\norm{\int_0^t A^{\beta-\frac12 +\delta}e^{-(t-s)A}A^{-\delta}(e^{-hA}-I)A^{\frac12} \bff{u}(s)\,\ds}{\bb{L}^2}^p\right]
    \nonumber \\
    &\quad
    +
    \bb{E} \left[\norm{\int_t^{t+h} A^{\beta-\frac12} e^{-(t-s)A} e^{-hA} A^{\frac12} \bff{u}(s)\,\ds}{\bb{L}^2}^p \right]
    \nonumber \\
    &\leq
    Ch^{\delta p} \,\bb{E} \left[ \left(\int_0^t (t-s)^{-(\beta-\frac12)-\delta} \norm{A^{\frac12} \bff{u}(s)}{\bb{L}^2} \ds \right)^p \right]
    \nonumber \\
    &\quad
    +
    C\, \bb{E} \left[ \left(\int_t^{t+h} (t+h-s)^{-(\beta-\frac12)} \norm{A^{\frac12} \bff{u}(s)}{\bb{L}^2} \ds \right)^p\right]
    \nonumber \\
    &\leq
    C\, \bb{E} \left[\sup_{\tau\in [0,T]} \norm{\bff{u}(\tau)}{\bb{H}^2}^p \right] 
    \left(h^{\delta p} t^{p(\frac32-\beta-\delta)}
    +
    h^{p(\frac32-\beta)}\right).
\end{align}
Similarly for the second term, using~\eqref{equ:del v2v} and~\eqref{equ:prod Hs mat dot}, $\bb{P}$-a.s. we have
\begin{align}\label{equ:I2 th}
    \bb{E}\left[\norm{A^{\beta}I_2(t+h)-A^{\beta} I_2(t)}{\bb{L}^2}^p\right]
    &\leq
    \bb{E}\left[\norm{\int_0^t A^{\beta-\frac12+\delta}e^{-(t-s)A} A^{-\delta}(e^{-hA}-I) A^{\frac12}\big(|\bff{u}(s)|^2 \bff{u}(s)\big)\,\ds}{\bb{L}^2}^p \right] 
    \nonumber \\
    &\quad
    + 
    \bb{E} \left[\norm{\int_t^{t+h} A^{\beta-\frac12} e^{-(t-s)A} e^{-hA} A^{\frac12}\big(|\bff{u}(s)|^2 \bff{u}(s)\big)\,\ds}{\bb{L}^2}^p \right]
    \nonumber \\
    &\leq
    Ch^{\delta p} \, \bb{E} \left[ \left(\int_0^t (t-s)^{-(\beta-\frac12)-\delta} \norm{|\bff{u}(s)|^2 \bff{u}(s)}{\bb{H}^2} \ds \right)^p \right] 
    \nonumber \\
    &\quad
    +
    C\, \bb{E} \left[ \left(\int_t^{t+h} (t+h-s)^{-(\beta-\frac12)} \norm{|\bff{u}(s)|^2 \bff{u}(s)}{\bb{H}^2} \ds\right)^p \right]
    \nonumber \\
    &\leq
    C\, \bb{E} \left[\sup_{\tau\in [0,T]} \norm{\bff{u}(\tau)}{\bb{H}^2}^{3p} \right] 
    \left(h^{\delta p} t^{p(\frac32-\beta-\delta)}
    +
    h^{p(\frac32-\beta)}\right).
\end{align}
For the third term, using~\eqref{equ:del v2v}, \eqref{equ:prod Hs mat dot} and interpolation inequalities, we have
\begin{align}\label{equ:I3 th}
    &\bb{E}\left[\norm{A^{\beta}I_3(t+h)-A^{\beta} I_3(t)}{\bb{L}^2}^p\right]
    \nonumber \\
    &\leq
    \bb{E} \left[\norm{\int_0^t A^{\beta-\frac12+\delta} e^{-(t-s)A} A^{-\delta} (e^{-hA}-I)A^{\frac12} \Delta\big(|\bff{u}(s)|^2 \bff{u}(s)\big)\,\ds}{\bb{L}^2}^p \right]
    \nonumber \\
    &\quad
    +
    \bb{E}\left[\norm{\int_t^{t+h} A^{\beta-\frac12} e^{-(t-s)A} e^{-hA} A^{\frac12} \Delta\big(|\bff{u}(s)|^2 \bff{u}(s)\big)\,\ds}{\bb{L}^2}^p \right]
    \nonumber \\
    &\leq
    Ch^{\delta p}\, \bb{E} \left[\left( \int_0^t (t-s)^{-(\beta-\frac12)-\delta} \norm{\Delta\big(|\bff{u}(s)|^2 \bff{u}(s)\big)}{\bb{H}^2} \ds\right)^p \right]
    \nonumber \\
    &\quad
    +
    C\, \bb{E} \left[\left(\int_t^{t+h} (t+h-s)^{-(\beta-\frac12)} \norm{\Delta\big(|\bff{u}(s)|^2 \bff{u}(s)\big)}{\bb{H}^2} \ds\right)^p \right]
    \nonumber \\
    &\leq
    Ch^{\delta p} \, \bb{E} \left[\left(\int_0^t (t-s)^{-(\beta-\frac12)-\delta} \norm{\bff{u}(s)}{\bb{H}^2}^2 \norm{\bff{u}(s)}{\bb{H}^4} \ds\right)^p \right]
    \nonumber \\
    &\quad
    +
    C\, \bb{E} \left[ \left(\int_t^{t+h} (t+h-s)^{-(\beta-\frac12)} \norm{\bff{u}(s)}{\bb{H}^2}^2 \norm{\bff{u}(s)}{\bb{H}^4} \ds\right)^p \right] 
    \nonumber \\
    &\leq
    Ch^{\delta p} \, \bb{E} \left[\left(\sup_{\tau\in [0,t]} \norm{\bff{u}(\tau)}{\bb{H}^2} \right)^{2p} \left(\int_0^t (t-s)^{-(2\beta-1)-2\delta} \ds\right)^{\frac{p}{2}} \left(\int_0^t \norm{\bff{u}(s)}{\bb{H}^4}^2 \ds \right)^{\frac{p}{2}} \right]
    \nonumber \\
    &\quad
    +
    C \, \bb{E}\left[\left(\sup_{\tau\in [0,T]} \norm{\bff{u}(\tau)}{\bb{H}^2} \right)^{2p}
    \left(\int_t^{t+h} (t+h-s)^{-(2\beta-1)} \ds\right)^{\frac{p}{2}}
    \left(\int_t^{t+h} \norm{\bff{u}(s)}{\bb{H}^4}^2 \ds \right)^{\frac{p}{2}} \right]
    \nonumber \\
    &\leq
    C\, \left(\bb{E} \left[\sup_{\tau\in [0,T]} \norm{\bff{u}(\tau)}{\bb{H}^2}^{4p} \right] \right)^{\frac12} 
    \left( \bb{E} \left[\int_0^T \norm{\bff{u}(s)}{\bb{H}^4}^2 \ds \right] \right)^{\frac12}  
    \left(h^{\delta p} t^{p(1-\beta-\delta)}
    +
    h^{p(1-\beta)}\right).
\end{align}
The term in $I_4$ can be estimated in a similar manner, giving
\begin{align}\label{equ:I4 th}
    \bb{E} \left[\norm{A^{\beta}I_4(t+h)-A^{\beta} I_4(t)}{\bb{L}^2}^p\right]
    \leq
    C \left(h^{\delta p} t^{p(1-\beta-\delta)}
    +
    h^{p(1-\beta)}\right).
\end{align}
Next, for the term $I_5$, recall that $S(\bff{u})=R(\bff{u})+L(\bff{u})$ as defined in~\eqref{equ:S u} and~\eqref{equ:torque}. Noting assumptions on $L$ in Section~\ref{subsec:assump}, we have
\begin{align}\label{equ:I5 th}
    \bb{E} \left[\norm{A^{\beta}I_5(t+h)-A^{\beta} I_5(t)}{\bb{L}^2}^p\right]
    &\leq
    \bb{E} \left[\norm{\int_0^t A^{\beta-\frac14+\delta}e^{-(t-s)A}A^{-\delta}(e^{-hA}-I)A^{\frac14} S\bff{u}(s)\,\ds}{\bb{L}^2}^p\right]
    \nonumber \\
    &\quad
    +
    \bb{E} \left[\norm{\int_t^{t+h} A^{\beta-\frac14+\delta} e^{-(t-s)A} e^{-hA} A^{\frac14} S\bff{u}(s)\,\ds}{\bb{L}^2}^p \right]
    \nonumber \\
    &\leq
    Ch^{\delta p}\, \bb{E} \left[ \left(\int_0^t (t-s)^{-(\beta-\frac14)-\delta} \norm{A^{\frac14} S(\bff{u}(s))}{\bb{L}^2} \ds\right)^p\right] 
    \nonumber\\
    &\quad
    +
    C \, \bb{E} \left[\left(\int_t^{t+h} (t+h-s)^{-(\beta-\frac14)} \norm{A^{\frac14} S(\bff{u}(s))}{\bb{L}^2} \ds\right)^p \right]
    \nonumber \\
    &\leq
    C\, \bb{E} \left[\sup_{\tau\in [0,T]} \norm{\bff{u}(\tau)}{\bb{H}^1}^p \right]
    \left(h^{\delta p} t^{p(\frac54-\beta-\delta)}
    +
    h^{p(\frac54-\beta)}\right).
\end{align}
Finally, we need to consider $I_6$. To this end, we first note that the process 
\[
W(t)=\sum_{k=1}^\infty \bff{g}_k W_k(t)
\]
is by assumption a Wiener process taking values in $\mathrm{D}(\Delta)$, hence the process $B(t)=\Delta W(t)$ is a Wiener process taking values in $\bb{L}^2$. 
Also, since $\sup_{t\in [0,T]} \norm{\bff{u}(t)}{\bb{H}^2}<\infty$, we find that the process 
\[
\Delta \left[\sum_{k=1}^\infty \big(\bff{h}_k \times \bff{u}(t) \big) W_k(t) \right]
\]
is well defined in $\mathbb L^2$, and moreover 
\[
\bb{E} \left[\sup_{t\in[0,T]} \sum_{k=1}^\infty \norm{G_k(\bff{u}(t))}{\bb{H}^2}^2 \right]
\leq
\sum_{k=1}^\infty \norm{\bff{g}_k}{\bb{H}^2}^2
+
\mathbb E \left[\sup_{t\in [0,T]}\sum_{k=1}^\infty \norm{\bff{h}_k \times \bff{u}(t)}{\bb{H}^2}^2 \right] <\infty.
\]
We then have 
\begin{align*}
    \bb{E} \left[\norm{A^{\beta}I_6(t+h)- A^{\beta} I_6(t)}{\bb{L}^2}^p \right] 
    &\leq
    2^{p-1}\bb{E}\left[\norm{\sum_{k=1}^\infty \int_0^t A^{\beta-\frac12} e^{-(t-s)A} (e^{-hA}-I) A^{\frac12} G_k(\bff{u}(s))\,\dif W_k(s)}{\bb{L}^2}^p \right]
    \nonumber \\
    &\quad
    +
    2^{p-1}\bb{E}\left[
    \norm{\sum_{k=1}^\infty \int_t^{t+h} A^{\beta-\frac12} e^{-(t-s)A} e^{-hA} A^{\frac12}G_k(\bff{u}(s))\, \dif W_k(s)}{\bb{L}^2}^p \right]
    \nonumber \\
    &=
    2^{p-1}(J_1+J_2).
\end{align*}
Invoking the Burkholder--Davis--Gundy inequality, we obtain for $J_1$:
\begin{align}\label{equ:J1}
    J_1&\leq
    C\, \bb{E} \left[ \left(\sum_{k=1}^\infty \int_0^t \norm{A^{\beta-\frac12} e^{-(t-s)A}(e^{-hA}-I) A^{\frac12} G_k(\bff{u}(s))}{\bb{L}^2}^2 \ds \right)^{p/2} \right]
    \nonumber \\
    &= 
    C\, \bb{E} \left[ \left(\sum_{k=1}^\infty \int_0^t \norm{A^{\beta-\frac12+\delta} e^{-(t-s)A}A^{-\delta}(e^{-hA}-I) A^{\frac12} G_k(\bff{u}(s))}{\bb{L}^2}^2 \ds \right)^{p/2} \right]
    \nonumber \\
    &\le Ch^{\delta p}
    \left(\int_0^t
    \frac{\ds}{(t-s)^{2\beta-1+2\delta}}\right)^{p/2}\, \bb{E} \left[ \left(\sup_{\tau\in [0,t]} \sum_{k=1}^\infty \norm{A^{\frac12} G_k(\bff{u}(s))}{\bb{L}^2}^2 \right)^{p/2} \right]
    \nonumber \\
   &\le Ch^{\delta p} t^{p(1-\beta-\delta)}.
\end{align}
Similar arguments for $J_2$ give
\begin{align}\label{equ:J2}
    J_2
    &\le 
    C\, \bb{E} \left[\sum_{k=1}^\infty \int_t^{t+h}\norm{A^{\beta-\frac12} e^{-(t-s)A} e^{-hA} A^{\frac12} G_k(\bff{u}(s))}{\bb{L}^2}^2 \ds \right]^{p/2}
    \nonumber \\
    &\leq
    C\, \bb{E}\left[ \sum_{k=1}^\infty \int_t^{t+h}\frac{1}{(t+h-s)^{2\beta-1}} \norm{A^{\frac12}G_k(\bff{u}(s))}{\bb{L}^2}^2 \ds \right]^{p/2}
    \nonumber \\
    &\leq\left(\int_0^h\frac{\ds}{s^{2\beta-1}}\right)^{p/2}\bb
    E\left[ \left(\sup_{\tau \in [0,T]} \sum_{k=1}^\infty \norm{A^{\frac12}G_k(u(s))}{\bb{L}^2}^2\right)^{p/2} \right]
    \nonumber \\
    &\le Ch^{p(1-\beta)}.
\end{align}
Combining \eqref{equ:J1} and \eqref{equ:J2}, we obtain
\begin{equation}\label{equ:I6 th}
 \bb{E} \left[\norm{A^{\beta}I_6(t+h)- A^{\beta} I_6(t)}{\bb{L}^2}^p \right]\le 
 C\left(h^{\delta p} t^{p(1-\beta-\delta)} + h^{p(1-\beta)}\right).
 \end{equation}
Altogether, estimates~\eqref{equ:I1 th}, \eqref{equ:I2 th}, \eqref{equ:I3 th}, \eqref{equ:I4 th}, \eqref{equ:I5 th}, and \eqref{equ:I6 th} imply that for any $t>0$,
\begin{equation}\label{equ:E cont H3}
    \bb{E} \left[\norm{A^{\beta} \bff{v}(t+h)-A^{\beta} \bff{v}(t)}{\bb{L}^2}^p \right]
    \leq
    Ch^{\delta p} t^{p(1-\beta-\delta)} + h^{p(1-\beta)}.
\end{equation}
Since $p$ can be taken arbitrarily large, we find that for every $\delta\in(0,1-\beta)$, where $\beta\in [\frac12,1)$, $\bb{P}$-a.s. we have $\bff{v} \in C^\delta([0,T];\text{D}(A^\beta))$ by the Kolmogorov continuity theorem.

Finally, since the map $(t,\bff{u}_0) \mapsto e^{-tA}\bff{u}_0$ is continuous in $[0,\infty)\times \text{D}(A^{1/2})$, noting~\eqref{equ:E cont H3} for $\beta=\frac12$, we have $\bff{u}\in C([0,T];\bb{H}^2)$.
This completes the proof of the lemma.
\end{proof}

\section{Proof of Theorem~\ref{the:unique}: Pathwise Uniqueness}

We prove Theorem~\ref{the:unique} in this section.

\begin{proof}[Proof of Theorem~\ref{the:unique}]
Let $\bff{v}:= \bff{u}_1-\bff{u}_2$ and $\bff{B}:=\bff{H}_1-\bff{H}_2$, where $\bff{u}_1$ and $\bff{u}_2$ are martingale solutions of \eqref{equ:sllbar} (on the same probability space and with the same Wiener process).
Then $\bff{v}$ satisfies
\begin{align*}
    \nonumber
	\dif \bff{v}
	&=
	\big( \lambda_r \bff{B}- \lambda_e \Delta \bff{B}- \gamma(\bff{v}\times \bff{H}_1 + \bff{u}_2 \times \bff{B})+ R(\bff{u}_1)-R(\bff{u}_2) + L(\bff{u}_1)-L(\bff{u}_2) \big) \, \dt
    \\
    &\quad
	+
	\left( \sum_{k=1}^\infty G_k(\bff{u}_1)- G_k(\bff{u}_2) \right) \dif W_k(t)
\end{align*}
with $\bff{v}(0)=\bff{0}$.
Moreover, for any $R>0$, define the stopping time
\begin{equation}\label{equ:tau R}
	\tau_R := \inf\{t\geq 0: \norm{\bff{u}_1(t)}{\bb{H}^2} \vee \norm{\bff{u}_2(t)}{\bb{H}^2} > R\}.
\end{equation}
Note that $\tau_R\to \infty$ as $R\to \infty$.
Now, by It\^o's lemma,
\begin{align}\label{equ:ito dv L2}
\nonumber
	\frac{1}{2} \dif \norm{\bff{v}}{\bb{L}^2}^2
	&=
	\Big( \lambda_r \inpro{\bff{B}}{\bff{v}}_{\bb{L}^2}
	+
	\lambda_e \inpro{\nabla \bff{B}}{\nabla \bff{v}}_{\bb{L}^2}
	-
	\gamma\inpro{\bff{u}_2\times \bff{B}}{\bff{v}}_{\bb{L}^2}
    \\
    &\qquad 
    +
    \inpro{R(\bff{u}_1)- R(\bff{u}_2)}{\bff{v}}_{\bb{L}^2}
    +
    \inpro{L(\bff{u}_1)- L(\bff{u}_2)}{\bff{v}}_{\bb{L}^2}
    +
    \frac{1}{2} \sum_{k=1}^\infty \norm{G_k(\bff{u}_1)- G_k(\bff{u}_2)}{\bb{L}^2}^2
	\Big) \,\dt.
\end{align}
We have
\begin{align}\label{equ:lambda B v}
	\lambda_r \inpro{\bff{B}}{\bff{v}}_{\bb{L}^2}
	&=
	-
	\lambda_r \norm{\nabla \bff{v}}{\bb{L}^2}^2
	+
	\lambda_r \norm{\bff{v}}{\bb{L}^2}^2
	-
	\lambda_r \norm{|\bff{u}_1|\, |\bff{v}|}{\bb{L}^2}^2
	-
	\lambda_r \norm{\bff{u}_2\cdot \bff{v}}{\bb{L}^2}^2
	-
	\lambda_r \inpro{(\bff{u}_1\cdot \bff{v})\bff{u}_2}{\bff{v}}_{\bb{L}^2},
	\\
	\label{equ:lambda nab B v}
	\lambda_e \inpro{\nabla \bff{B}}{\nabla \bff{v}}_{\bb{L}^2}
	&=
	-
	\lambda_e \norm{\Delta \bff{v}}{\bb{L}^2}^2
	+
	\lambda_e \norm{\nabla \bff{v}}{\bb{L}^2}^2
	+
	\lambda_e \inpro{|\bff{u}_1|^2 \bff{v}}{\Delta \bff{v}}_{\bb{L}^2}
	+
	\lambda_e \inpro{\big((\bff{u}_1+\bff{u}_2)\cdot \bff{v} \big) \bff{u}_2}{\Delta \bff{v}}_{\bb{L}^2},
\end{align}
and
\begin{align}\label{equ:gamma u B}
	-\gamma \inpro{\bff{u}_2\times \bff{B}}{\bff{v}}_{\bb{L}^2}
	&=
	-\gamma \inpro{\bff{u}_2 \times \Delta \bff{v}}{\bff{v}}_{\bb{L}^2}.
\end{align}
Substituting \eqref{equ:lambda B v}, \eqref{equ:lambda nab B v}, and \eqref{equ:gamma u B} into \eqref{equ:ito dv L2}, and integrating with respect to $t$ yield
\begin{align}\label{equ:I vt L2}
	\nonumber
	&\frac{1}{2} \norm{\bff{v}(t)}{\bb{L}^2}^2
	+
	\lambda_e \int_0^t \norm{\Delta \bff{v}(s)}{\bb{L}^2}^2 \ds 
	+
	(\lambda_r - \lambda_e) \int_0^t \norm{\nabla \bff{v}(s)}{\bb{L}^2}^2 \ds 
	\\
	\nonumber
	&\leq
	\lambda_r \int_0^t \norm{\bff{v}(s)}{\bb{L}^2}^2 \ds 
	+
	\lambda_r \int_0^t \inpro{(\bff{u}_1(s)\cdot \bff{v}(s))\bff{u}_2(s)}{\bff{v}(s)}_{\bb{L}^2} \ds 
	+
	\lambda_e \int_0^t \inpro{|\bff{u}_1(s)|^2 \bff{v}(s)}{\Delta \bff{v}(s)}_{\bb{L}^2} \ds
	\\
	\nonumber
	&\quad 
	+
	\lambda_e \int_0^t \inpro{\big((\bff{u}_1(s)+\bff{u}_2(s))\cdot \bff{v}(s) \big) \bff{u}_2(s)}{\Delta \bff{v}(s)}_{\bb{L}^2} \ds
    -
	\gamma \int_0^t \inpro{\bff{u}_2(s) \times \Delta \bff{v}(s)}{\bff{v}(s)}_{\bb{L}^2} \ds 
    \\
    \nonumber
    &\quad
    +
    \int_0^t \inpro{R(\bff{u}_1)- R(\bff{u}_2)}{\bff{v}}_{\bb{L}^2} \ds
    +
    \int_0^t \inpro{L(\bff{u}_1)- L(\bff{u}_2)}{\bff{v}}_{\bb{L}^2} \ds
    +
    \frac{1}{2} \sum_{k=1}^\infty \int_0^t \norm{G_k(\bff{u}_1)- G_k(\bff{u}_2)}{\bb{L}^2}^2 \ds
	\\
	&=:
	I_1(t)+I_2(t)+I_3(t)+I_4(t)+I_5(t)+I_6(t)+I_7(t)+I_8(t).
\end{align}
With $\tau_R$ as defined in \eqref{equ:tau R}, we will derive bounds for $I_j(t\wedge \tau_R)$, for $j=2,\ldots, 8$, as follows. For the term $I_2(t\wedge \tau_R)$, by H\"older's inequality and Sobolev embedding,
\begin{align*}
	I_2(t\wedge \tau_R) 
	\leq 
	\lambda_r \int_0^{t\wedge \tau_R} \norm{\bff{u}_1(s)}{\bb{L}^\infty} \norm{\bff{u}_2(s)}{\bb{L}^\infty}
	\norm{\bff{v}(s)}{\bb{L}^2}^2 \ds
	\leq 
	C\lambda_r R^2 \int_0^{t\wedge \tau_R} \norm{\bff{v}(s)}{\bb{L}^2}^2 \ds.
\end{align*}
For the term $I_3(t\wedge \tau_R)$, by H\"older's and Young's inequalities, and Sobolev embedding, we have
\begin{align*}
	I_3(t \wedge \tau_R)
	&\leq 
	\lambda_e \int_0^{t\wedge \tau_R} \norm{\bff{u}_1(s)}{\bb{L}^\infty}^2 \norm{\bff{v}(s)}{\bb{L}^2} \norm{\Delta \bff{v}(s)}{\bb{L}^2} \ds
	\\
	&\leq
	 \frac{\lambda_e}{8} \int_0^{t\wedge \tau_R} \norm{\Delta \bff{v}(s)}{\bb{L}^2}^2 \ds
	 +
	 \frac{1}{2\lambda_e} \int_0^{t\wedge \tau_R} \norm{\bff{u}_1(s)}{\bb{L}^\infty}^4 \norm{\bff{v}(s)}{\bb{L}^2}^2 \ds
	\\
	&\leq
	\frac{\lambda_e}{8} \int_0^{t\wedge \tau_R} \norm{\Delta \bff{v}(s)}{\bb{L}^2}^2 \ds
	+
	\frac{CR^4}{\lambda_e} \int_0^{t\wedge \tau_R} \norm{\bff{v}(s)}{\bb{L}^2}^2 \ds
\end{align*}
Similarly,
\begin{align*}
	I_4(t\wedge \tau_R)
	&\leq
	\frac{\lambda_e}{8} \int_0^{t\wedge \tau_R} \norm{\Delta \bff{v}(s)}{\bb{L}^2}^2 \ds
	+
	\frac{1}{2\lambda_e} \int_0^{t\wedge \tau_R} \norm{\bff{u}_1(s)+\bff{u}_2(s)}{\bb{L}^\infty}^2 \norm{\bff{u}_2(s)}{\bb{L}^\infty}^2 \norm{\bff{v}(s)}{\bb{L}^2}^2 \ds
	\\
	&\leq 
	\frac{\lambda_e}{8} \int_0^{t\wedge \tau_R} \norm{\Delta \bff{v}(s)}{\bb{L}^2}^2 \ds
	+
	\frac{CR^4}{\lambda_e} \int_0^{t\wedge \tau_R} \norm{\bff{v}(s)}{\bb{L}^2}^2 \ds 
\end{align*}
and
\begin{align*}
	I_5(t\wedge \tau_R)
	&\leq
	\frac{\lambda_e}{8} \int_0^{t\wedge \tau_R} \norm{\Delta \bff{v}(s)}{\bb{L}^2}^2 \ds
	+
	\frac{\gamma^2}{2\lambda_e} \int_0^{t\wedge \tau_R} \norm{\bff{u}_2(s)}{\bb{L}^\infty}^2 \norm{\bff{v}(s)}{\bb{L}^2}^2 \ds
	\\
	&\leq 
	\frac{\lambda_e}{8} \int_0^{t\wedge \tau_R} \norm{\Delta \bff{v}(s)}{\bb{L}^2}^2 \ds
	+
	\frac{C\gamma^2 R^2}{\lambda_e} \int_0^{t\wedge \tau_R} \norm{\bff{v}(s)}{\bb{L}^2}^2 \ds.
\end{align*}
Next, for the term $I_6(t\wedge \tau_R)$, by \eqref{equ:Rvw vw},
\begin{align*}
    I_6(t\wedge \tau_R)
    &\leq
    \frac{\lambda_e}{8} \int_0^{t\wedge \tau_R} \norm{\Delta \bff{v}(s)}{\bb{L}^2}^2 \ds
    +
    C\int_0^{t\wedge\tau_R} \norm{\bff{\nu}}{L^\infty(0,T;\bb{L}^\infty)}^2 \left(1+\norm{\bff{u}_2(s)}{\bb{L}^\infty}^2\right) \norm{\bff{v}(s)}{\bb{L}^2}^2 \ds
    \\
    &\leq
    \frac{\lambda_e}{8} \int_0^{t\wedge \tau_R} \norm{\Delta \bff{v}(s)}{\bb{L}^2}^2 \ds
    +
    C(1+R^2) \int_0^{t\wedge \tau_R} \norm{\bff{v}(s)}{\bb{L}^2}^2 \ds
\end{align*}
For the term $I_7(t\wedge \tau_R)$, by the assumptions on $L$,
\begin{align*}
    I_7(t\wedge \tau_R)
    \leq
    C\int_0^{t\wedge \tau_R} \norm{\bff{v}(s)}{\bb{L}^2}^2 \ds.
\end{align*}
Finally, for the term $I_8(t\wedge \tau_R)$, we have
\begin{align*}
    I_8(t\wedge \tau_R)
    \leq
    \sigma_h \gamma^2 \int_0^{t\wedge \tau_R} \norm{\bff{v}(s)}{\bb{L}^2}^2 \ds.
\end{align*}

For brevity, we will assume $\lambda_e-\lambda_r >0$ (otherwise interpolation inequality~\eqref{equ:nabla v L2} can be used). Therefore, replacing $t$ by $t\wedge \tau_R$ in \eqref{equ:I vt L2}, rearranging the terms, and taking expectations, we obtain
\begin{align}\label{equ:E v L2 local t}
	\bb{E} \left[\sup_{\tau \in [0,t\wedge \tau_R]} \norm{\bff{v}(\tau)}{\bb{L}^2}^2 \right]
	+
	\bb{E} \left[ \int_0^{t\wedge \tau_R} \norm{\Delta \bff{v}(s)}{\bb{L}^2}^2 \ds \right]
	\leq
	C(1+R^4) \int_0^t \bb{E} \left[\sup_{\tau\in [0,s\wedge \tau_R]} \norm{\bff{v}(\tau)}{\bb{L}^2}^2 \right] \ds.
\end{align}
By Gronwall's inequality, we infer that for any $t\geq 0$,
\begin{equation*}
	\bb{E} \left[\sup_{\tau \in [0,t\wedge \tau_R]} \norm{\bff{v}(\tau)}{\bb{L}^2}^2 \right] = 0.
\end{equation*}
Letting $R\uparrow \infty$, and applying the monotone convergence theorem, we obtain $\bff{v}(t)\equiv \bff{0}$ for $\bb{P}$-a.s. $\omega\in \Omega$, as required.
The existence of a pathwise unique strong solution (in the sense of probability) and the uniqueness in law of the martingale solution are consequences of~\cite[Theorem~2.2 and Theorem~12.1]{Ond04}.
\end{proof}

\section{Proof of Theorem~\ref{the:invariant}: Existence of Invariant Measures}

We will show first the existence of invariant measures for the stochastic LLBar equation~\eqref{equ:sllbar}. 
Recall that by Theorem~\ref{the:exist} and Theorem~\ref{the:unique}, for any initial data $\bff{u}_0\in \bb{H}^2$, the stochastic LLBar equation has a unique pathwise strong solution which defines an $\bb{H}^2$-valued Markov process $\bff{u}$. Let $\bff{u}(t;\bff{u}_0)$ denote the process $\bff{u}$ starting at $\bff{u}(0)=\bff{u}_0$. For a Borel set $B$ in $\bb{H}^2$ and any $t\geq 0$, we define the transition functions
\[
    P_t(\bff{u}_0, B):= \bb{P}(\bff{u}(t;\bff{u}_0)\in B).
\]
We can study the Markov semigroup $\{P_t\}_{t\geq 0}$, defined for $\phi\in B_b(\bb{H}^2)$ by
\begin{align}\label{equ:transition}
    P_t \phi(\bff{u}_0) := \bb{E}\left[\phi(\bff{u}(t;\bff{u}_0))\right]
    =
    \int_{\bb{H}^2} \phi(\bff{u}) P_t(\bff{u}_0, \mathrm{d}\bff{u}),
    \quad \forall \bff{u}_0\in \bb{H}^2.
\end{align}
A Borel probability measure $\mu$ on $\bb{H}^2$ is called an invariant measure for the Markov semigroup associated to the problem~\eqref{equ:sllbar} if for every $t\geq 0$,
\[
    \int_{\bb{H}^2} \phi(\bff{u}_0) \,\mathrm{d}\mu(\bff{u}_0)
    =
    \int_{\bb{H}^2} P_t \phi(\bff{u}_0) \,\mathrm{d}\mu(\bff{u}_0), \quad \forall \phi\in C_b(\bb{H}^2).
\]
First, we prove the following bounds which will be used later.

\begin{lemma}\label{lem:E u H2 t}
Let $\bff{u}$ be the solution of \eqref{equ:sllbar} given by Theorem~\ref{the:exist}. Then there exists a positive constant $C$ such that for all $t\geq 0$,
\begin{align*}
    \bb{E} \left[\norm{\bff{u}(t)}{\bb{H}^1}^2 \right] 
    +
    \int_0^t \bb{E} \left[ \norm{\bff{u}(s)}{\bb{H}^2}^2 \right] \ds 
    +
    \int_0^t \bb{E} \left[ \norm{\bff{H}(s)}{\bb{H}^1}^2 \right] \ds 
    +
    \int_0^t \bb{E} \big[ \norm{\nabla\Delta \bff{u}(s)}{\bb{L}^2} \big] \ds 
    \leq 
    C(1+t),
\end{align*}
where $C$ depends on $\norm{\bff{u}_0}{\bb{H}^1}$, but is independent of $t$.
\end{lemma}

\begin{proof}
The proof follows the same line of argument as in Proposition~\ref{pro:E sup un L2} and Proposition~\ref{pro:E sup nab un L2}, with $\bff{u}_n$ and $\bff{H}_n$ replaced by $\bff{u}$ and $\bff{H}$, respectively. As such, we will just outline the key steps here, while ensuring that the constant $C$ is independent of $t$.

Firstly, following the argument leading to~\eqref{equ:un L2}, and assuming $\lambda_r-\lambda_e>0$ for simplicity of presentation, we obtain
\begin{align}\label{equ:u L2}
    \nonumber
	&\frac{1}{2} \norm{\bff{u}(t)}{\bb{L}^2}^2
	+
	(\lambda_r-\lambda_e) \int_0^t \norm{\nabla \bff{u}(s)}{\bb{L}^2}^2 \ds
	+
	\lambda_e \int_0^t \norm{\Delta \bff{u}(s)}{\bb{L}^2}^2 \ds
	+
	\lambda_r \int_0^t \norm{\bff{u}(s)}{\bb{L}^4}^4 \ds
	\\
	\nonumber
	&\quad
	+
	\lambda_e \int_0^t \norm{|\bff{u}(s)|\, |\nabla \bff{u}(s)|}{\bb{L}^2}^2 \ds
	+
	2\lambda_e \int_0^t \norm{\bff{u}(s) \cdot \nabla \bff{u}(s)}{\bb{L}^2}^2 \ds
    \\
    \nonumber
	&\leq
	\frac{1}{2} \norm{\bff{u}_{0}}{\bb{L}^2}^2
	+
	\lambda_r \int_0^t \norm{\bff{u}(s)}{\bb{L}^2}^2 \ds
	+
	C \int_0^t \norm{\bff{\nu}}{\bb{L}^2(\mathscr{D};\bb{R}^d)} \norm{|\bff{u}(s)|\, |\nabla \bff{u}(s)|}{\bb{L}^2} \ds
	+
	C \int_0^t \big(1+ \norm{\bff{u}(s)}{\bb{L}^2}^2 \big) \ds
	\\
	\nonumber
	&\quad
	+
	\frac{1}{2} \sum_{k=1}^n \int_0^t \big( \norm{\bff{g}_k}{\bb{L}^2}^2 + \norm{\bff{h}_k}{\bb{L}^\infty}^2 \norm{\bff{u}(s)}{\bb{L}^2}^2 \big) \ds
	+
	\sum_{k=1}^\infty \int_0^t \inpro{\bff{g}_k}{\bff{u}(s)}_{\bb{L}^2} \dif W_k(s)
    \\
    \nonumber
    &\leq
    \frac{1}{2} \norm{\bff{u}_{0}}{\bb{L}^2}^2
	+
    Ct
    +
    C \norm{\bff{\nu}}{L^2(0,t;\bb{L}^2)}^2
    +
    \frac{\lambda_e}{2} \int_0^t \norm{|\bff{u}(s)|\, |\nabla \bff{u}(s)|}{\bb{L}^2}^2 \ds
    +
    C \int_0^t \norm{\bff{u}(s)}{\bb{L}^2}^2 \ds
    \\
    &\quad
    +
    \sum_{k=1}^\infty \int_0^t \inpro{\bff{g}_k}{\bff{u}(s)}_{\bb{L}^2} \dif W_k(s)
\end{align}
where in the last step we used Young's inequality and the assumptions in \ref{subsec:assump}. Therefore, noting the inequality $x^2\leq \epsilon x^4 + 1/(4\epsilon)$ for any $\epsilon>0$, we can absorb the penultimate term on~\eqref{equ:u L2} to the left. Furthermore, since
\[
    \bb{E} \left[\sum_{k=1}^\infty \int_0^t \big| \inpro{\bff{g}_k}{\bff{u}(s)}_{\bb{L}^2} \big|^2 \,\ds \right]
    \leq
    \left(\sum_{k=1}^\infty \norm{\bff{g}_k}{\bb{L}^2}^2\right) \,
    \bb{E} \left[\int_0^t \norm{\bff{u}(s)}{\bb{L}^2}^2 \ds \right] < \infty,
\]
the process $t\mapsto \sum_{k=1}^\infty \int_0^t \inpro{\bff{g}_k}{\bff{u}(s)}_{\bb{L}^2}\,\dif W_k(s)$ is a martingale on $[0,T]$. Therefore,
\begin{equation}\label{equ:E gk u dW}
    \bb{E} \left[\int_0^t \inpro{\bff{g}_k}{\bff{u}(s)}_{\bb{L}^2}  \dif W_k(s)\right] = 0.
\end{equation}
Taking expectation on both sides of \eqref{equ:u L2}, rearranging the terms, and noting~\eqref{equ:E gk u dW}, we then have
\begin{equation}\label{equ:E u H2}
    \bb{E} \left[\norm{\bff{u}(t)}{\bb{L}^2}^2 \right] 
    +
    \bb{E} \left[\int_0^t \norm{\bff{u}(s)}{\bb{H}^2}^2 \ds \right]
    +
    \bb{E} \left[\int_0^t \norm{|\bff{u}(s)|\, |\nabla \bff{u}(s)|}{\bb{L}^2}^2 \ds\right]
    \leq
    C(1+t).
\end{equation}

Next, we follow the argument leading to~\eqref{equ:psi un nab Hn} to obtain
\begin{align}\label{equ:nab u L2}
    &\frac{1}{2} \norm{\nabla \bff{u}(t)}{\bb{L}^2}^2
    +
    \frac{1}{4} \norm{\bff{u}(t)}{\bb{L}^4}^4
    +
	\lambda_r \int_0^t \norm{\bff{H}(s)}{\bb{L}^2}^2 \ds
	+
	\lambda_e \int_0^t \norm{\nabla \bff{H}(s)}{\bb{L}^2}^2 \ds
	+
	\frac{1}{2} \sum_{k=1}^\infty \int_0^t \norm{G_k(\bff{u}(s))}{\bb{L}^2}^2 \ds
    \nonumber \\
    &\leq
    C \norm{\bff{u}_0}{\bb{H}^1}^2
    +
    C(1+t)
    +
    \frac{\lambda_r}{2} \int_0^t  \norm{\bff{H}(s)}{\bb{L}^2}^2 \ds
    +
    C \int_0^t \norm{|\bff{u}(s)|\, |\nabla \bff{u}(s)|}{\bb{L}^2}^2 \ds
    \nonumber \\
    &\quad
    +
    C \int_0^t \norm{\bff{u}(s)}{\bb{H}^1}^2 \ds
    +
    C\int_0^t \norm{\bff{u}(s)}{\bb{L}^4}^4 \ds
    -
	\sum_{k=1}^\infty \int_0^t \inpro{\bff{g}_k + \gamma \bff{u}(s) \times \bff{h}_k}{\bff{H}(s)}_{\bb{L}^2} \dif W_k(s).
\end{align}
Now, note that 
\begin{align*}
    &\bb{E} \left[\sum_{k=1}^\infty \int_0^t \big|\inpro{\bff{g}_k + \gamma \bff{u}(s) \times \bff{h}_k}{\bff{H}(s)}_{\bb{L}^2} \big|^2\, \ds\right]
    \\
    &\le
    \left(\sum_{k=1}^\infty \norm{\bff{g}_k}{\bb{L}^2}^2\right) 
    \bb{E} \left[\int_0^t \norm{\bff{H}(s)}{\bb{L}^2}^2 \ds \right]
    +
    \left(\sum_{k=1}^\infty \norm{\bff{h}_k}{\bb{L}^2}^2\right) \bb{E}\left[ \int_0^t \norm{\bff{u}(s)}{\bb{L}^\infty}^2 \ds \right] 
    + 
    C \bb{E} \left[\int_0^t \norm{\bff{H}(s)}{\bb{L}^2}^2 \ds \right]
\end{align*}
which is finite by our assumptions in \ref{subsec:assump} and Theorem~\ref{the:exist}. In particular, the process
\[
t\mapsto \sum_{k=1}^\infty \int_0^t \inpro{\bff{g}_k + \gamma \bff{u}(s) \times \bff{h}_k}{\bff{H}(s)}_{\bb{L}^2} \dif{W}_k(s)
\]
is a martingale on $[0,T]$. This implies
\[
    \bb{E} \left[\sum_{k=1}^\infty \int_0^t \inpro{\bff{g}_k + \gamma \bff{u}(s) \times \bff{h}_k}{\bff{H}(s)}_{\bb{L}^2}  \dif{W}_k(s)\right] = 0.
\]
Therefore, taking expectation on both sides of~\eqref{equ:nab u L2} and noting \eqref{equ:E u H2}, we obtain
\begin{align}\label{equ:E int H 1}
    \bb{E} \left[\norm{\nabla \bff{u}(t)}{\bb{L}^2}^2 \right]
    +
    \bb{E} \left[\norm{\bff{u}(t)}{\bb{L}^4}^4 \right]
    +
    \bb{E} \left[\int_0^t \norm{\bff{H}(s)}{\bb{H}^1}^2 \ds \right] 
    \leq
    C(1+t).
\end{align}
Finally, note that $\nabla\Delta \bff{u}= \nabla \bff{H} -\nabla \bff{u} + \nabla(|\bff{u}|^2 \bff{u})$. By Young's inequality and Sobolev embedding,
\begin{align*}
    \int_0^t \bb{E}\big[\norm{\nabla \Delta \bff{u}(s)}{\bb{L}^2}\big] \ds
    &\leq
    \int_0^t \bb{E} \big[\norm{\bff{H}(s)}{\bb{H}^1} + \norm{\bff{u}(s)}{\bb{H}^1} + 2\norm{|\bff{u}(s)|\,|\nabla \bff{u}(s)|}{\bb{L}^2} \norm{\bff{u}(s)}{\bb{L}^\infty}\big] \,\ds
    \\
    &\leq
    C \int_0^t \left(1+ \bb{E}\left[ \norm{\bff{H}(s)}{\bb{H}^1}^2 + \norm{\bff{u}(s)}{\bb{H}^1}^2 + \norm{|\bff{u}(s)|\,|\nabla \bff{u}(s)|}{\bb{L}^2}^2 + \norm{\bff{u}(s)}{\bb{H}^2}^2\right] \right) \ds 
    \\
    &\leq
    C(1+t),
\end{align*}
where in the last step we used~\eqref{equ:E u H2} and~\eqref{equ:E int H 1}.
This completes the proof of the lemma.
\end{proof}

\begin{lemma}\label{lem:t u H3}
Let $\bff{u}$ be the solution of \eqref{equ:sllbar} given by Theorem~\ref{the:exist}. Then there exists a positive constant $C$ independent of $t$ such that for all $t\geq 0$,
\[
    \bb{E}\left[\sup_{\tau\in [0,t]} \tau \norm{\nabla\Delta \bff{u}(\tau)}{\bb{L}^2}^2\right] 
    \leq
    C(1+e^{Ct}),
\]
where $C$ is a positive constant depending on $\norm{\bff{u}_0}{\bb{H}^2}$.
\end{lemma}

\begin{proof}
We will use~\eqref{equ:u mild} and estimate $\bb{E}\left[t\norm{A^{\frac34} I_j(t)}{\bb{L}^2}^2\right]$ for $j=0,1,\ldots,6$.
Firstly, we have
\begin{align}\label{equ:t I0}
    \bb{E}\left[t\norm{A^{\frac34} I_0(t)}{\bb{L}^2}^2\right]
    =
    \bb{E}\left[t\norm{A^{\frac14} e^{-tA} A^{\frac12}\bff{u}_0}{\bb{L}^2}^2\right]
    &\leq
    \sqrt{t} \norm{A^{\frac12} \bff{u}_0}{\bb{L}^2}^2
\end{align}
The rest of the estimates follows the same argument as in the proof of Lemma~\ref{lem:H3}, noting Proposition~\ref{pro:E sup un L2} and Proposition~\ref{pro:E sup un H2}. Therefore, we have
\begin{align}\label{equ:t Ij}
    \bb{E}\left[t\norm{A^{\frac34} I_j(t)}{\bb{L}^2}^2\right]
    &\leq
    C(1+e^{Ct}),
\end{align}
for $j=1,2,\ldots,6$, where $C$ depends on $\norm{\bff{u}_0}{\bb{H}^2}$. The required estimate follows from~\eqref{equ:t I0} and~\eqref{equ:t Ij}.
\end{proof}

To prove the existence of an invariant measure, we adopt the techniques in~\cite{GlaKukVic14}. First, we prove a continuous dependence estimate and some strong moment bounds to establish the Feller property of the transition semigroup $P_t$ defined in~\eqref{equ:transition}. The existence of an invariant measure then follows from the well-known theorem of Krylov--Bogoliubov.

In the following, we will use the stopping time defined for any $S>0$ by
\begin{equation}\label{equ:sigma S v0}
    \sigma_S(\bff{v}_0)
    := 
    \inf\{t\geq 0: \norm{\bff{u}(t;\bff{v}_0)}{\bb{H}^2}^2 > S\}.
\end{equation}
The next lemma shows a continuous dependence result up to a stopping time.

\begin{lemma}\label{lem:cont dep H2}
Let $\bff{u}^m$ and $\bff{u}$ be the solution of \eqref{equ:sllbar} with initial data $\bff{u}_0^m$ and $\bff{u}_0$ respectively. Let $\sigma_S:= \sigma_S(\bff{u}_0^m) \wedge \sigma_S(\bff{u}_0)$.
Then for any $t\in [0,T]$ and $R>0$,
\begin{align*}
    \bb{E}\left[ \sup_{\tau\in [0,t\wedge \sigma_S]} \norm{\bff{u}^m(\tau)-\bff{u}(\tau)}{\bb{H}^2}^2 \right]
    +
	\bb{E} \left[ \int_0^{t\wedge \sigma_S} \norm{\bff{u}^m(s)-\bff{u}(s)}{\bb{H}^4}^2 \ds \right]
    \leq
    C_{S,T} \norm{\bff{u}_0^m-\bff{u}_0}{\bb{H}^2}^2,
\end{align*}
where $C_{S,T}$ is a constant depending on $S$ and $T$.
\end{lemma}

\begin{proof}
Let $\bff{v}(t):= \bff{u}^m(t)-\bff{u}(t)$. By It\^o's lemma, following the same argument as in the proof of Theorem~\ref{the:unique}, we obtain an inequality similar to~\eqref{equ:E v L2 local t}, namely
\begin{align}\label{equ:E sup loc L2}
    \bb{E} \left[\sup_{\tau \in [0,t\wedge \sigma_S]} \norm{\bff{v}(\tau)}{\bb{L}^2}^2 \right]
	&+
	\bb{E} \left[ \int_0^{t\wedge \sigma_S} \norm{\Delta \bff{v}(s)}{\bb{L}^2}^2 \ds \right]
    \nonumber \\
	&\leq
    C \norm{\bff{v}(0)}{\bb{L}^2}^2
    +
	C(1+S^2) \int_0^t \bb{E} \left[\sup_{\tau\in [0,s\wedge \sigma_S]} \norm{\bff{v}(\tau)}{\bb{L}^2}^2 \right] \ds.
\end{align}

Subsequently, for clarity we will suppress the dependence of the functions on $t$. Applying It\^o's lemma to the function $\psi(\bff{v})= \frac{1}{2} \norm{\Delta \bff{v}}{\bb{L}^2}^2$ then integrating with respect to $t$ (cf. proof of Proposition~\ref{pro:E sup un H2}), we obtain
\begin{align*}
    &\frac{1}{2} \norm{\Delta \bff{v}(t)}{\bb{L}^2}^2
    +
    \lambda_e \int_0^t \norm{\Delta^2 \bff{v}}{\bb{L}^2}^2 \ds
    +
    (\lambda_r-\lambda_e) \int_0^t \norm{\nabla \Delta \bff{v}}{\bb{L}^2}^2 \ds
    \\
    &=
    \frac{1}{2} \norm{\Delta \bff{v}(0)}{\bb{L}^2}^2
    +
    \lambda_r \int_0^t \norm{\Delta \bff{v}}{\bb{L}^2}^2 \ds
    \\
    &\quad
    -
    \lambda_r \int_0^t \inpro{\Delta\big(|\bff{u}^m|^2 \bff{v}\big)}{\Delta \bff{v}}_{\bb{L}^2} \ds 
    -
    \lambda_r \int_0^t \inpro{\Delta\big(((\bff{u}^m+\bff{u})\cdot \bff{v})\bff{u}\big)}{\Delta \bff{v}}_{\bb{L}^2} \ds
    \\
    &\quad
    +
    \lambda_e \int_0^t \inpro{\Delta\big(|\bff{u}^m|^2 \bff{v}\big)}{\Delta^2 \bff{v}}_{\bb{L}^2} \ds 
    +
    \lambda_r \int_0^t \inpro{\Delta\big(((\bff{u}^m+\bff{u})\cdot \bff{v})\bff{u}\big)}{\Delta^2 \bff{v}}_{\bb{L}^2} \ds
    \\
    &\quad
    +
    \int_0^t \inpro{R(\bff{u}^m)-R(\bff{u})}{\Delta^2 \bff{v}}_{\bb{L}^2} \ds
    +
    \int_0^t \inpro{L(\bff{u}^m)-L(\bff{u})}{\Delta^2 \bff{v}}_{\bb{L}^2} \ds
    \\
    &\quad
    +
    \sum_{k=1}^\infty \inpro{G_k(\bff{u}_1)-G_k(\bff{u}_2)}{\Delta^2 \bff{v}}_{\bb{L}^2} \dif W_k(s)
    \\
    &=:
    \frac{1}{2} \norm{\Delta \bff{v}(0)}{\bb{L}^2}^2
    +
    I_2(t)+I_3(t)+\ldots+I_9.
    \end{align*}
We will estimate $I_j(t\wedge \sigma_S)$ for $j=2,\ldots,9$ as follows. Firstly, for the terms $I_2(t\wedge \sigma_S)$ and $I_3(t\wedge \sigma_S)$, by~\eqref{equ:prod Hs mat dot},
\begin{align*}
    |I_2(t\wedge \sigma_S)|
    &\leq
    C \int_0^t \norm{\bff{u}^m}{\bb{H}^2}^2 \norm{\bff{v}}{\bb{H}^2}^2 \ds
    \leq
    CS \int_0^t \norm{\bff{v}}{\bb{H}^2}^2 \ds,
    \\
    |I_3(t\wedge \sigma_S)|
    &\leq
    C \int_0^t \norm{\bff{u}^m+\bff{u}}{\bb{H}^2} \norm{\bff{u}}{\bb{H}^2} \norm{\bff{v}}{\bb{H}^2}^2 \ds
    \leq
    CS \int_0^t \norm{\bff{v}}{\bb{H}^2}^2 \ds.
\end{align*}
Similarly, for the next term, by~\eqref{equ:prod Hs mat dot} and Young's inequality,
\begin{align*}
    |I_4(t\wedge \sigma_S)|
    &\leq
    C \int_0^t \norm{\bff{u}^m}{\bb{H}^2}^2 \norm{\bff{v}}{\bb{H}^2} \norm{\Delta^2 \bff{v}}{\bb{L}^2} \ds
    \leq
    CS \int_0^t \norm{\bff{v}}{\bb{H}^2}^2 \ds 
    +
    \frac{\lambda_e}{8} \int_0^t \norm{\Delta^2 \bff{v}}{\bb{L}^2}^2 \ds.
\end{align*}
The rest of the terms can be estimated in a similar manner, which we omit for brevity. Taking expected value, we obtain the inequality
\begin{align}\label{equ:E sup loc H2}
    \bb{E} \left[\sup_{\tau \in [0,t\wedge \sigma_S]} \norm{\Delta \bff{v}(\tau)}{\bb{L}^2}^2 \right]
	&+
	\bb{E} \left[ \int_0^{t\wedge \sigma_S} \norm{\Delta^2 \bff{v}(s)}{\bb{L}^2}^2 \ds \right]
    \nonumber \\
	&\leq
    C \norm{\Delta \bff{v}(0)}{\bb{L}^2}^2
    +
	C(1+S) \int_0^t \bb{E} \left[\sup_{\tau\in [0,s\wedge \sigma_S]} \norm{\bff{v}(\tau)}{\bb{H}^2}^2 \right] \ds.
\end{align}
Adding~\eqref{equ:E sup loc L2} and \eqref{equ:E sup loc H2}, then applying Gronwall's inequality yields
\begin{equation*}
    \bb{E} \left[\sup_{\tau \in [0,t\wedge \sigma_S]} \norm{\bff{v}(\tau)}{\bb{H}^2}^2 \right]
    +
	\bb{E} \left[ \int_0^{t\wedge \sigma_S} \norm{\Delta^2 \bff{v}(s)}{\bb{L}^2}^2 \ds \right]
    \leq
    C \norm{\bff{v}(0)}{\bb{H}^2}^2 e^{C(1+S^2)T},
\end{equation*}
as required.
\end{proof}

We introduce another stopping time related to the instant parabolic regularisation provided by the equation (cf. Lemma~\ref{lem:H3}). For any $R>0$ and $\bff{v}_0\in \bb{H}^2$, define
\begin{equation}\label{equ:rho R v0}
    \rho_R(\bff{v}_0):= \inf\{t\geq 0: t\norm{\bff{v}(t;\bff{v}_0)}{\bb{H}^3}^2 > R\}.
\end{equation}

\begin{lemma}\label{lem:sigma rho less t}
Let $\bff{v}_0\in \bb{H}^2$ and let $\sigma_S(\bff{v}_0)$ and $\rho_R(\bff{v}_0)$ be as defined in~\eqref{equ:sigma S v0} and~\eqref{equ:rho R v0}, respectively. Then
\begin{align}
\label{equ:P sigma S}
    \bb{P}(\sigma_S(\bff{v}_0) < t)
    &\leq
    \frac{C}{S}(1+e^{Ct}),
    \\
\label{equ:P rho R}
    \bb{P}(\rho_R(\bff{v}_0) < t)
    &\leq
    \frac{C}{R} (1+e^{Ct}),
\end{align}
where $C$ depends on $\norm{\bff{v}_0}{\bb{H}^2}$.
\end{lemma}

\begin{proof}
By Proposition~\ref{pro:E sup un L2}, Proposition~\ref{pro:E sup un H2}, and Markov's inequality, we have
\begin{align*}
    \bb{P}(\sigma_S(\bff{v}_0) < t)
    =
    \bb{P} \left(\sup_{\tau\in [0,t]} \norm{\bff{u}(\tau; \bff{v}_0)}{\bb{H}^2}^2 > S\right)
    \leq
    \frac{1}{S}\, \bb{E} \left[\sup_{\tau\in[0,t]} \norm{\bff{u}(\tau)}{\bb{H}^2}^2 \right] 
    \leq
    \frac{C}{S}(1+e^{Ct}),
\end{align*}
thus~\eqref{equ:P sigma S} is shown.
Similarly, inequality~\eqref{equ:P rho R} follows from Lemma~\ref{lem:t u H3} and Markov's inequality.
\end{proof}

Next, define
\begin{align}\label{equ:sigma S}
    \sigma_S&:= \sigma_S(\bff{u}_0^m) \wedge \sigma_S(\bff{u}_0),
    \\
    \label{equ:rho R}
    \rho_R&:= \rho_R(\bff{u}_0^m) \wedge \rho_R(\bff{u}_0).
\end{align}
We can now show the Feller property of the transition semigroup $P_t$. 

\begin{proposition}\label{pro:feller}
The semigroup $(P_t)$ is Feller on $\bb{H}^2$, that is $P_t\left(C_b(\bb{H}^2)\right)\subset C_b(\bb{H}^2)$ for all $t\ge 0$.
\end{proposition}

\begin{proof}
Let $\epsilon>0$ be given and fix $t>0$ and $\bff{u}_0\in\bb{H}^2$. Let $\bff{u}_0^m\in \bb{H}^2$ be such that $\norm{\bff{u}_0^m-\bff{u}_0}{\bb{H}^2}<1$. With $\sigma_S$ and $\rho_R$ as defined in \eqref{equ:sigma S} and \eqref{equ:rho R} respectively, we write
\begin{align*}
    \abs{P_t\phi(\bff{u}_0^m)-P_t\phi(\bff{u}_0)}
    &=
    \abs{\bb{E}\left[\phi(\bff{u}(t;\bff{u}_0^m))-\phi(\bff{u}(t;\bff{u}_0))\right]} 
    \\
    &\leq
    \abs{\bb{E}\left[\phi(\bff{u}(t;\bff{u}_0^m))-\phi(\bff{u}(t;\bff{u}_0))\right] \bb{I}_{\{\sigma_S<t\}}}
    \\
    &\quad
    +
    \abs{\bb{E}\left[\phi(\bff{u}(t;\bff{u}_0^m))-\phi(\bff{u}(t;\bff{u}_0))\right] \bb{I}_{\{\rho_R<t\}}}
    \\
    &\quad
    +
    \abs{\bb{E}\left[\phi(\bff{u}(t;\bff{u}_0^m))-\phi(\bff{u}(t;\bff{u}_0))\right] \bb{I}_{\{\sigma_S\geq t\}} \bb{I}_{\{\rho_R\geq t\}}}
    =: E_1+E_2+E_3.
\end{align*}
We will estimate each term in the following. To this end, let $\norm{\phi}{\infty}:= \sup_{\bff{v}\in \bb{H}^2} \abs{\phi(\bff{v})}$. By Lemma~\ref{lem:sigma rho less t},
\begin{align*}
    E_1
    &\leq
    2\norm{\bff{\phi}}{\infty} \left(\bb{P}(\sigma_S(\bff{u}_0^m)<t) + \bb{P}(\sigma_S(\bff{u}_0)<t) \right)
    \leq
    \frac{C}{S} \norm{\phi}{\infty} (1+e^{Ct}).
\end{align*}
Similarly,
\begin{align*}
    E_2
    &\leq
    \frac{C}{R} \norm{\phi}{\infty} (1+e^{Ct}),
\end{align*}
where $C$ depends on $\norm{\bff{u}_0}{\bb{H}^2}$. By taking $S$ and $R$ sufficiently large, we can ensure that $E_1+E_2< \epsilon/2$. 

Next, we will estimate $E_3$. Note that on the set $\{\rho_R\geq t\}$, we have $\bff{u}_0^m, \bff{u}_0\in \bff{B}:= \bff{B}_3\left((R/t)^{\frac12}\right)$, where $\bff{B}_3(r)$ is the closed ball in $\bb{H}^3$ with radius $r$ centred at $\bff{0}$. On this ball $\bff{B}$ (which is compact in $\bb{H}^2$ since the embedding $\bb{H}^3\hookrightarrow \bb{H}^2$ is compact), we can approximate the given $\phi\in C_b(\bb{H}^2)$ by a Lipschitz function $\varphi$ with Lipschitz constant $L_\varphi$ such that
\[
    \sup_{\bff{v}\in \bff{B}} \abs{\phi(\bff{v})-\varphi(\bff{v})} < \frac{\epsilon}{8}.
\]
Then by the Kirszbraun theorem \cite{kirszbraun}, $\varphi$ can be extended to a Lipschitz function on $\bb H^2$.
Therefore, by the triangle inequality, Jensen's inequality, and Lemma~\ref{lem:cont dep H2}, we have
\begin{align*}
    E_3
    &\leq
    2 \sup_{\bff{v}\in \bff{B}} \abs{\phi(\bff{v})-\varphi(\bff{v})}
    +
    L_{\varphi}\, \bb{E} \left[\norm{(\bff{u}(t;\bff{u}_0^m))-\bff{u}(t;\bff{u}_0)) \bb{I}_{\{\sigma_S\geq t\}}}{\bb{H}^2} \right]
    \\
    &\leq
    \frac{\epsilon}{4}
    +
    L_\varphi \left(\bb{E} \left[\sup_{s\in [0,t\wedge \sigma_S]} \norm{\bff{u}(s; \bff{u}_0^m)-\bff{u}(s;\bff{u}_0)}{\bb{H}^2}^2 \right]\right)^{\frac12}
    \\
    &\leq
    \frac{\epsilon}{4}
    +
    C_S L_\varphi \norm{\bff{u}_0^m-\bff{u}_0}{\bb{H}^2},
\end{align*}
which can be made less than $\epsilon/2$ whenever $\norm{\bff{u}_0^m-\bff{u}_0}{\bb{H}^2}< \epsilon/(4C_S L_\varphi) =:\delta$. Hence, we have shown that if $\norm{\bff{u}_0^m-\bff{u}_0}{\bb{H}^2}<\delta$, we have $ \abs{P_t\phi(\bff{u}_0^m)-P_t\phi(\bff{u}_0)}<\epsilon$. This completes the proof of the statement.
\end{proof}

Theorem~\ref{the:invariant} is an immediate consequence of the above lemmas, which we prove below.

\begin{proof}[Proof of Theorem~\ref{the:invariant}]
We have shown in Proposition~\ref{pro:feller} that the transition semigroup $\{P_t\}_{t\geq 0}$ is Feller in $\bb{H}^2$. Next, we show tightness on $\bb{H}^2$ of the family of probability measures $\{\mu_T\}_{T\geq 1}$ given by
\[
    \mu_T(\cdot):= \frac{1}{T} \int_0^T P_t(\bff{u}_0, \cdot)\, \dt.
\]
Let $\bff{B}_3(R)$ denote the closed ball in $\bb{H}^3$ with radius $R$ centred at $\bff{0}$. Then, by Markov's inequality and Lemma~\ref{lem:E u H2 t}, we have
\begin{align*}
    \sup_{T\geq 1} \mu_T\left(\bb{H}^2 \setminus \bff{B}_3(R)\right)
    &=
    \sup_{T\geq 1} \frac{1}{T} \int_0^T \bb{P}\big(\{\norm{\bff{u}(s;\bff{u}_0)}{\bb{H}^3} \geq R\}\big)\, \ds
    \\
    &\leq
    \sup_{T\geq 1} \frac{1}{TR} \int_0^T \bb{E}\left[\norm{\bff{u}(s)}{\bb{H}^3}\right] \ds
    \\
    &\leq
    \sup_{T\geq 1} \frac{C(1+T)}{TR}
    \leq 
    \frac{2C}{R} \to 0 \; \text{ as } R\to\infty.
\end{align*}
Since the embedding $\bb{H}^3\hookrightarrow \bb{H}^2$ is compact, the family of probability measures $\{\mu_T\}_{T\geq 1}$ is tight on $\bb{H}^2$.
This yields the existence of at least one invariant measure by the Krylov--Bogoliubov theorem.

We now show that such an invariant measure $\mu$ is supported on $\bb{H}^3$. By the invariance property of $\mu$, for any $T>0$ and $\phi\in C_b(\bb{H}^2)$ we have
\begin{equation}\label{equ:ex inv}
    \int_{\bb{H}^2} \phi(\bff{u}_0)\,\mathrm{d}\mu(\bff{u}_0)
    =
    \int_{\bb{H}^2}\int_{\bb{H}^2} \frac{1}{T} \int_0^T P_t(\bff{u}_0, \mathrm{d}\bff{u}) \phi(\bff{u})\, \dt\,\mathrm{d}\mu(\bff{u}_0).
\end{equation}
For any $n\in\bb{N}$, $R>0$, and $\bff{v}\in \bb{H}^2$, let
\[
    \Psi_{n,R}(\bff{v}):= \norm{\nabla \Delta \Pi_n \bff{v}}{\bb{L}^2} \wedge R,
\]
where $\Pi_n$ is the projection operator defined in~\eqref{equ:proj}.
Note that for any $\bff{u}_0\in \bff{B}_2(\alpha)$, the ball of radius $\alpha$ about the origin in $\bb{H}^2$, by Lemma~\ref{lem:E u H2 t} we have
\begin{align*}
    \left| \frac{1}{T} \int_0^T \int_{\bb{H}^2} P_t(\bff{u}_0,\mathrm{d}\bff{u}) \Psi_{n,R}(\bff{u})\,\dt \right|
    =
	\left|\frac{1}{T} \int_0^T \bb{E}\Psi_{n,R}(\bff{u}(t;\bff{u}_0))\, \dt\right| 
	\leq
	C\left(1+\frac{1+\norm{\bff{u}_0}{\bb{H}^2}^2}{T}\right)
	\leq
	C\left(1+\frac{\alpha^2}{T}\right).
\end{align*}
Then by~\eqref{equ:ex inv} and the above inequality, we have
\begin{align*}
	\int_{\bb{H}^2} \Psi_{n,R}(\bff{u}_0) \, \mathrm{d}\mu(\bff{u}_0)
	&\leq
	\int_{\bb{H}^2} \left|\frac{1}{T} \int_0^T \int_{\bb{H}^2} P_t(\bff{u}_0,\mathrm{d}\bff{u}) \Psi_{n,R}(\bff{u}) \,\dt\right| \mathrm{d}\mu(\bff{u}_0)
	\\
	&=
	\int_{\bff{B}_2(\alpha)} \left|\frac{1}{T} \int_0^T \bb{E}\Psi_{n,R}(\bff{u}(t;\bff{u}_0))\, \dt\right| \mathrm{d}\mu(\bff{u}_0)
    \\
    &\quad 
	+
	\int_{\bb{H}^2 \setminus \bff{B}_2(\alpha)} \left|\frac{1}{T} \int_0^T \bb{E}\Psi_{n,R}(\bff{u}(t;\bff{u}_0))\, \dt\right|  \mathrm{d}\mu(\bff{u}_0)
	\\
	&\leq
	C(1+\alpha^2 T^{-1})\, \mu(\bff{B}_2(\alpha)) + R\, \mu(\bb{H}^2\setminus \bff{B}_2(\alpha)),
\end{align*}
where $C$ is a constant which does not depend on $T$, $\alpha$, and $R$. First, we take $\alpha$ to be sufficiently large (depending on $R$) so that $R\, \mu(\bb{H}^2\setminus \bff{B}_2(\alpha))\leq 1$. Then we choose $T$ sufficiently large (depending on $\alpha$), so that $\alpha^2 T^{-1}\leq 1$. This implies
\[
    \int_{\bb{H}^2} \Psi_{n,R}(\bff{u}_0) \, \mathrm{d}\mu(\bff{u}_0) \leq C,
\]
where $C$ is independent of $n$ and $R$. Letting $n\to\infty$, by Fatou's lemma we have
\[
	\int_{\bb{H}^2} \left(\norm{\nabla \Delta \bff{u}_0}{\bb{L}^2} \wedge R\right) \mathrm{d}\mu(\bff{u}_0) \leq C,
\]
where $C$ is independent of $R$. By the monotone convergence theorem, taking $R\uparrow \infty$, we obtain
\[
	\int_{\bb{H}^2} \norm{\nabla \Delta \bff{u}_0}{\bb{L}^2} \mathrm{d}\mu(\bff{u}_0) < \infty.
\]
In particular, this implies that $\mu$ is supported on $\bb{H}^3$. 

The existence of an ergodic invariant measure follows by standard arguments. The set of all invariant measures $K$ considered in the space of probability measures on $\bb H^2$ is convex and compact, hence a convex hull of the set of its extremal points by the Krein-Milman theorem. Then, by Proposition 11.12 in \cite{DaZab14} every extremal point of $K$ is an ergodic invariant measure.
This completes the proof of the theorem.
\end{proof}

\section{Proof of Theorem~\ref{the:stab above}: Exponential Stability above the Curie Temperature}\label{sec:stab above}

We shall consider a special case of the stochastic LLBar equation~\eqref{equ:sllbar}, namely one perturbed by a random precession (multiplicative noise) \emph{above} the Curie temperature in the absence of the spin current. In other words, we assume $\kappa_1=-1$, $\bff{g}_k\equiv \bff{0}$, and $S(\bff{u})\equiv \bff{0}$ in this section. We also assume that $\lambda_r>\lambda_e$. In this case, we have an exponential stability result, which implies the uniqueness of invariant measure under certain conditions. To this end, we start with the following lemmas.

\begin{lemma}\label{lem:exp L2}
Suppose that $\lambda_r>\frac12 \gamma^2 \sigma_h^2$ and let $\mu:= \lambda_r- \frac12 \gamma^2 \sigma_h^2$. Let $\bff{u}$ be the solution of \eqref{equ:sllbar} given by Theorem~\ref{the:exist}, corresponding to the case $\kappa_1=-1$, $\kappa_2=1$, $\bff{g}_k\equiv \bff{0}$, and $S(\bff{u})\equiv \bff{0}$. Then $\bb{P}$-a.s. for all $t> 0$,
\begin{equation}\label{equ:u L2 exp}
	\norm{\bff{u}(t)}{\bb{L}^2}^2 \leq e^{-\mu t} \norm{\bff{u}_0}{\bb{L}^2}^2,
\end{equation}
Furthermore,
\begin{equation}\label{equ:K0}
	\int_0^t \norm{\bff{u}(s)}{\bb{H}^2}^2 \ds 
	+
	\int_0^t \norm{\bff{u}(s)}{\bb{L}^4}^4 \ds
	\leq 
	K_0 \norm{\bff{u}_0}{\bb{L}^2}^2,
\end{equation} 
where $K_0$ is a constant independent of $t$.
\end{lemma}

\begin{proof}
Applying It\^o's lemma to the process $\frac12 e^{\mu t} \norm{\bff{u}(t)}{\bb{L}^2}^2$, we obtain
\begin{align*}
	\frac{1}{2} \dif \left(e^{\mu t} \norm{\bff{u}(t)}{\bb{L}^2}^2 \right)
	-
	\frac12 \mu e^{\mu t} \norm{\bff{u}(t)}{\bb{L}^2}^2
	&=
	\frac12 e^{\mu t} \Big( -(\lambda_r+\lambda_e)\norm{\nabla \bff{u}}{\bb{L}^2}^2 -\lambda_e \norm{\Delta \bff{u}}{\bb{L}^2}^2 -\lambda_r \norm{\bff{u}}{\bb{L}^2}^2 - \lambda_r \norm{\bff{u}}{\bb{L}^4}^4 
	\\
	&\quad
	- \lambda_e \norm{|\bff{u}||\nabla \bff{u}|}{\bb{L}^2}^2 - 2\lambda_e \norm{\bff{u}\cdot \nabla \bff{u}}{\bb{L}^2}^2
	+
	\frac{\gamma^2}{2} \sum_{k=1}^n \norm{\bff{u}\times \bff{h}_k}{\bb{L}^2}^2
	\Big) \,\dt.
\end{align*}
Integrating and applying H\"older's inequality, we have
\begin{align*}
	&\frac{1}{2} e^{\mu t} \norm{\bff{u}(t)}{\bb{L}^2}^2
	+
	\frac12 (\lambda_r+\lambda_e) \int_0^t e^{\mu s} \norm{\nabla \bff{u}(s)}{\bb{L}^2}^2 \ds
	+
	\frac12 \lambda_e \int_0^t e^{\mu s} \norm{\Delta \bff{u}(s)}{\bb{L}^2}^2 \ds
	+
	\frac12 \lambda_r \int_0^t e^{\mu s} \norm{\bff{u}(s)}{\bb{L}^2}^2 \ds
	\\
	&\quad
	+
	\frac12 \lambda_r \int_0^t e^{\mu s} \norm{\bff{u}(s)}{\bb{L}^4}^4 \ds
	+
	\frac12 \lambda_e \int_0^t e^{\mu s} \norm{|\bff{u}(s)|\, |\nabla \bff{u}(s)|}{\bb{L}^2}^2 \ds
	+
	\lambda_e \int_0^t e^{\mu s} \norm{\bff{u}(s) \cdot \nabla \bff{u}(s)}{\bb{L}^2}^2 \ds
	\\
	&\leq
	\frac12 \norm{\bff{u}_0}{\bb{L}^2}^2
	+
	\frac14 \gamma^2 \sigma_h^2 \int_0^t e^{\mu s} \norm{\bff{u}(s)}{\bb{L}^2}^2 \ds.
\end{align*}
Noting $\mu>0$, we obtain the required result after rearranging the terms.
\end{proof}

\begin{lemma}
Under the same assumptions as in Lemma~\ref{lem:exp L2}, there exists a constant $K_1$ such that for any $t>0$,
\begin{align}\label{equ:K1}
	\bb{E} \left[\sup_{\tau \in [0,t]} \left(\norm{\nabla \bff{u}(\tau)}{\bb{L}^2}^{2} + \norm{\bff{u}(\tau)}{\bb{L}^4}^{4} \right) \right]
	&+
	\bb{E} \left[\int_0^t \norm{\bff{H}(s)}{\bb{H}^1}^2 \ds \right]
	\leq 
	K_1 \norm{\bff{u}_0}{\bb{H}^1}^4.
\end{align}
The constant $K_1$ is independent of $t$.
\end{lemma}

\begin{proof}
The proof follows the same argument as that of Proposition~\ref{pro:E sup nab un L2}, but with $\bff{g}_k\equiv \bff{0}$ and $S(\bff{u})\equiv \bff{0}$.
\end{proof}

Under a slightly stronger dissipativity condition, we have the following exponential moment bound.

\begin{lemma}\label{lem:Ke}
Let $\delta>0$ and suppose that $\lambda_r \geq 3\gamma^2 \sigma_h^2$. Then for any $t>0$,
\begin{equation}\label{equ:Ke}
	\bb{E}\left[\exp\left(\frac12 \delta \norm{\bff{u}(t)}{\bb{H}^1}^2 + \frac14 \delta \norm{\bff{u}(t)}{\bb{L}^4}^4\right) \right]
	\leq
	\exp\left(\frac12 \delta \norm{\bff{u}_0}{\bb{H}^1}^2 + \frac14 \delta \norm{\bff{u}_0}{\bb{L}^4}^4 \right).
\end{equation}
\end{lemma}

\begin{proof}
Let $\psi(t):= \frac12 \norm{\nabla \bff{u}(t)}{\bb{L}^2}^2 + \frac14 \norm{\bff{u}(t)}{\bb{L}^4}^4 + \frac12 \norm{\bff{u}(t)}{\bb{L}^2}^2$. Applying It\^o's lemma to the process $\exp(\delta \psi(t))$, we have
\begin{align*}
	&e^{\delta \psi(t)} + \lambda_r \delta \int_0^t e^{\delta \psi(s)} \norm{\bff{H}(s)}{\bb{L}^2}^2 \ds + \lambda_e \delta \int_0^t e^{\delta \psi(s)} \norm{\nabla \bff{H}(s)}{\bb{L}^2}^2
	\\
	&=
	e^{\delta \psi(0)}
	+
	\frac{\delta\gamma^2}{2} \sum_{k=1}^\infty \int_0^t e^{\delta \psi(s)} \norm{\bff{u}(s) \times \bff{h}_k}{\bb{L}^2}^2 \ds 
	+
	\frac{\delta\gamma^2}{2} \sum_{k=1}^\infty \int_0^t e^{\delta \psi(s)} \norm{\nabla (\bff{u}(s) \times \bff{h}_k)}{\bb{L}^2}^2 \ds
	\\
	&\quad 
	+
	\frac{3\delta\gamma^2}{2} \sum_{k=1}^\infty \int_0^t e^{\delta \psi(s)} \norm{\abs{\bff{u}(s)} \abs{\bff{u}(s)\times \bff{h}_k}}{\bb{L}^2}^2 \ds
	+
	\frac{\delta\gamma^2}{2} \sum_{k=1}^\infty \int_0^t e^{\delta \psi(s)} \abs{\inpro{\bff{u}(s)\times \bff{h}_k}{\bff{H}(s)}_{\bb{L}^2}}^2 \ds
	\\
	&\quad
	-
	\gamma \sum_{k=1}^\infty \int_0^t e^{\delta \psi(s)} \inpro{\bff{u}(s)\times \bff{h}_k}{\bff{H}(s)}_{\bb{L}^2} \dW_k(s).
\end{align*}
For any $R>0$, define the stopping time $\tau_R:= \inf\left\{t\geq 0: \psi(t)\geq R \right\}$. Therefore, integrating over $[0, t\wedge\tau_R]$ and taking the expectation, noting that $\bff{H}=\Delta \bff{u}-\bff{u}-\abs{\bff{u}}^2 \bff{u}$, following the same argument as in the proof of Proposition~\ref{pro:E sup nab un L2} and Lemma~\ref{lem:E u H2 t}, we obtain
\begin{align*}
	&\bb{E} \left[e^{\delta \psi(t\wedge \tau_R)} \right]
	+
	\bb{E} \left[\lambda_r \delta \int_0^t e^{\delta \psi(s\wedge \tau_R)} \norm{\bff{H}(s\wedge \tau_R)}{\bb{L}^2}^2 \ds \right]
	+
	\bb{E} \left[\lambda_e \delta \int_0^t e^{\delta \psi(s\wedge \tau_R)} \norm{\nabla \bff{H}(s\wedge \tau_R)}{\bb{L}^2}^2 \ds \right]
	\\
	&\leq
	e^{\delta \psi(0)}
	+
	\frac52 \delta\gamma^2 \sigma_h^2 \bb{E} \left[\int_0^t e^{\delta \psi(s\wedge \tau_R)} \left(\norm{\bff{u}(s\wedge\tau_R)}{\bb{H}^1}^2 + \norm{\bff{u}(s\wedge\tau_R)}{\bb{L}^4}^4 \right) \ds \right] 
	\\
	&\quad
	+
	\frac12 \delta \gamma^2 \sigma_h^2 \bb{E} \left[ \int_0^t e^{\delta \psi(s\wedge\tau_R)} \norm{\abs{\bff{u}(s\wedge \tau_R)} \abs{\nabla \bff{u}(s\wedge\tau_R)}}{\bb{L}^2}^2 \ds \right].
\end{align*}
Now, note that
\[
\norm{\bff{H}}{\bb{L}^2}^2= \norm{\Delta \bff{u}}{\bb{L}^2}^2 + 2\norm{\nabla \bff{u}}{\bb{L}^2}^2 + 2\norm{\abs{\bff{u}} \abs{\nabla \bff{u}}}{\bb{L}^2}^2 + 4\norm{\bff{u}\cdot\nabla \bff{u}}{\bb{L}^2}^2 + \norm{\bff{u}}{\bb{L}^2}^2 + 2\norm{\bff{u}}{\bb{L}^4}^4 + \norm{\bff{u}}{\bb{L}^6}^6.
\]
Thus, applying Young's inequality, we deduce that
\begin{align*} 
	&\bb{E} \left[e^{\delta \psi(t\wedge \tau_R)} \right]
	+
	\lambda_r \delta \bb{E} \left[\int_0^t e^{\delta \psi(s\wedge \tau_R)} \norm{\bff{H}(s\wedge \tau_R)}{\bb{L}^2}^2 \ds \right]
	+
	\lambda_e \delta \bb{E} \left[\int_0^t e^{\delta \psi(s\wedge \tau_R)} \norm{\nabla \bff{H}(s\wedge \tau_R)}{\bb{L}^2}^2 \ds \right]
	\\
	&\leq
	e^{\delta \psi(0)}
	+
	3\delta\gamma^2 \sigma_h^2 \bb{E} \left[\int_0^t e^{\delta \psi(s\wedge \tau_R)} \norm{\bff{H}(s\wedge \tau_R)}{\bb{L}^2}^2 \ds \right].
\end{align*}
If $\lambda_r\geq 3\gamma^2 \sigma_h^2$, then we can absorb the last term to the left-hand side, which implies
\begin{align*}
	\bb{E} \left[e^{\delta \psi(t\wedge \tau_R)} \right]
	\leq
	e^{\delta \psi(0)}.
\end{align*}
By Fatou's lemma and continuity in time of $\psi$, taking $R\to \infty$, we have the required inequality.
\end{proof}

We can now prove Theorem~\ref{the:stab above}.

\begin{proof}[Proof of Theorem~\ref{the:stab above}]
Firstly, inequality~\eqref{equ:u L2 decay} is shown in Lemma~\ref{lem:exp L2}.

Now, let $\bff{u}_1$ and $\bff{u}_2$ be solutions corresponding to initial data $\bff{u}_{0,1}$ and $\bff{u}_{0,2}\in \bb{H}^2$, respectively. We will show an exponential stability result, namely there are constants $C$ and $\alpha$ independent of $t$ such that
\begin{equation}\label{equ:E u0 u1}
	\bb{E}\left[\norm{\bff{u}_1(t)-\bff{u}_2(t)}{\bb{L}^2}^2 \right]
	\leq
	Ce^{-\alpha t} \norm{\bff{u}_{0,1}-\bff{u}_{0,2}}{\bb{L}^2}^2,
\end{equation}
which implies the uniqueness of invariant measure.

To this end, let $\bff{v}:= \bff{u}_1-\bff{u}_2$. By It\^o's lemma applied to $\frac12 \norm{\bff{v}(t)}{\bb{L}^2}^2$, we have
\begin{align}\label{equ:I vt above}
	\nonumber
	&\frac{1}{2} \norm{\bff{v}(t)}{\bb{L}^2}^2
	+
	\lambda_e \int_0^t \norm{\Delta \bff{v}(s)}{\bb{L}^2}^2 \ds 
	+
	(\lambda_r - \lambda_e) \int_0^t \norm{\nabla \bff{v}(s)}{\bb{L}^2}^2 \ds 
	+
	\lambda_r \int_0^t \norm{\bff{v}(s)}{\bb{L}^2}^2 \ds 
	\\
	\nonumber
	&\quad
	+
	\lambda_r \int_0^t \norm{|\bff{u}_1(s)|\, |\bff{v}(s)|}{\bb{L}^2}^2 \ds 
	+
	\lambda_r \int_0^t \norm{|\bff{u}_2(s)|\, |\bff{v}(s)|}{\bb{L}^2}^2 \ds
	\\
	\nonumber
	&\leq
	\frac12 \norm{\bff{u}_{0,1}-\bff{u}_{0,2}}{\bb{L}^2}^2
	+
	\lambda_r \int_0^t \inpro{(\bff{u}_1(s)\cdot \bff{v}(s))\bff{u}_2(s)}{\bff{v}(s)}_{\bb{L}^2} \ds 
	+
	\lambda_e \int_0^t \inpro{|\bff{u}_1(s)|^2 \bff{v}(s)}{\Delta \bff{v}(s)}_{\bb{L}^2} \ds
	\\
	\nonumber
	&\quad 
	+
	\lambda_e \int_0^t \inpro{\big((\bff{u}_1(s)+\bff{u}_2(s))\cdot \bff{v}(s) \big) \bff{u}_2(s)}{\Delta \bff{v}(s)}_{\bb{L}^2} \ds
	-
	\gamma \int_0^t \inpro{\bff{u}_2(s) \times \Delta \bff{v}(s)}{\bff{v}(s)}_{\bb{L}^2} \ds 
	\\
	\nonumber
	&\quad
	+
	\frac{1}{2} \sum_{k=1}^\infty \int_0^t \norm{G_k(\bff{u}_1)- G_k(\bff{u}_2)}{\bb{L}^2}^2 \ds
	\\
	&=:
	\frac12 \norm{\bff{u}_{0,1}-\bff{u}_{0,2}}{\bb{L}^2}^2 +J_1+J_2+J_3+J_4+J_5.
\end{align}
For the first term on the right-hand side, we have by Young's inequality,
\begin{align*}
	J_1 \leq
	\frac{\lambda_r}{2} \int_0^t \norm{|\bff{u}_1(s)|\, |\bff{v}(s)|}{\bb{L}^2}^2 \ds 
	+
	\frac{\lambda_r}{2} \int_0^t \norm{|\bff{u}_1(s)|\, |\bff{v}(s)|}{\bb{L}^2}^2 \ds.
\end{align*}
For the remaining terms, similarly we have
\begin{align*}
	J_2 
	&\leq
	4\lambda_e \int_0^t \norm{\bff{u}_1(s)}{\bb{L}^\infty}^4 \norm{\bff{v}(s)}{\bb{L}^2}^2 \ds
	+
	\frac{\lambda_e}{4} \int_0^t \norm{\Delta \bff{v}(s)}{\bb{L}^2}^2 \ds,
	\\
	J_3
	&\leq
	8\lambda_e \int_0^t \left(\norm{\bff{u}_1(s)}{\bb{L}^\infty}^4 + \norm{\bff{u}_2(s)}{\bb{L}^\infty}^4 \right) \norm{\bff{v}(s)}{\bb{L}^2}^2 \ds
	+
	\frac{\lambda_e}{4} \int_0^t \norm{\Delta \bff{v}(s)}{\bb{L}^2}^2 \ds,
	\\
	J_4
	&\leq
	\frac{\gamma^2}{\lambda_e} \int_0^t \norm{\bff{u}_2(s)}{\bb{L}^\infty}^2 \norm{\bff{v}(s)}{\bb{L}^2}^2 \ds 
	+
	\frac{\lambda_e}{4} \int_0^t \norm{\Delta \bff{v}(s)}{\bb{L}^2}^2 \ds,
	\\
	J_5
	&\leq
	\sigma_h \gamma^2 \int_0^t \norm{\bff{v}(s)}{\bb{L}^2}^2 \ds.
\end{align*}
Thus, from~\eqref{equ:I vt above} we obtain
\begin{equation}\label{equ:vt exp}
	\norm{\bff{v}(t)}{\bb{L}^2}^2
	\leq
	\norm{\bff{u}_{0,1}-\bff{u}_{0,2}}{\bb{L}^2}^2\,
	\exp \left(-2\lambda_r t + R(t) \right),
\end{equation}
where
\begin{equation}\label{equ:Rt}
	R(t):= 12\lambda_e \int_0^t \norm{\bff{u}_1(s)}{\bb{L}^\infty}^4 \ds 
	+ 8 \lambda_e \int_0^t \norm{\bff{u}_2(s)}{\bb{L}^\infty}^4 \ds
	+
	\frac{\gamma^2}{\lambda_e} \int_0^t \norm{\bff{u}_2(s)}{\bb{L}^\infty}^2 \ds 
	+
	\sigma_h \gamma^2.
\end{equation}
We will estimate $R(t)$ defined above.

\medskip
\noindent
\underline{Case 1 ($d=1$)}: In this case, by Agmon's inequality we have
\begin{equation*}
	R(t)
	\leq
	12\lambda_e C_{\mathrm{A}} K_0 \norm{\bff{u}_{0,1}}{\bb{L}^2}^4
	+
	8\lambda_e C_{\mathrm{A}} K_0 \norm{\bff{u}_{0,2}}{\bb{L}^2}^4
	+
	\frac{\gamma^2}{\lambda_e} C_{\mathrm{S}} K_0 \norm{\bff{u}_{0,2}}{\bb{L}^2}^2
	+
	\sigma_h \gamma^2
	=: \Gamma_1,
\end{equation*}
where $K_0$ is the constant in~\eqref{equ:K0}, $C_{\mathrm{S}}$ is the constant in the Sobolev embedding $\bb{H}^1\hookrightarrow \bb{L}^\infty$, and $C_{\mathrm{A}}$ is the constant in Agmon's inequality $\norm{\bff{v}}{\bb{L}^\infty}^2 \leq C_{\mathrm{A}} \norm{\bff{v}}{\bb{L}^2} \norm{\bff{v}}{\bb{H}^1}$.
Therefore, from~\eqref{equ:vt exp} we obtain
\[
	\bb{E}\left[\norm{\bff{u}_1-\bff{u}_2}{\bb{L}^2}^2 \right]
	\leq
	e^{\Gamma_1} e^{-2\lambda_r t} \norm{\bff{u}_{0,1}-\bff{u}_{0,2}}{\bb{L}^2}^2.
\]
Here, $\Gamma_1$ is independent of $t$.

\medskip
\noindent
\underline{Case 2 ($d=2$)}: Similarly, by Agmon's inequality we have from~\eqref{equ:Rt},
\begin{equation*}
	R(t)
	\leq
	12\lambda_e C_{\mathrm{A}}' K_0 \norm{\bff{u}_{0,1}}{\bb{L}^2}^4
	+
	8\lambda_e C_{\mathrm{A}}' K_0 \norm{\bff{u}_{0,2}}{\bb{L}^2}^4
	+
	\frac{\gamma^2}{\lambda_e} C_{\mathrm{S}}' K_0 \norm{\bff{u}_{0,2}}{\bb{L}^2}^2
	+
	\sigma_h \gamma^2
	=: \Gamma_2,
\end{equation*}
where $K_0$ is the constant in~\eqref{equ:K1}, $C_{\mathrm{S}}'$ is the constant in the Sobolev embedding $\bb{H}^2\hookrightarrow \bb{L}^\infty$, and $C_{\mathrm{A}}'$ is the constant in Agmon's inequality $\norm{\bff{v}}{\bb{L}^\infty}^2 \leq C_{\mathrm{A}}' \norm{\bff{v}}{\bb{L}^2} \norm{\bff{v}}{\bb{H}^2}$.
Therefore, from~\eqref{equ:vt exp} we obtain
\[
\bb{E}\left[\norm{\bff{u}_1-\bff{u}_2}{\bb{L}^2}^2 \right]
\leq
e^{\Gamma_2} e^{-2\lambda_r t} \norm{\bff{u}_{0,1}-\bff{u}_{0,2}}{\bb{L}^2}^2.
\]
Here, $\Gamma_2$ is independent of $t$.

\medskip
\noindent
\underline{Case 3 ($d=3$)}: In this case, we cannot bound $R(t)$ pathwise as in previous cases. Instead, taking expectation on both sides of~\eqref{equ:vt exp}, we have
\begin{align}\label{equ:E exp Rt}
	\bb{E}\left[\norm{\bff{v}(t)}{\bb{L}^2}^2 \right]
	\leq
	\norm{\bff{u}_{0,1}-\bff{u}_{0,2}}{\bb{L}^2}^2\,
	e^{-2\lambda_r t}\, \bb{E}\left[\exp(R(t))\right].
\end{align}
Note that by Agmon's inequality and~\eqref{equ:K0},
\begin{align*}
	R(t) 
	&\leq
	12\lambda_e C_{\mathrm{A}}'' K_0 \left(\sup_{\tau\in [0,t]} \norm{\bff{u}_1(\tau)}{\bb{H}^1}^2\right) \norm{\bff{u}_{0,1}}{\bb{L}^2}^2
	+
	8\lambda_e C_{\mathrm{A}}'' K_0 \left(\sup_{\tau\in [0,t]} \norm{\bff{u}_1(\tau)}{\bb{H}^1}^2\right) \norm{\bff{u}_{0,2}}{\bb{L}^2}^2
	\\
	&\quad
	+
	\frac{\gamma^2}{\lambda_e} C_{\mathrm{S}}' K_0 \norm{\bff{u}_{0,2}}{\bb{L}^2}^2
	+
	\sigma_h \gamma^2.
\end{align*}
Therefore, under the assumption that $\lambda_r \geq 3\gamma^2 \sigma_h^2$, applying Lemma~\ref{lem:Ke} to~\eqref{equ:E exp Rt} yields
\begin{align}\label{equ:K delta}
\bb{E}\left[\norm{\bff{u}_1-\bff{u}_2}{\bb{L}^2}^2 \right]
	\leq
	e^{\Gamma_3} e^{-2\lambda_r t} \norm{\bff{u}_{0,1}-\bff{u}_{0,2}}{\bb{L}^2}^2,
\end{align}
where $\Gamma_3:= 24\lambda_e C_{\mathrm{A}}'' K_0 \left(\norm{\bff{u}_{0,1}}{\bb{H}^1}^6+ \norm{\bff{u}_{0,2}}{\bb{H}^1}^6\right)$ is obtained by applying \eqref{equ:Ke}. The constant $C_{\mathrm{A}}''$ is the constant in Agmon's inequality $\norm{\bff{v}}{\bb{L}^\infty}^2 \leq C_{\mathrm{A}}'' \norm{\bff{v}}{\bb{H}^1} \norm{\bff{v}}{\bb{H}^2}$. Thus, we have a bound of the form~\eqref{equ:E u0 u1} when $\lambda_r \geq 3\gamma^2 \sigma_h^2$.
 
This completes the proof of the theorem.
\end{proof}

\section{Relationship between the Stochastic LLBar and the Stochastic LLB Equations}

Recall that the stochastic LLB equation is~\eqref{equ:sllbar} with $\lambda_e=0$ and $\kappa_1=-1$. The existence of a strong solution for the stochastic LLB equation is obtained in~\cite{BrzGolLe20} for $d=1,2$. In the following theorem, we show the convergence of the strong solution of the stochastic LLBar equation (with $\kappa_1<0$) to that of the stochastic LLB equation.

\begin{theorem}\label{the:llbar llb}
For $d\in \{1,2\}$, let $(\bff{u}^\varepsilon, \bff{H}^\varepsilon)$ be a strong solution of~\eqref{equ:sllbar} with $\lambda_e=\varepsilon$ (where $\varepsilon<\lambda_r$) and initial data $\bff{u}_0\in \bb{H}^2$. Let $(\bff{u},\bff{H})$ be a strong solution of the stochastic LLB equation with the same initial data $\bff{u}_0$. Then
\begin{align*}
	\lim_{\varepsilon\to 0^+} \bb{E}\left[\sup_{r\in [0,t]} \norm{\bff{u}^\varepsilon(r)-\bff{u}(r)}{\bb{L}^2}^2\right] 
	+
	\bb{E} \left[\int_0^{t} \norm{\bff{u}^\varepsilon(s)-\bff{u}(s)}{\bb{H}^1}^2 \ds\right]
	=0.
\end{align*}
\end{theorem}

\begin{proof}
Let $\bff{v}:= \bff{u}^\varepsilon-\bff{u}$ and $\bff{B}:=\bff{H}^\varepsilon-\bff{H}$. Then $\bff{v}$ satisfies the equation
\begin{align*}
		\dif \bff{v}
		&=
		\big( \lambda_r \bff{B}- \varepsilon \Delta \bff{H}^\varepsilon- \gamma(\bff{u}^\varepsilon \times \bff{B} +\bff{v}\times \bff{H})
		+
		R(\bff{u}^\varepsilon) - R(\bff{u}) 
		+
		L(\bff{u}^\varepsilon) - L(\bff{u}) \big) \, \dt
		\\
		&\quad
		+
		\left( \sum_{k=1}^\infty G_k(\bff{u}^\varepsilon)- G_k(\bff{u}) \right) \dif W_k(t)
\end{align*}
with $\bff{v}(0)=\bff{0}$. Applying It\^o's lemma, we have
\begin{align*}
	\frac{1}{2} \dif \norm{\bff{v}}{\bb{L}^2}^2
	&=
	\Big( \lambda_r \inpro{\bff{B}}{\bff{v}}_{\bb{L}^2}
	+
	\varepsilon \inpro{\nabla \bff{H}^\varepsilon}{\nabla \bff{v}}_{\bb{L}^2}
	-
	\gamma\inpro{\bff{u}^\varepsilon \times \bff{B}}{\bff{v}}_{\bb{L}^2}
	\\
	&\qquad 
	+
	\inpro{R(\bff{u}^\varepsilon) - R(\bff{u})}{\bff{v}}_{\bb{L}^2}
	+
	\inpro{L(\bff{u}^\varepsilon)- L(\bff{u})}{\bff{v}}_{\bb{L}^2}
	+
	\frac{1}{2} \sum_{k=1}^\infty \norm{G_k(\bff{u}^\varepsilon)- G_k(\bff{u})}{\bb{L}^2}^2
	\Big) \,\dt.
\end{align*}
Integrating this with respect to $t$, noting~\eqref{equ:lambda B v} and~\eqref{equ:gamma u B}, we obtain
\begin{align}\label{equ:dv llb}
	\nonumber
	&\frac{1}{2} \norm{\bff{v}(t)}{\bb{L}^2}^2
	+
	\lambda_r \int_0^t \norm{\bff{v}(s)}{\bb{H}^1}^2 \ds 
	+
	\lambda_r \int_0^t \norm{|\bff{u}^\varepsilon(s)|\, |\bff{v}(s)|}{\bb{L}^2}^2 \ds
	+
	\lambda_r \int_0^t \norm{\bff{u}(s)\cdot \bff{v}(s)}{\bb{L}^2}^2 \ds
	\\
	\nonumber
	&\leq
	\varepsilon \int_0^t \inpro{\nabla \bff{H}^\varepsilon(s)}{\nabla \bff{v}(s)}_{\bb{L}^2}\ds
	+
	\lambda_r \int_0^t \inpro{(\bff{u}^\varepsilon(s)\cdot \bff{v}(s))\bff{u}(s)}{\bff{v}(s)}_{\bb{L}^2} \ds 
	+
	\gamma \int_0^t \inpro{\bff{v}(s) \times \nabla \bff{u}^\varepsilon(s)}{\nabla\bff{v}(s)}_{\bb{L}^2} \ds 
	\\
	\nonumber
	&\quad
	+
	\int_0^t \inpro{R(\bff{u}^\varepsilon) - R(\bff{u})}{\bff{v}}_{\bb{L}^2} \ds
	+
	\int_0^t \inpro{L(\bff{u}^\varepsilon)- L(\bff{u})}{\bff{v}}_{\bb{L}^2} \ds
	+
	\frac{1}{2} \sum_{k=1}^\infty \int_0^t \norm{G_k(\bff{u}^\varepsilon)- G_k(\bff{u})}{\bb{L}^2}^2 \ds
	\\
	&=:
	I_1(t)+I_2(t)+I_3(t)+I_4(t)+I_5(t)+I_6(t).
\end{align}
Now for any $R>0$, define the stopping time
\begin{equation}\label{equ:tau R llb}
	\tau_R := \inf \left\{t\geq 0: \norm{\bff{u}^\varepsilon(t)}{\bb{H}^1}^2 \vee \int_0^t \norm{\bff{u}^\varepsilon(s)}{\bb{H}^2}^4 \ds
	> R \right\}.
\end{equation}
We will derive bounds for $I_j(t\wedge \tau_R)$, where $j=1,2,\ldots,6$, by using H\"older's and Young's inequalities and the Sobolev embedding as follows. For the terms besides $I_3$ and $I_4$, we have
\begin{align}
	\label{equ:I1 llb}
	I_1(t\wedge\tau_R)
	&\leq
	\frac{\varepsilon^2}{\lambda_r} \int_0^{t\wedge\tau_R} \norm{\nabla \bff{H}^\varepsilon(s)}{\bb{L}^2}^2 \ds
	+
	\frac{\lambda_r}{8} \int_0^{t\wedge\tau_R} \norm{\nabla \bff{v}(s)}{\bb{L}^2}^2 \ds,
	\\
	\label{equ:I2 llb}
	I_2(t\wedge\tau_R)
	&\leq
	\frac{\lambda_r}{2} \int_0^{t\wedge\tau_R} \norm{|\bff{u}^\varepsilon(s)|\, |\bff{v}(s)|}{\bb{L}^2}^2 \ds
	+
	\frac{\lambda_r}{2} \int_0^{t\wedge\tau_R} \norm{\bff{u}(s)\cdot \bff{v}(s)}{\bb{L}^2}^2 \ds,
	\\
	\label{equ:I5 llb}
	I_5(t\wedge\tau_R)
	&\leq
	C \int_0^{t\wedge\tau_R} \norm{\bff{v}(s)}{\bb{L}^2}^2 \ds,
	\\
	\label{equ:I6 llb}
	I_6(t\wedge\tau_R)
	&\leq
	\sigma_h\gamma^2 \int_0^{t\wedge\tau_R} \norm{\bff{v}(s)}{\bb{L}^2}^2 \ds.
\end{align}
For the term $I_3$, we use Lemma~\ref{lem:tec lem} to obtain
\begin{align}\label{equ:I3 llb}
	I_3(t\wedge\tau_R) 
	\leq 
	C \int_0^{t\wedge\tau_R} \Phi\big(\bff{u}^\varepsilon(s)\big) \norm{\bff{v}(s)}{\bb{L}^2}^2 \ds
	+
	\frac{\lambda_r}{4} \int_0^{t\wedge\tau_R} \norm{\nabla \bff{v}(s)}{\bb{L}^2}^2 \ds,
\end{align}
where $\Phi$ was defined in~\eqref{equ:Phi}. For the term $I_4$, by~\eqref{equ:Rvw vw} we have
\begin{align}\label{equ:I4 llb}
	I_4(t\wedge\tau_R)
	\leq
	C \int_0^{t\wedge\tau_R} \norm{\bff{\nu}}{\bb{L}^\infty(\mathscr{D};\bb{R}^d)}^2 \left(1+\norm{\bff{u}^\varepsilon(s)}{\bb{L}^\infty}^2 \right)
	\norm{\bff{v}(s)}{\bb{L}^2}^2 \ds
	+
	\frac{\lambda_r}{8} \int_0^{t\wedge\tau_R} \norm{\nabla \bff{v}(s)}{\bb{L}^2}^2 \ds.
\end{align}
The constant $C$ is independent of $\varepsilon$. We now apply the estimates \eqref{equ:I1 llb}, \eqref{equ:I2 llb}, \eqref{equ:I5 llb}, \eqref{equ:I6 llb}, \eqref{equ:I3 llb}, \eqref{equ:I4 llb}, and the Gronwall inequality on~\eqref{equ:dv llb}. Noting~\eqref{equ:tau R llb} and~\eqref{pro:E sup nab un L2}, we have
\begin{align*}
	&\sup_{r\in [0,t\wedge\tau_R]} \norm{\bff{v}(r)}{\bb{L}^2}^2
	+
	\int_0^{t\wedge\tau_R} \norm{\bff{v}(s)}{\bb{H}^1}^2 \ds
	\\
	&\leq
	C\varepsilon \left(\varepsilon \int_0^{t\wedge\tau_R} \norm{\nabla \bff{H}^\varepsilon(s)}{\bb{L}^2}^2 \ds\right)
	\cdot
	\exp\left(C \int_0^{t\wedge\tau_R} 1+\norm{\bff{u}^\varepsilon(s)}{\bb{L}^\infty}^2 +\Phi\big(\bff{u}^\varepsilon(s)\big) \ds\right)
	\\
	&\leq
	C \varepsilon e^{C(t+R^2)} \left(\varepsilon \int_0^{t\wedge\tau_R} \norm{\nabla \bff{H}^\varepsilon(s)}{\bb{L}^2}^2 \ds\right).
\end{align*}
Taking expectation and applying Proposition~\ref{pro:E sup nab un L2}, we obtain
\begin{align*}
	\bb{E}\left[\sup_{r \in [0,t\wedge\tau_R]} \norm{\bff{v}(r)}{\bb{L}^2}^2\right] 
	+
	\bb{E} \left[\int_0^{t\wedge\tau_R} \norm{\bff{v}(s)}{\bb{H}^1}^2 \ds\right]
	\leq
	C \varepsilon e^{C(t+R^2)},
\end{align*}
where $C$ is independent of $\varepsilon$ by Remark~\ref{rem:small lambda e}. Therefore, for any $R>0$,
\begin{align*}
	\lim_{\varepsilon\to 0^+} \bb{E}\left[\sup_{r\in [0,t\wedge\tau_R]} \norm{\bff{v}(r)}{\bb{L}^2}^2\right] 
	+
	\bb{E} \left[\int_0^{t\wedge\tau_R} \norm{\bff{v}(s)}{\bb{H}^1}^2 \ds\right]
	=0.
\end{align*}
Taking $R\uparrow\infty$ and applying the monotone convergence theorem, we obtain the required result.
\end{proof}

\section*{Acknowledgements}
The authors acknowledge financial support through the Australian Research Council's Discovery Projects funding scheme (projects DP240100781 and DP220101811).
Agus L. Soenjaya is supported by the Australian Government Research Training Program (RTP) Scholarship awarded at the University of New South Wales, Sydney.

We thank the referees for their valuable comments and suggestions, which led to significant improvements in the manuscript, and in particular to strengthening of Theorem \ref{the:invariant}.

%
%


\newcommand{\noopsort}[1]{}\def\cprime{$'$}
\def\soft#1{\leavevmode\setbox0=\hbox{h}\dimen7=\ht0\advance \dimen7
	by-1ex\relax\if t#1\relax\rlap{\raise.6\dimen7
		\hbox{\kern.3ex\char'47}}#1\relax\else\if T#1\relax
	\rlap{\raise.5\dimen7\hbox{\kern1.3ex\char'47}}#1\relax \else\if
	d#1\relax\rlap{\raise.5\dimen7\hbox{\kern.9ex \char'47}}#1\relax\else\if
	D#1\relax\rlap{\raise.5\dimen7 \hbox{\kern1.4ex\char'47}}#1\relax\else\if
	l#1\relax \rlap{\raise.5\dimen7\hbox{\kern.4ex\char'47}}#1\relax \else\if
	L#1\relax\rlap{\raise.5\dimen7\hbox{\kern.7ex
			\char'47}}#1\relax\else\message{accent \string\soft \space #1 not
		defined!}#1\relax\fi\fi\fi\fi\fi\fi}

\end{document}